\documentclass[12pt,reqno]{amsart}

\input{preamble.tex}

\title[Convolutionally observed diffusion processes]{Parametric estimation for convolutionally observed diffusion processes}
\author[S H Nakakita]{Shogo H Nakakita$^{1}$}
\author[M Uchida]{Masayuki Uchida$^{1,2}$}
\address{$^{1}$Graduate School of Engineering Science, Osaka University}
\address{$^{2}$Center for Mathematical Modeling and Data Science, Osaka University}

\begin{document}
\maketitle

\begin{abstract}
    We propose a new statistical observation scheme of diffusion processes named convolutional observation, where it is possible to deal with smoother observation than ordinary diffusion processes by considering convolution of diffusion processes and some kernel functions with respect to time parameter. We discuss the estimation and test theories for the parameter determining the smoothness of the observation, as well as the least-square-type estimation for the parameters in the diffusion coefficient and the drift one of the latent diffusion process. In addition to the theoretical discussion, we also examine the performance of the estimation and the test with computational simulation, and show an example of real data analysis for one EEG data whose observation can be regarded as smoother one than ordinary diffusion processes with statistical significance.
\end{abstract}

\section{Introduction}

We consider a $d$-dimensional diffusion process defined by the following stochastic differential equation,
\begin{align*}
    \mathrm{d}X_{t}=b\left(X_{t},\beta\right)\mathrm{d}t+a\left(X_{t},\alpha\right)\mathrm{d}w_{t},\ X_{-\lambda}=x_{-\lambda},
\end{align*}
where $\lambda> 0$, $\left\{w_{t}\right\}_{t\ge -\lambda}$ is a standard $r$-dimensional Wiener process, $x_{-\lambda}$ is an $\mathbf{R}^{d}$-valued random variable independent of  $\left\{w_{t}\right\}_{t\ge -\lambda}$, $\alpha\in\Theta_{1}$ and $\beta\in\Theta_{2}$ are unknown parameters, $\Theta_{1}\subset\mathbf{R}^{m_{1}}$ and $\Theta_{2}\subset\mathbf{R}^{m_{2}}$ are compact and convex parameter spaces, $a:\mathbf{R}^{d}\times\Theta_{1}\to\mathbf{R}^{d}\otimes\mathbf{R}^{r}$ and $b:\mathbf{R}^{d}\times\Theta_{2}\to\mathbf{R}^{d}$ are known functions. 
Our concern is statistical estimation for $\alpha$ and $\beta$ from observation.
$\theta_{\star}=\left(\alpha_{\star},\beta_{\star}\right)$ denotes the true value of $\theta:=\left(\alpha,\beta\right)$.

We denote the observation as the discretised process $\left\{\overline{X}_{ih_{n},n}:i=0,\ldots,n\right\}$ with discretisation step $h_{n}>0$ such that $h_{n}\to0$ and $T_{n}:=nh_{n}\to\infty$, where the convoluted process $\left\{\overline{X}_{t,n}\right\}_{t\ge0}$ is defined as
\begin{align*}
    \overline{X}_{t,n}:=\int_{t-\overline{\rho}h_{n}}^{t}V_{h_{n}}\left(t-s\right)X_{s}\mathrm{d}s=\int_{\mathbf{R}}V_{h_{n}}\left(t-s\right)X_{s}\mathrm{d}s=\left(V_{h_{n}}\ast X\right)\left(t\right),
\end{align*}
where $V_{h_{n}}$ is an $\mathbf{R}^{d}\otimes\mathbf{R}^{d}$-valued kernel function whose support is a subset of $\left[0,\overline{\rho}h_{n}\right]$, and $\overline{\rho}>0$ such that $\sup_{n}\overline{\rho}h_{n}\le\lambda$. In this paper, we specify
$V_{h_{n}}=V_{\rho,h_{n}}$ which is a parametric kernel function whose support is a subset of $\left[0,\overline{\rho}h_{n}\right]$ defined as
\begin{align*}
    V_{\rho,h_{n}}^{\left(i,j\right)}\left(t\right):=\begin{cases}
    \left(\rho^{\left(i\right)}h_{n}\right)^{-1}\mathbf{1}_{\left[0,\rho^{\left(i\right)}h_{n}\right]}\left(t\right) & \text{if }i=j\text{ and }\rho^{\left(i\right)}>0,\\
    \delta\left(t\right) & \text{if }i=j\text{ and }\rho^{\left(i\right)}=0,\\
    0 &\text{if }i\neq j,
    \end{cases}
\end{align*}
$\delta\left(t\right)$ is the Dirac-delta function, $\rho=\left[\rho^{\left(1\right)},\ldots,\rho^{\left(d\right)}\right]^{T}\in\Theta_{\rho}:=\left[0,\overline{\rho}\right]^{d}$ is the smoothing parameter determining the smoothness of observation. That is to say, the observed process is defined as follows:
\begin{align*}
    \overline{X}_{ih_{n},n}^{\left(\ell\right)}=\begin{cases}
    \left(\rho^{\left(\ell\right)}h_{n}\right)^{-1}\int_{\left(i-\rho^{\left(\ell\right)}\right)h_{n}}^{ih_{n}}X_{s}^{\left(\ell\right)}\mathrm{d}s&\text{ if }\rho^{\left(\ell\right)}>0,\\
    X_{ih_{n}}^{\left(\ell\right)}&\text{ if }\rho^{\left(\ell\right)}=0,
    \end{cases}
\end{align*}
for all $\ell=1,\ldots,d$.
Let us consider both the problems that (i) $\rho$ is a known parameter, and 
(ii) $\rho$ is an unknown one 
and this is estimated by observation $\left\{\overline{X}_{ih_{n},n}\right\}$, and the parameter space is denoted as $\Xi:=\Theta_{\rho}\times\Theta$.

When assuming $\rho$ as a known parameter, we can find researches for parametric estimation for $\alpha$ and/or $\beta$ based on observation schemes which can be represented as special cases for some specific $\rho$. If $\rho=\mathbf{0}$, our scheme is simply equivalent to parametric inference based on discretely observed diffusion processes $\left\{X_{ih_{n}}:i=0,\ldots,n\right\}$ studied in \citet{Florens-Zmirou-1989, Yoshida-1992, Bibby-Sorensen-1995, Kessler-1997, Kessler-Sorensen-1999, Yoshida-2011, Uchida-Yoshida-2012, Uchida-Yoshida-2014} and references therein. If $\rho=\left[1,\ldots,1\right]^{T}$, we can regard the problem as parametric estimation for integrated diffusion processes discussed in \citet{Gloter-2000, Ditlevsen-Sorensen-2004, Gloter-2006, Gloter-Gobet-2008, Sorensen-2011}. Even for the case $\rho=\left[0,\ldots,0,1,\ldots,1\right]^{T}$ where some axes correspond to direct observation and the others do to integrated observation, we give consistent estimators for $\alpha$ and $\beta$ by considering the scheme of convolutionally observed diffusion processes and this is one of the contributions of our study.

What is more, our contribution is to consider the scheme where $\rho$ is unknown and succeed in representation of the microstructure noise which makes the observation smoother than the latent diffusion process itself.
As \citet{Zhang-et-al-2005} studies, the existence of microstructure noise in financial data affects realised volatilities to increase as the subsampling frequency gets higher \citep[for instance, see Figure 7.1 in][]{Ait-Sahalia-Jacod-2014}. However, realised volatilities of some biological data such as EEG decrease as subsampling frequency increases: for instance, some time series data for the 2nd participant in the dataset named Two class motor imagery (002-2014) of \citet{BNCI-2014} show clear tendency of decreasing realised volatilities as subsampling frequency increases. 
Figure \ref{fig:BNCI:paths} shows
the path of the 2nd axis of the data S02E.mat \citet{BNCI-2014} for all 222 seconds (the observation frequency is 512Hz, 
and hence the entire data size is 113664) and that for the first one second;
it seems to perturb like a diffusion process.
\begin{figure}[ht]
    \centering
    \includegraphics[width=.45\textwidth]{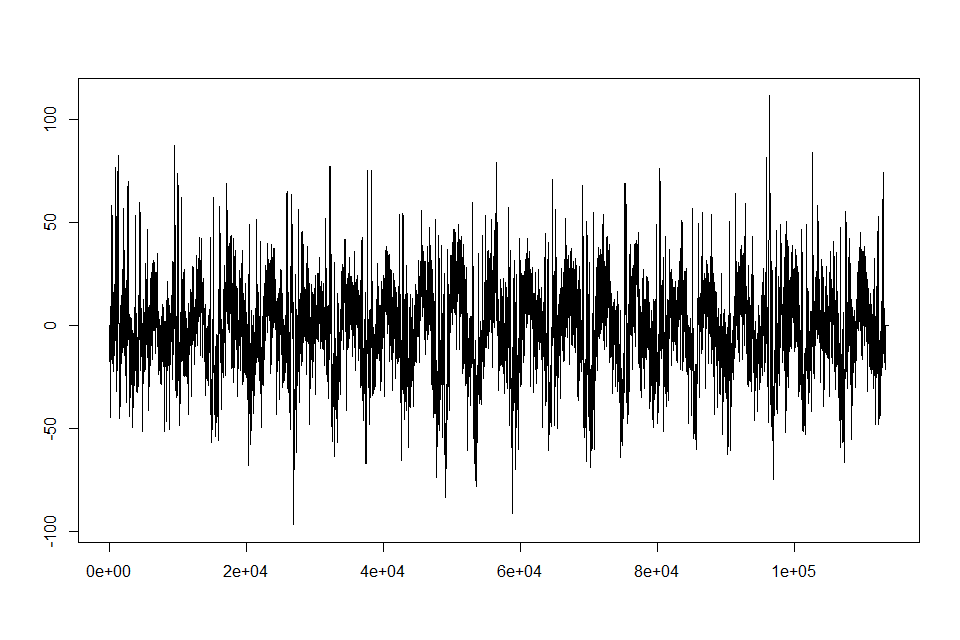}
    \includegraphics[width=.45\textwidth]{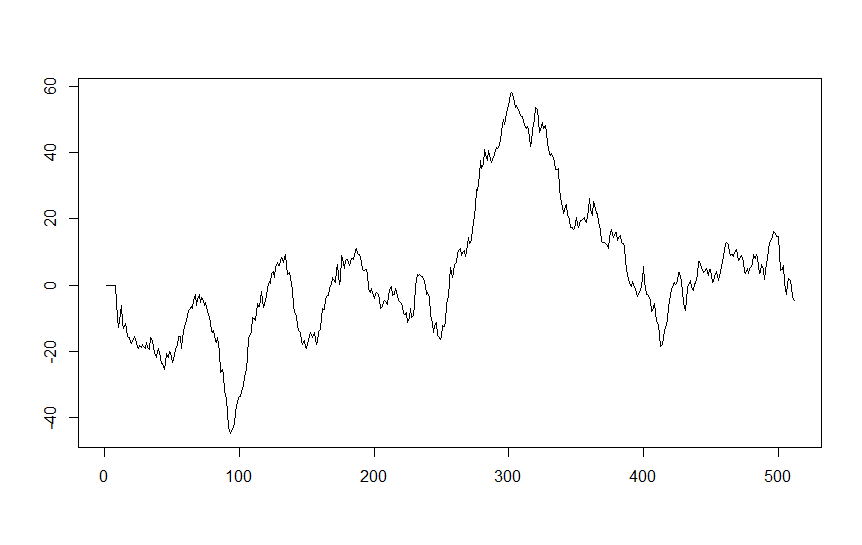}
    \caption{The path of the second column of S02E.mat of \citet{BNCI-2014} 
    for all 222 seconds (left) and the first one second (right).}
    \label{fig:BNCI:paths}
\end{figure}
We define realised volatilities with subsampling as for a one dimensional observation $\left\{Y_{i}\right\}_{i=0,\ldots,n}$,
\begin{align*}
    RV\left(k\right)=\sum_{1\le i\le \left[n/k\right]}\left(Y_{ik}-Y_{\left(i-1\right)k}\right)^{2},
\end{align*}
where $k=1,\ldots,100$ is the subsampling frequency parameter, and provide a plot of the realised volatilities the 2nd axis of the data S02E.mat in Figure \ref{fig:BNCI:RV}: 
\begin{figure}[ht]
    \centering
    \includegraphics[width=.8\textwidth]{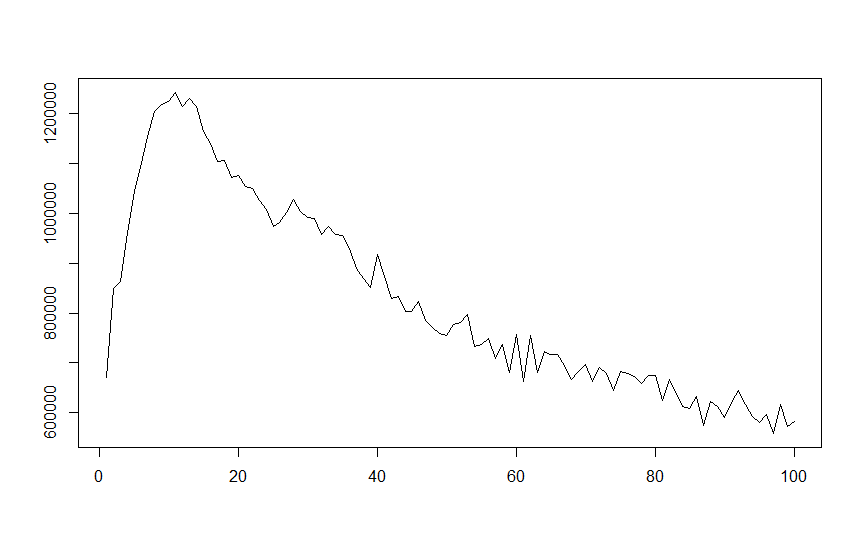}
    \caption{Realised volatilities with subsampling of the 2nd axis of data S02E.mat in Two class motor imagery (002-2014) \citep{BNCI-2014}.}
    \label{fig:BNCI:RV}
\end{figure}
the altitudes of the graph represented in the $y$-axis correspond to the values of the realised volatilities $RV\left(k\right)$ with subsampling at every $k$ observation represented in the $x$-axis. It is observable that the increasing subsampling frequency results in decreasing realised volatilities, which cannot be explained by the existent major microstructure noises \citep[e.g., see][]{Jacod-et-al-2009, Jacod-et-al-2010, Bibinger-et-al-2014, Koike-2016, Ogihara-2018}. To explain this phenomenon, we consider the smoother process than the latent one though ordinarily microstructure noises make the observation rougher than the latent process, because quadratic variation of a sufficiently smooth function is zero. One way to deal with smoother observation than the latent state is convolutional observation. As a concrete example, we show a convolutionally observed diffusion process and its characteristics in realised volatilities: let us consider the following 1-dimensional stochastic differential equation defining an Ornstein-Uhlenbeck (OU) process:
\begin{align*}
    \mathrm{d}X_{t} = -20X_{t}\mathrm{d}t+10\mathrm{d}w_{t}, X_{-\lambda}=0,
\end{align*}
where $\lambda=10^{-2/5}$. We simulate the stochastic differential equation by Euler-Maruyama method \citep[see][]{Iacus-2008} with parameters $n=10^{7}$, $h_{n}=10^{-5}$, and $T_{n}=10^{2}$ and its convolution approximated by summation with the smoothing parameter $\rho=10$ (for details, see Section 5). 
Figure \ref{fig:intro:latentAndConvoluted} shows the latent diffusion process and the convoluted observation on $\left[0,1\right]$, and we can see that the observation is indeed smoothed compared to the latent state.
\begin{figure}[ht]
    \centering
    \includegraphics[width=.45\textwidth]{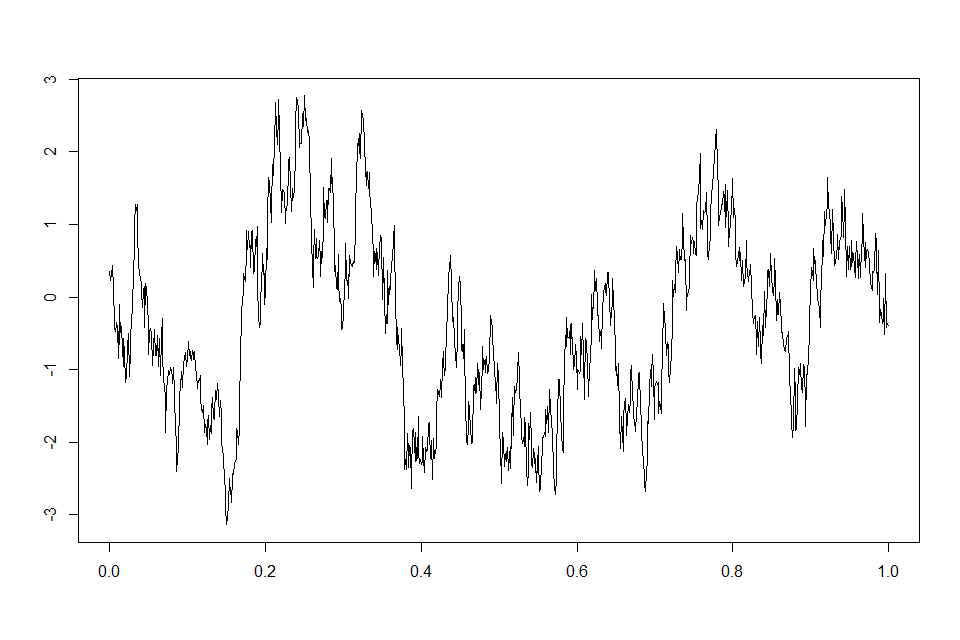}
    \includegraphics[width=.45\textwidth]{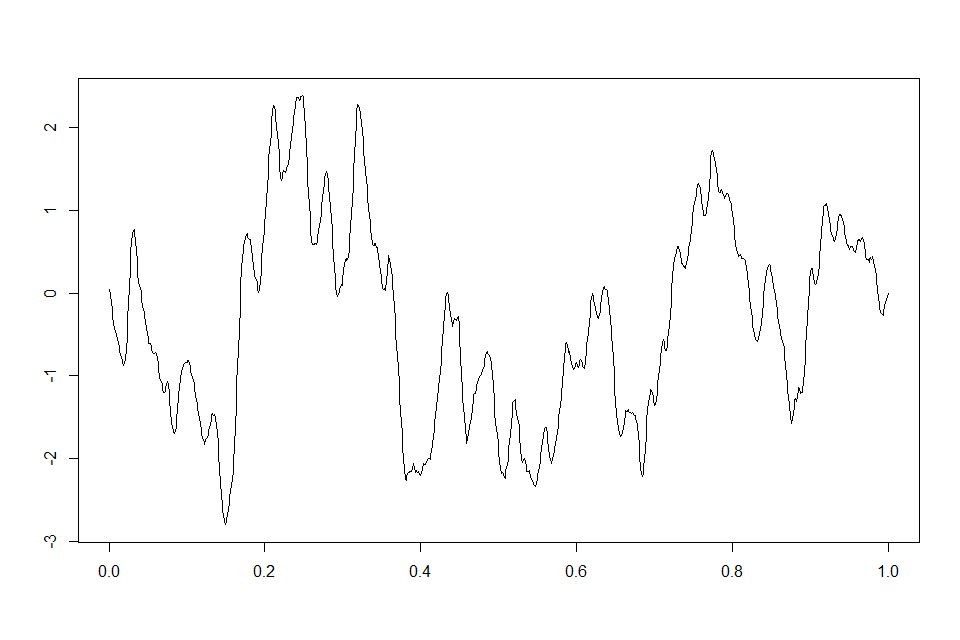}
    \caption{The left figure is the plot of the latent diffusion process, and the right one is that of the convolutionally observed process on $\left[0,1\right]$ respectively.}
    \label{fig:intro:latentAndConvoluted}
\end{figure}
In Figure \ref{fig:intro:rv}, we also give the plot of realised volatilities of the convolutionally observed process with subsampling as Figure \ref{fig:BNCI:RV}.
\begin{figure}[ht]
    \centering
    \includegraphics[width=.8\textwidth]{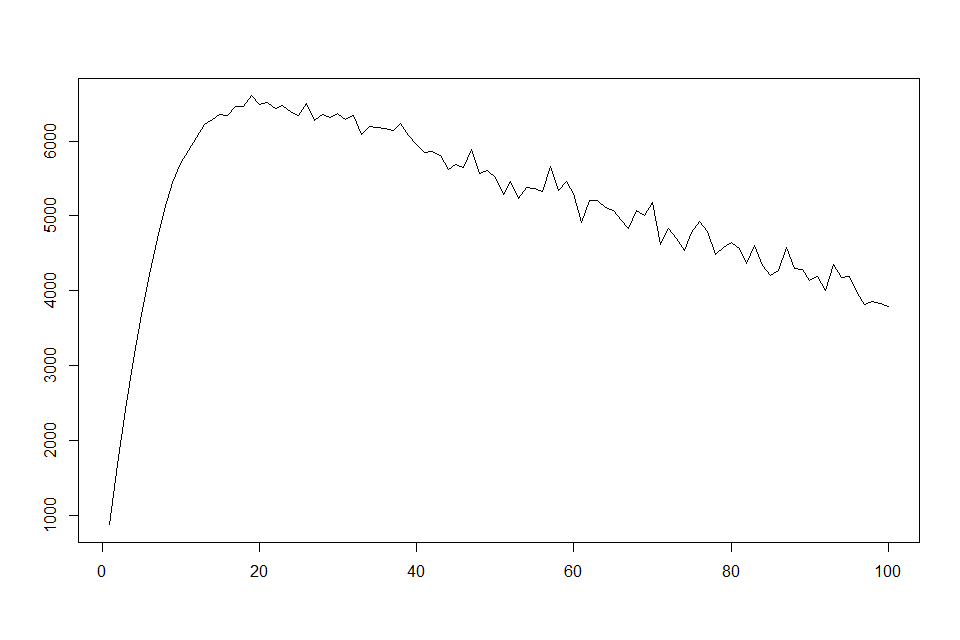}
    \caption{The realised volatilities of the convolutionally observed diffusion process with subsampling.}
    \label{fig:intro:rv}
\end{figure}
It is seen that the convolutional observation of a diffusion process also has the characteristics of decreasing realised volatilities as subsampling frequency increases, which can bee seen in some biological data such as \citet{BNCI-2014}. Of course, graphically comparing characteristics of simulation and real data is insufficient to verify the convolutional observation with smoothing parameter $\rho>0$ in 1-dimensional case; therefore, we propose statistical estimation method for unknown $\rho$ and hypothesis test with the null hypothesis $H_{0}:\rho=\mathbf{0}$ and the alternative one $H_{1}:\rho\neq\mathbf{0}$ from convolutional observation in Section 3. Moreover, in Section 6, we examine the real EEG data plotted in Figure \ref{fig:BNCI:RV} by the statistical hypothesis testing we propose, and see it is more appropriate to consider the data as a convolutional observation of a latent diffusion process with $\rho\neq\mathbf{0}$ rather than direct observation of the latent process, which indicates the validity to deal with the problem 
of the convolutional observation scheme with unknown $\rho$.

The paper is composed of the following sections: Section 2 gives the notations and assumptions used in this paper; Section 3 discusses the estimation and test for smoothing parameter $\rho$, and the discussion provides us with the tools to examine whether we should consider the convolutional observation scheme; Section 4 proposes the quasi-likelihood functions for the parameter of diffusion processes $\alpha$ and $\beta$, and corresponding estimators with consistency; Section 5 is for the computational simulation to examine the theoretical results in the previous sections; 
and Section 6 shows an application of the methods we propose in real data analysis.

\section{Notations and assumptions}
\subsection{Notations}
First of all, we set $A\left(x,\alpha\right):=a\left(x,\alpha\right)^{\otimes 2}$, $a\left(x\right):=a\left(x,\alpha_{\star}\right)$, $A\left(x\right):=A\left(x,\alpha_{\star}\right)$ and $b\left(x\right):=b\left(x,\beta_{\star}\right)$. We also give the notation for a matrix-valued function $\mathbb{G}\left(x,\alpha|\rho\right)$ such that
$\mathbb{G}^{\left(i,j\right)}\left(x,\alpha|\rho\right):=A^{\left(i,j\right)}\left(x,\alpha\right)f_{\mathbb{G}}\left(\rho^{\left(i\right)},\rho^{\left(j\right)}\right)$, where
\begin{footnotesize}
\begin{align*}
    &f_{\mathbb{G}}\left(\rho^{\left(i\right)},\rho^{\left(j\right)}\right)\\
    &:=\begin{cases}
    1& \text{ if }\rho^{\left(i\right)}=\rho^{\left(j\right)}=0,\\
    1-\frac{\rho^{\left(j\right)}}{2}& \text{ if }\rho^{\left(i\right)}=0,\rho^{\left(j\right)}\in\left(0,1\right],\\
    \frac{1}{2\rho^{\left(j\right)}}& \text{ if }\rho^{\left(i\right)}=0,\rho^{\left(j\right)}\in\left(1,\overline{\rho}\right],\\
    1-\frac{\rho^{\left(i\right)}}{2}& \text{ if }\rho^{\left(i\right)}\in\left(0,1\right],\rho^{\left(j\right)}=0,\\
    \frac{1}{2\rho^{\left(i\right)}}& \text{ if }\rho^{\left(i\right)}\in\left(1,\overline{\rho}\right],\rho^{\left(j\right)}=0,\\
    \frac{-3\left(\rho^{\left(i\right)}\right)^{2}\rho^{\left(j\right)}+3\rho^{\left(i\right)}\left(\rho^{\left(j\right)}\right)^{2}+6\rho^{\left(i\right)}\rho^{\left(j\right)}-2\left(\rho^{\left(j\right)}\right)^{3}}{6\rho^{\left(i\right)}\rho^{\left(j\right)}}& \text{ if }\rho^{\left(i\right)},\rho^{\left(j\right)}\in\left(0,1\right],\rho^{\left(i\right)}>\rho^{\left(j\right)},\\
    \frac{3\left(\rho^{\left(i\right)}\right)^{2}\rho^{\left(j\right)}-3\rho^{\left(i\right)}\left(\rho^{\left(j\right)}\right)^{2}+6\rho^{\left(i\right)}\rho^{\left(j\right)}-2\left(\rho^{\left(i\right)}\right)^{3}}{6\rho^{\left(i\right)}\rho^{\left(j\right)}}& \text{ if }\rho^{\left(i\right)},\rho^{\left(j\right)}\in\left(0,1\right],\rho^{\left(i\right)}\le\rho^{\left(j\right)},\\
    \frac{3\left(\rho^{\left(j\right)}\right)^{2}+3\rho^{\left(j\right)}-\left(\rho^{\left(j\right)}\right)^{3}}{6\rho^{\left(i\right)}\rho^{\left(j\right)}}&\text{ if }\rho^{\left(i\right)}\in\left(1,\overline{\rho}\right],\rho^{\left(j\right)}\in\left(0,1\right],\rho^{\left(i\right)}>\rho^{\left(j\right)}+1,\\
    \frac{6\rho^{\left(j\right)}-1}{6\rho^{\left(i\right)}\rho^{\left(j\right)}}&\text{ if }\rho^{\left(i\right)},\rho^{\left(j\right)}\in\left(1,\overline{\rho}\right],\rho^{\left(i\right)}>\rho^{\left(j\right)}+1,\\
    \frac{\left(\rho^{\left(i\right)}-\rho^{\left(j\right)}\right)^3 - 3\left(\rho^{\left(i\right)}\right)^2 + 6\rho^{\left(i\right)}\rho^{\left(j\right)} + 3 \rho^{\left(i\right)}- 1-\left(\rho^{\left(j\right)}\right)^{3}}{6\rho^{\left(i\right)}\rho^{\left(j\right)}}&\text{ if }\rho^{\left(i\right)}\in\left(1,\overline{\rho}\right],\rho^{\left(j\right)}\in\left(0,1\right],\rho^{\left(i\right)}\le \rho^{\left(j\right)}+1,\\
    \frac{\left(\rho^{\left(i\right)}-\rho^{\left(j\right)}\right)^3 - 3\left(\rho^{\left(i\right)}\right)^2 + 6\rho^{\left(i\right)}\rho^{\left(j\right)} + 3 \rho^{\left(i\right)}-3\left(\rho^{\left(j\right)}\right)^{2}+3\rho^{\left(j\right)}- 2}{6\rho^{\left(i\right)}\rho^{\left(j\right)}}&\text{ if }\rho^{\left(i\right)},\rho^{\left(j\right)}\in\left(1,\overline{\rho}\right],\rho^{\left(j\right)}<\rho^{\left(i\right)}\le \rho^{\left(j\right)}+1,\\
    \frac{-\left(\rho^{\left(i\right)}-\rho^{\left(j\right)}\right)^{3}+6\rho^{\left(i\right)}\rho^{\left(j\right)}-\left(\rho^{\left(i\right)}\right)^{3}-3\left(\rho^{\left(j\right)}\right)^{2}+3\rho^{\left(j\right)}-1}{6\rho^{\left(i\right)}\rho^{\left(j\right)}}&\text{ if }\rho^{\left(i\right)}\in\left(0,1\right],\rho^{\left(j\right)}\in\left(1,\overline{\rho}\right],\rho^{\left(j\right)}\le \rho^{\left(i\right)}+1,\\
    \frac{-\left(\rho^{\left(i\right)}-\rho^{\left(j\right)}\right)^{3}-3\left(\rho^{\left(i\right)}\right)^{2}-3\left(\rho^{\left(j\right)}\right)^{2}+6\rho^{\left(i\right)}\rho^{\left(j\right)}+3\rho^{\left(i\right)}+3\rho^{\left(j\right)}-2}{6\rho^{\left(i\right)}\rho^{\left(j\right)}}&\text{ if }\rho^{\left(i\right)},\rho^{\left(j\right)}\in\left(1,\overline{\rho}\right],\rho^{\left(i\right)}\le\rho^{\left(j\right)}\le \rho^{\left(i\right)}+1,\\
    \frac{3\left(\rho^{\left(i\right)}\right)^{2}+3\rho^{\left(i\right)}-\left(\rho^{\left(i\right)}\right)^{3}}{6\rho^{\left(i\right)}\rho^{\left(j\right)}}&\text{ if }\rho^{\left(i\right)}\in\left(0,1\right],\rho^{\left(j\right)}\in\left(1,\overline{\rho}\right],\rho^{\left(j\right)}> \rho^{\left(i\right)}+1,\\
    \frac{6\rho^{\left(i\right)}-1}{6\rho^{\left(i\right)}\rho^{\left(j\right)}}&\text{ if }\rho^{\left(i\right)},\rho^{\left(j\right)}\in\left(1,\overline{\rho}\right],\rho^{\left(j\right)}> \rho^{\left(i\right)}+1.
    \end{cases}
\end{align*}
\end{footnotesize}

The continuity of these functions is examined in Lemma \ref{LemmaFunctionD} and Lemma \ref{LemmaFunctionG}. In addition, we also give the notation used throughout this paper.
\begin{itemize}
	\item For every matrix $A$, $A^{T}$ is the transpose of $A$, and $A^{\otimes 2}:=AA^{T}$.
	\item For every set of matrices $A$ and $B$ 
of the same size,
$A\left[B\right]:=\mathrm{tr}\left(AB^{T}\right)$. Moreover, for any $m\in\mathbf{N}$, $A\in\mathbf{R}^{m}\otimes\mathbf{R}^{m}$ and $u,v\in\mathbf{R}^{m}$, $A\left[u,v\right]:=v^{T}Au$.
	\item $v^{\left(\ell\right)}$ and $A^{\left(\ell_{1},\ell_{2}\right)}$ denote 
	the $\ell$-th element of a vector $v$ and the $\left(\ell_{1},\ell_{2}\right)$-th one of a matrix $A$, respectively.
	\item For any vector $v$ and any matrix $A$, $\left|v\right|:=\sqrt{\mathrm{tr}\left(v^{T}v\right)}$ and $\left\|A\right\|:=\sqrt{\mathrm{tr}\left(A^{T}A\right)}$.
	\item $\left(\Omega,P,\mathcal{F},\mathcal{F}_{t}\right)$ denotes the stochastic basis, where  $\mathcal{F}_{t}:=\sigma\left(x_{-\lambda},w_{s}:s\le t\right)$.
\end{itemize}

\subsection{Assumptions}
With respect to $X_{t}$, we assume the following conditions.
\begin{itemize}
\item[{[A1]}]
\begin{itemize}
\item[(i)] For a constant $C$, for all $x_{1},x_{2}\in\mathbf{R}^{d}$,
\begin{align*}
	\sup_{\alpha\in\Theta_{1}}\left\|a\left(x_{1},\alpha\right)-a\left(x_{2},\alpha\right)\right\|+
	\sup_{\beta\in\Theta_{2}}\left|b\left(x_{1},\beta\right)-b\left(x_{2},\beta\right)\right|\le C\left|x_{1}-x_{2}\right|.
\end{align*}
\item[(ii)] For all $p\ge0$, $\sup_{t\ge-\lambda}\mathbf{E}_{\theta_{\star}}\left[\left|X_{t}\right|^{p}\right]<\infty$.
\item[(iii)] There exists a unique invariant measure $\nu_{0}$ on $\left(\mathbf{R}^{d},\mathcal{B}\left(\mathbf{R}^{d}\right)\right)$ and for all $p\ge1$ and $f\in L^{p}\left(\nu_{0}\right)$ with polynomial growth,
\begin{align*}
\frac{1}{T}\int_{-\lambda}^{T}f\left(X_{t}\right)\mathrm{d}t\to^{P}\int_{\mathbf{R}^{d}}f\left(x\right)\nu_{0}\left(\mathrm{d}x\right).
\end{align*}
\end{itemize}
\item[{[A2]}] There exists $C>0$ such that $a:\mathbf{R}^{d}\times \Theta_{1}\to \mathbf{R}^{d}\otimes \mathbf{R}^{r}$ and $b:\mathbf{R}^{d}\times \Theta_{2}\to \mathbf{R}^{d}$ have continuous derivatives satisfying
\begin{align*}
\sup_{\alpha\in\Theta_{1}}\left|\partial_{x}^{j}\partial_{\alpha}^{i}a\left(x,\alpha\right)\right|&\le
C\left(1+\left|x\right|\right)^{C},\ 0\le i\le 2,\ 0\le j\le 2,\\
\sup_{\beta\in\Theta_{2}}\left|\partial_{x}^{j}\partial_{\beta}^{i}b\left(x,\beta\right)\right|&\le C\left(1+\left|x\right|\right)^{C},\ 0\le i\le 2,\ 0\le j\le 2.
\end{align*}
\end{itemize}

With the invariant measure $\nu_{0}$, we define
\begin{align*}
    \mathbb{V}_{1} \left(\alpha|\xi_{\star}\right)
	&:=-\int_{\mathbf{R}^{d}}
	\left\|\mathbb{G}\left(x,\alpha|\rho_{\star}\right)-
    \mathbb{G}\left(x,\alpha_{\star}|\rho_{\star}\right)\right\|^{2}\nu_{0} \left(\mathrm{d}x\right),\\
	\mathbb{V}_{2}\left(\beta|\xi_{\star}\right)
    &:=-\int_{\mathbf{R}^{d}}\left|b\left(x,\beta\right)-b\left(x,\beta_{\star}\right)\right|^{2}\nu_{0}\left(\mathrm{d}x\right).
\end{align*}
For these functions, let us assume the following identifiability conditions hold.

\begin{itemize}
\item[{[A3]}] 
There exist $\chi_1\left(\alpha_{\star}\right)>0$ and $\chi_1'\left(\beta_{\star}\right)>0$
such that for all $\alpha\in\Theta_{1}$ and $\beta\in\Theta_{2}$,
$\mathbb{V}_{1}\left(\alpha|\xi_{\star}\right)\le -\chi_1\left(\alpha_{\star}\right)\left|\alpha-\alpha_{\star}\right|^2$ 
and $\mathbb{V}_{2}\left(\beta|\xi_{\star}\right)\le -\chi_1'\left(\beta_{\star}\right)\left|\beta-\beta_{\star}\right|^2$.
\end{itemize}

\section{Estimation and test of the smoothing parameter}
In this section, we discuss the case where the smoothing parameter $\rho$ of the kernel function $V_{\rho,h_{n}}$ is unknown. 
The estimation is significant for estimation of $\alpha$ and $\beta$ since we utilise the estimate for $\rho$ in quasi-likelihood functions of $\alpha$ and $\beta$. The test problem for hypotheses $H_{0}:\rho=\mathbf{0}$ and $H_{1}:\rho\neq\mathbf{0}$ is also important to examine whether our framework of convolutional observation is meaningful.
\subsection{Estimation of the smoothing parameter}
For simplicity of notation, let us consider the case $\overline{\rho}>2$; otherwise the discussion is quite parallel. We should note that for all $i=1,\ldots,d$,
\begin{align*}
    \mathbb{G}^{\left(i,i\right)}\left(x|\rho\right)=\begin{cases}
    A^{\left(i,i\right)}\left(x\right)&\text{ if }\rho^{\left(i\right)}=0,\\
    A^{\left(i,i\right)}\left(x\right)\left(1-\frac{\rho^{\left(i\right)}}{3}\right)&\text{ if }\rho^{\left(i\right)}\in\left(0,1\right],\\
    A^{\left(i,i\right)}\left(x\right)\left(\frac{1}{\rho^{\left(i\right)}}-\frac{1}{3\left(\rho^{\left(i\right)}\right)^{2}}\right)&\text{ if }\rho^{\left(i\right)}\in\left(1,\overline{\rho}\right].
    \end{cases}
\end{align*}
Let us consider the estimation of $\rho^{\left(i\right)}$ with using the next statistics: the full quadratic variation
\begin{align*}
    \frac{1}{nh_{n}}\sum_{k=1}^{n}\left(\overline{X}_{kh_{n},n}^{\left(i\right)}-\overline{X}_{\left(k-1\right)h_{n},n}^{\left(i\right)}\right)^{2}\to^{P} \begin{cases}
    \nu_{0}\left(A^{\left(i,i\right)}\left(\cdot\right)\right)&\text{ if }\rho_{\star}^{\left(i\right)}=0,\\
    \nu_{0}\left(A^{\left(i,i\right)}\left(\cdot\right)\right)\left(1-\frac{\rho_{\star}^{\left(i\right)}}{3}\right)&\text{ if }\rho_{\star}^{\left(i\right)}\in\left(0,1\right],\\
    \nu_{0}\left(A^{\left(i,i\right)}\left(\cdot\right)\right)\left(\frac{1}{\rho_{\star}^{\left(i\right)}}-\frac{1}{3\left(\rho_{\star}^{\left(i\right)}\right)^{2}}\right)&\text{ if }\rho_{\star}^{\left(i\right)}\in\left(1,\overline{\rho}\right],
    \end{cases}
\end{align*}
because of Proposition \ref{propEmpQ} in Appendix A,
and the reduced quadratic variation defined as $\frac{1}{nh_{n}}\sum_{2\le 2k\le n}\left(\overline{X}_{2kh_{n},n}^{\left(i\right)}-\overline{X}_{\left(2k-2\right)h_{n},n}^{\left(i\right)}\right)^{2}$ converges in probability as follows.
\begin{lemma}\label{LemmaReducedQ}
Under [A1], we have the convergence in probability such that
\begin{align*}
    &\frac{1}{nh_{n}}\sum_{2\le 2k\le n}\left(\overline{X}_{2k,n}^{\left(i\right)}-\overline{X}_{2k-2,n}^{\left(i\right)}\right)^{2}\\
    &\to^{P}\begin{cases}
    \nu_{0}\left(A^{\left(i,i\right)}\left(\cdot\right)\right)&\text{ if }\rho_{\star}^{\left(i\right)}=0,\\
    \nu_{0}\left(A^{\left(i,i\right)}\left(\cdot\right)\right)\left(1-\frac{\rho_{\star}^{\left(i\right)}}{6}\right)&\text{ if }\rho_{\star}^{\left(i\right)}\in\left(0,2\right],\\
    \nu_{0}\left(A^{\left(i,i\right)}\left(\cdot\right)\right)\left(\frac{2}{\rho_{\star}^{\left(i\right)}}-\frac{4}{3\left(\rho_{\star}^{\left(i\right)}\right)^{2}}\right)&\text{ if }\rho_{\star}^{\left(i\right)}\in\left(2,\overline{\rho}\right].
    \end{cases}
\end{align*}
\end{lemma}
Then we define the ratio of those statistics such that
\begin{align*}
    R_{n}^{\left(i\right)}&:=\left(\frac{1}{nh_{n}}\sum_{k=1}^{n}\left(\overline{X}_{kh_{n},n}^{\left(i\right)}-\overline{X}_{\left(k-1\right)h_{n},n}^{\left(i\right)}\right)^{2}\right)\left(\frac{1}{nh_{n}}\sum_{2\le 2k\le n}\left(\overline{X}_{2kh_{n},n}^{\left(i\right)}-\overline{X}_{\left(2k-2\right)h_{n},n}^{\left(i\right)}\right)^{2}\right)^{-1}\\
    &\to^{P}\begin{cases}
    1 & \text{ if }\rho_{\star}^{\left(i\right)}=0,\\
    \left(1-\frac{\rho_{\star}^{\left(i\right)}}{3}\right)\left(1-\frac{\rho_{\star}^{\left(i\right)}}{6}\right)^{-1} & \text{ if }\rho_{\star}^{\left(i\right)}\in\left(0,1\right],\\
    \left(\frac{1}{\rho_{\star}^{\left(i\right)}}-\frac{1}{3\left(\rho_{\star}^{\left(i\right)}\right)^{2}}\right)\left(1-\frac{\rho_{\star}^{\left(i\right)}}{6}\right)^{-1} &\text{ if }\rho_{\star}^{\left(i\right)}\in\left(1,2\right],\\
    \left(\frac{1}{\rho_{\star}^{\left(i\right)}}-\frac{1}{3\left(\rho_{\star}^{\left(i\right)}\right)^{2}}\right)\left(\frac{2}{\rho_{\star}^{\left(i\right)}}-\frac{4}{3\left(\rho_{\star}^{\left(i\right)}\right)^{2}}\right)^{-1}&\text{ if }\rho_{\star}^{\left(i\right)}\in\left(2,\overline{\rho}\right]
    \end{cases}\\
    &=\begin{cases}
    1 & \text{ if }\rho^{\left(i\right)}=0,\\
    \left(6-2\rho_{\star}^{\left(i\right)}\right)\left(6-\rho_{\star}^{\left(i\right)}\right)^{-1} & \text{ if }\rho_{\star}^{\left(i\right)}\in\left(0,1\right],\\
    \left(6\rho_{\star}^{\left(i\right)}-2\right)\left(6\left(\rho_{\star}^{\left(i\right)}\right)^{2}-\left(\rho_{\star}^{\left(i\right)}\right)^{3}\right)^{-1} &\text{ if }\rho_{\star}^{\left(i\right)}\in\left(1,2\right],\\
    \left(3\rho_{\star}^{\left(i\right)}-1\right)\left(6\rho_{\star}^{\left(i\right)}-4\right)^{-1}&\text{ if }\rho_{\star}^{\left(i\right)}\in\left(2,\overline{\rho}\right],
    \end{cases}\\
    &=: R\left(\rho_{\star}^{\left(i\right)}\right),
\end{align*}
where $R$ has the next property.
\begin{lemma}\label{LemmaFunctionR}
    $R$ is a $\left[\left(3\overline{\rho}-1\right)\left(6\overline{\rho}-4\right)^{-1},1\right]$-valued monotonically decreasing continuous function, and has a continuous inverse $R^{-1}:\left[\left(3\overline{\rho}-1\right)\left(6\overline{\rho}-4\right)^{-1},1\right]\to\left[0,\overline{\rho}\right]$.
\end{lemma}

We define $\hat{\rho}_{n}$ such that
\begin{align*}
    \hat{\rho}_{n}^{\left(i\right)}&:=\begin{cases}
    0 & \text{ if }R_{n}^{\left(i\right)}>1,\\
    R^{-1}\left(R_{n}^{\left(i\right)}\right) & \text{ if }R_{n}^{\left(i\right)}\in \left[\left(3\overline{\rho}-1\right)\left(6\overline{\rho}-4\right)^{-1},1\right],\\
    \overline{\rho} & \text{ if } R_{n}^{\left(i\right)}<\left(3\overline{\rho}-1\right)\left(6\overline{\rho}-4\right)^{-1},
    \end{cases}
\end{align*}
and then continuous mapping theorem for convergence in probability verifies the next result.

\begin{theorem}\label{thmRhoEstimate}
Under [A1], $\hat{\rho}_{n}$ has consistency, i.e., $\hat{\rho}_{n}\to^{P}\rho_{\star}$.
\end{theorem}

\begin{remark}
    We can compute $y=R^{-1}\left(x\right),\ x\in\left[\left(3\overline{\rho}-1\right)\left(6\overline{\rho}-4\right)^{-1},1\right]$ by solving the following equations:
    \begin{align*}
        \text{(i) }&x=\left(6-2y\right)\left(6-y\right)^{-1} && \text{ if }x\in\left(4/5,1\right],\\
        \text{(ii) }& x=\left(6-2y\right)\left(6y^2-y^3\right)^{-1}&&\text{ if }x\in\left(5/8,4/5\right],\\
        \text{(iii) }& x=\left(3y-1\right)\left(6y-4\right)^{-1}&&\text{ if }x\in\left[\left(3\overline{\rho}-1\right)\left(6\overline{\rho}-4\right)^{-1},5/8\right].
    \end{align*}
\end{remark}

\subsection{Test for smoothed observation}
For all $i=1,\ldots,d$, we consider the next hypothesis testing:
\begin{align*}
    H_{0}:\rho^{\left(i\right)}=0,\ H_{1}:\rho^{\left(i\right)}>0.
\end{align*}
Let us consider the following test statistic:
\begin{align*}
    \mathcal{T}_{i,n}&:=\sqrt{\frac{n}{\frac{2}{3nh_{n}^{2}}\sum_{k=1}^{n}\left(\overline{X}_{kh_{n},n}^{\left(i\right)}-\overline{X}_{\left(k-1\right)h_{n},n}^{\left(i\right)}\right)^{4}}}\\
    &\quad\times\left(\frac{1}{nh_{n}}\sum_{k=1}^{n}\left(\overline{X}_{kh_{n},n}^{\left(i\right)}-\overline{X}_{\left(k-1\right)h_{n},n}^{\left(i\right)}\right)^{2}-\frac{1}{nh_{n}}\sum_{2\le 2k\le n}\left(\overline{X}_{2kh_{n},n}^{\left(i\right)}-\overline{X}_{\left(2k-2\right)h_{n},n}^{\left(i\right)}\right)^{2}\right)\\
    &=\sqrt{\frac{3/2}{\sum_{k=1}^{n}\left(\overline{X}_{kh_{n},n}^{\left(i\right)}-\overline{X}_{\left(k-1\right)h_{n},n}^{\left(i\right)}\right)^{4}}}\\
    &\quad\times\left(\sum_{k=1}^{n}\left(\overline{X}_{k,n}^{\left(i\right)}-\overline{X}_{k-1,n}^{\left(i\right)}\right)^{2}-\sum_{2\le 2k\le n}\left(\overline{X}_{2kh_{n},n}^{\left(i\right)}-\overline{X}_{\left(2k-2\right)h_{n},n}^{\left(i\right)}\right)^{2}\right),
\end{align*}
and we abbreviate $\mathcal{T}_{i,n}$ to $\mathcal{T}_{n}$ if $d=1$.
Under $H_{0}$, we have the next result.
\begin{theorem}\label{thmTestH0}
    Under $H_{0}$ and [A1], we have the convergence in law such that
    \begin{align*}
        \mathcal{T}_{i,n}\to^{\mathcal{L}}N\left(0,1\right).
    \end{align*}
\end{theorem}
We also obtain the result to support the consistency of the test.
\begin{theorem}\label{thmTestH1}
    Under $H_{1}$ and [A1], we have the divergence in probability such that for any $c\in\mathbf{R}$,
    \begin{align*}
        &P\left(\mathcal{T}_{i,n}<c\right)\to 1.
    \end{align*}
\end{theorem}

Hence when we set the significance level $\alpha_{\mathrm{sig}}\in\left(0,1\right)$, then we have the rejection region
\begin{align*}
    &\mathcal{T}_{i,n}<\Phi^{-1}\left(\alpha_{\mathrm{sig}}\right)
\end{align*}
where 
$\Phi$ is the distribution function of the standard Gaussian distribution. 
Theorem \ref{thmTestH1} supports the consistency of the test.

This test is essential in terms of examining the validity to consider the scheme of convolutional observation: if $\rho=\mathbf{0}$, then the ordinary observation scheme can be applied, but if $\rho\neq\mathbf{0}$, then we have the motivation to consider the convolutional 
observation scheme.

\section{Least square estimation of the diffusion and drift parameters} 
Let us set the least-square quasi-likelihood functions such that
\begin{align*}
    \mathbb{H}_{1,n}\left(\alpha|\rho\right)&:=-\sum_{k=1}^{n}\left\|\frac{1}{h_{n}}\left(\overline{X}_{kh_{n},n}-\overline{X}_{\left(k-1\right)h_{n},n}\right)^{\otimes2}-\mathbb{G}\left(\overline{X}_{\left(k-1\right)h_{n},n},\alpha|\rho\right)\right\|^{2},\\
    \mathbb{H}_{2,n}\left(\beta|\rho\right)& :=-\sum_{k=\left[\max_{i}\rho^{\left(i\right)}\right]+2}^{n}\frac{1}{h_{n}}\left|\overline{X}_{kh_{n},n}-\overline{X}_{\left(k-1\right)h_{n},n}-h_{n}b\left(\overline{X}_{\left(k-2-\left[\max_{i}\rho^{\left(i\right)}\right]\right)h_{n},n},\beta\right)\right|^{2},
\end{align*}
and the least-square estimators $\hat{\alpha}_{n}$ and $\hat{\beta}_{n}$ satisfying 
\begin{align*}
    \mathbb{H}_{1,n}\left(\hat{\alpha}_{n}|\rho_{\star}\right)=\sup_{\alpha\in\Theta_{1}}\mathbb{H}_{1,n}\left(\alpha|\rho_{\star}\right),\ 
    \mathbb{H}_{2,n}\left(\hat{\beta}_{n}|\rho_{\star}\right)=
    \sup_{\beta\in\Theta_{2}}\mathbb{H}_{2,n}\left(\beta|\rho_{\star}\right).
\end{align*}
when $\rho_{\star}$ is known, and 
\begin{align*}
    \mathbb{H}_{1,n}\left(\hat{\alpha}_{n}|\hat{\rho}_{n}\right)=\sup_{\alpha\in\Theta_{1}}\mathbb{H}_{1,n}\left(\alpha|\hat{\rho}_{n}\right),\ 
    \mathbb{H}_{2,n}\left(\hat{\beta}_{n}|\hat{\rho}_{n}\right)=
    \sup_{\beta\in\Theta_{2}}\mathbb{H}_{2,n}\left(\beta|\hat{\rho}_{n}\right).
\end{align*}
when $\rho_{\star}$ is unknown.

\begin{theorem}\label{thmThetaEstimate}
    Under [A1]-[A3], $\hat{\alpha}_{n}$ and $\hat{\beta}_{n}$ are consistent, i.e., $\hat{\alpha}_{n}\to^{P}\alpha_{\star}$ and $\hat{\beta}_{n}\to^{P}\beta_{\star}$.
\end{theorem}

\section{Simulations}
In this simulation section, we only consider the case where $\rho$ is unknown and should be estimated by data with the method proposed in Section 3.

\subsection{1-dimensional simulation} We examine the following 1-dimensional stochastic differential equation whose solution is a 1-dimensional Ornstein-Uhlenbeck (OU) process:
\begin{align*}
    \mathrm{d}X_{t}=\left(\beta^{\left(1\right)}X_{t}+\beta^{\left(2\right)}\right)\mathrm{d}t+\alpha\mathrm{d}w_{t},\ X_{-\lambda}=0,
\end{align*}
$\alpha\in\Theta_{1}:=\left[0.01,10\right]$, $\beta\in\Theta_{2}:=\left[-10,-0.01\right]\times\left[-10,10\right]$, and $\lambda=10^{-7/3}$. The procedure of the simulation is as follows: in the first place we iterate an approximated OU process by Euler-Maruyama scheme \citep[for example, see][]{Iacus-2008} with simulation parameters $n_{\mathrm{sim}}=10^{5+m}$, $h_{\mathrm{sim}}=10^{-10/3-m}$, $T_{\mathrm{sim}}=10^{5/3}$ where $m\in\mathbf{N}$ is 
a parameter to determine the precision of approximation; secondly, we give the approximation of convolution by summation such that
\begin{align*}
    \overline{X}_{ih_{n},n}\approxeq \begin{cases}\frac{1}{\left[10^{m}\rho\right]}\sum_{k=0}^{10^{m}-1}X_{ih_{n}-kh_{\mathrm{sim}}}&\text{ if }\left[10^{m}\rho\right]\ge1,\\
    X_{ih_{n}}&\text{ if } \left[10^{m}\rho\right]<0,
    \end{cases}
\end{align*}
where $i=0,\ldots,n$, the sampling frequency $h_{n}=10^{-10/3}$ and $n=10^{5/3}$. In this Section 5.1, we fix the true value of $\alpha$ and $\beta$ as $\alpha_{\star}=3$ and $\beta_{\star}=\left[-2,1\right]^{T}$, and change the true value of $\rho\in\Theta_{\rho}:=\left[0,100\right]$ to see the corresponding changes of performance of estimation for $\xi$, and test for $\rho$ in comparison to estimation by an existent method called local Gaussian approximation (LGA) for parametric estimation of discretely observed diffusion processes \citep[e.g., see][]{Kessler-1997} which does not concern convolutional observation. All the numbers of iterations for different $\rho$'s are 1000.

In the first place, we see the estimation and test with small values of $\rho_{\star}$ such that $\rho_{\star}=0,0.1,0.2,\ldots,1$ to observe how the performance of statistics changes by difference in $\rho$. Table \ref{tab:sim:1dim:small:rho:estimate} summarises the results of simulation of $\hat{\rho}_{n}$ for $\rho$'s with respective empirical means and root mean square error (RMSE).
\begin{table}[ht]
    \begin{tabular}{c|c|c|c|c|c}
         $\rho$ & $0.0$ & $0.1$ & $0.2$ & $0.3$ & $0.4$ \\\hline
         mean & $0.00990$ & $0.0971$ & $0.198$ & $0.298$ & $0.398$ \\
         RMSE & $(0.0182)$ & $(0.0256)$ & $(0.0235)$ & $(0.0215)$ & $(0.0197)$\\\\
    \end{tabular}
    \begin{tabular}{c|c|c|c|c|c|c}
         $\rho$ & 0.5 & 0.6 & 0.7 & 0.8 & 0.9 & 1.0\\\hline
         mean &  $0.498$ & $0.598$ & $0.699$ & $0.799$ & $0.899$ & $0.999$ \\
         RMSE & $(0.0180)$ & $(0.0164)$ & $(0.0150)$ & $(0.0135)$ & $(0.0123)$ & $(0.0110)$
    \end{tabular}
    \caption{Estimation performance of $\rho$ with small $\rho$.}
    \label{tab:sim:1dim:small:rho:estimate}
\end{table}
We can see the proposed estimator $\hat{\rho}_{n}$ works well for small $\rho$. With respect to the performance of the test statistic $\mathcal{T}_{n}$ proposed in Section 3.2, Table \ref{tab:sim:1dim:small:rho:test} shows the empirical ratio of the number of iterations whose $\mathcal{T}_{n}$ is lower than some typical critical values where $\Phi$ indicates the distribution function of 1-dimensional standard Gaussian distribution as well as the maximum value of $\mathcal{T}_{n}$ in 1000 iterations.
\begin{table}[ht]
    \centering
    \begin{tabular}{c|c|c|c|c|c|c}
        & \multicolumn{5}{c|}{empirical ratio of $\mathcal{T}_{n}$ less than ...}& \multirow{2}{*}{max. of $\mathcal{T}_{n}$}\\
        & $\Phi^{-1}\left(0.10\right)$ & $\Phi^{-1}\left(0.05\right)$ & $\Phi^{-1}\left(0.025\right)$ & $\Phi^{-1}\left(0.01\right)$ & $\Phi^{-1}\left(0.001\right)$\\\hline
        $\rho=0.0$ & $0.101$ & $0.053$ & $0.025$ & $0.005$ & $0.000$ & $3.060$\\\hline
        $\rho=0.1$ & $0.989$ & $0.980$ & $0.966$ & $0.914$ & $0.759$ & $-0.710$\\\hline
        $\rho=0.2$ & $1.000$ & $1.000$ & $1.000$ & $1.000$ & $1.000$ & $-4.593$\\\hline
        $\rho=0.3$ & $1.000$ & $1.000$ & $1.000$ & $1.000$ & $1.000$ &  $-9.341$\\\hline
        $\rho=0.4$ & $1.000$ & $1.000$ & $1.000$ & $1.000$ & $1.000$ & $-13.985$\\\hline
        $\rho=0.5$ & $1.000$ & $1.000$ & $1.000$ & $1.000$ & $1.000$ & $-19.152$\\\hline
        $\rho=0.6$ & $1.000$ & $1.000$ & $1.000$ & $1.000$ & $1.000$ & $-24.816$\\\hline
        $\rho=0.7$ & $1.000$ & $1.000$ & $1.000$ & $1.000$ & $1.000$ & $-30.848$\\\hline
        $\rho=0.8$ & $1.000$ & $1.000$ & $1.000$ & $1.000$ & $1.000$ & $-37.557$\\\hline
        $\rho=0.9$ & $1.000$ & $1.000$ & $1.000$ & $1.000$ & $1.000$ & $-44.829$\\\hline
        $\rho=1.0$ & $1.000$ & $1.000$ & $1.000$ & $1.000$ & $1.000$ & $-52.759$\\
    \end{tabular}
    \caption{Empirical ratio of test statistic $\mathcal{T}_{n}$ less than some critical values, and the maximum value of $\mathcal{T}_{n}$ in 1000 iterations}
    \label{tab:sim:1dim:small:rho:test}
\end{table}
Even for $\rho=0.1$, the simulation result supports the theoretical discussion of the test with consistency. Because $\Phi\left(10^{-16}\right)=-8.222$, all the iterations with $\rho\ge0.3$ result in rejection of $H_{0}$ with substantially significance level $10^{-16}$.
Let us see the estimation for $\alpha$ and $\beta$ by our proposal method and that by the LGA in Table \ref{tab:sim:1dim:small:theta:estimate}. 
\begin{table}[ht]
    \centering
    \begin{tabular}{c|c|c|c|c|c|c|c}
        \multicolumn{2}{c|}{} & \multicolumn{3}{c|}{the proposed method} & \multicolumn{3}{c}{LGA}\\\hline
        \multicolumn{2}{c|}{}& $\alpha$ & $\beta^{\left(1\right)}$ & $\beta^{\left(2\right)}$ & $\alpha$ & $\beta^{\left(1\right)}$ & $\beta^{\left(2\right)}$\\\hline
        \multicolumn{2}{c|}{true value}& $3.0$ & $-2.0$ & $1.0$ & $3.0$ & $-2.0$ & $1.0$\\\hline
        \multirow{2}{*}{$\rho=0.0$} & mean & $3.004$ & $-2.091$ & $1.036$ & $2.999$ & $-2.095$ & $1.037$ \\
        & RMSE & $(0.0109)$ & $(0.318)$ & $(0.497)$ & $(0.00679)$ & $(0.320)$ & $(0.497)$\\\hline
        \multirow{2}{*}{$\rho=0.1$} & mean & $2.999$ & $-2.091$ & $1.035$ & $2.949$ & $-2.026$ & $1.003$ \\
        & RMSE & $(0.0146)$ & $(0.319)$ & $(0.496)$ & $(0.0509)$ & $(0.297)$ & $(0.480)$ \\\hline
        \multirow{2}{*}{$\rho=0.2$} & mean & $2.998$ & $-2.091$ & $1.035$ & $2.898$ & $-1.955$ & $0.967$ \\
        & RMSE & $(0.0142)$ & $(0.319)$ & $(0.496)$ & $(0.102)$ & $(0.290)$ & $(0.464)$ \\\hline
        \multirow{2}{*}{$\rho=0.3$} & mean & $2.998$ & $-2.092$ & $1.036$ & $2.846$ & $-1.885$ & $0.932$ \\
        & RMSE & $(0.0139)$ & $(0.319)$ & $(0.497)$ & $(0.155)$ & $(0.299)$ & $(0.452)$\\\hline
        \multirow{2}{*}{$\rho=0.4$} & mean & $2.998$ & $-2.091$ & $1.036$ & $2.792$ & $-1.815$ & $0.897$\\
        & RMSE & $(0.0135)$ & $(0.319)$ & $(0.497)$ & $(0.208)$ & $(0.324)$ & $(0.442)$\\\hline
        \multirow{2}{*}{$\rho=0.5$} & mean & $2.998$ & $-2.092$ & $1.036$ & $2.738$ & $-1.744$ & $0.862$\\
        & RMSE & $(0.0132)$ & $(0.319)$ & $(0.497)$ & $(0.262)$ & $(0.361)$ & $(0.436)$ \\\hline
        \multirow{2}{*}{$\rho=0.6$} & mean & $2.998$ & $-2.091$ & $1.036$ & $2.683$ & $-1.674$ & $0.827$\\
        & RMSE & $(0.0129)$ & $(0.319)$ & $(0.497)$ & $(0.317)$ & $(0.408)$ & $(0.434)$\\\hline
        \multirow{2}{*}{$\rho=0.7$} & mean & $2.998$ & $-2.092$ & $1.036$ & $2.626$ & $-1.604$ & $0.792$\\
        & RMSE & $(0.0126)$ & $(0.319)$ & $(0.497)$ & $(0.374)$ & $(0.460)$ & $(0.434)$ \\\hline
        \multirow{2}{*}{$\rho=0.8$} & mean & $2.998$ & $-2.092$ & $1.036$ & $2.568$ & $-1.534$ & $0.757$\\
        & RMSE & $(0.0124)$ & $(0.319)$ & $(0.496)$ & $(0.432)$ & $(0.517)$ & $(0.439)$\\\hline
        \multirow{2}{*}{$\rho=0.9$} & mean & $2.998$ & $-2.092$ & $1.036$ & $2.509$ & $-1.464$ & $0.722$ \\
        & RMSE & $(0.0121)$ & $(0.319)$ & $(0.497)$ & $(0.491)$ & $(0.578)$ & $(0.445)$ \\\hline
        \multirow{2}{*}{$\rho=1.0$} & mean & $2.998$ & $-2.091$ & $1.036$ & $2.449$ & $-1.394$ & $0.687$\\
        & RMSE & $(0.0119)$ & $(0.319)$ & $(0.497)$ & $(0.551)$ & $(0.640)$ & $(0.456)$\\
    \end{tabular}
    \caption{Estimation of $\theta$ by the proposed method and LGA with small $\rho$}
    \label{tab:sim:1dim:small:theta:estimate}
\end{table}
Note that the biases of the estimation by LGA increase as the true value of $\rho$ gets larger, while the estimation by our proposal method is not influenced by the true value of $\rho$. This result of the simulation supports the theoretical discussion in Section 4 stating the consistency of $\hat{\theta}_{n}$, and necessity to consider the convolutional observation scheme where the LGA method does not work properly. 

Secondly, we consider the estimation and test with large $\rho_{\star}$ such that $\rho_{\star}=10,15,20$ to see if our proposal methods work even for large $\rho$. We note that the maximum values of $\mathcal{T}_{n}$ for $\rho=10,15,20$ in 1000 iterations are $-55.091$, $-68.462$ and $-79.105$, and hence we can detect the smoothed observation easily. Table \ref{tab:sim:1dim:large:rho:estimate} shows the empirical means and RMSEs of $\hat{\rho}_{n}$ for $\rho=10,15,20$ and we can see that the RMSEs increase as $\rho$'s increase; it indicates the difficulty to estimate $\rho$ accurately when $\rho_{\star}$ is large.
\begin{table}[ht]
    \centering
    \begin{tabular}{c|c|c|c}
         & $\rho=10$ & $\rho=15$ & $\rho=20$\\\hline
        mean of $\hat{\rho}_{n}$ & $9.919$ & $14.980$ & $19.751$\\
        RMSE of $\hat{\rho}_{n}$ & $(0.145)$ & $(0.240)$ & $(0.409)$\\
    \end{tabular}
    \caption{The performance of $\hat{\rho}_{n}$ for $\rho=10, 15, 20$ in 1000 iterations}
    \label{tab:sim:1dim:large:rho:estimate}
\end{table}
Table \ref{tab:sim:1dim:large:theta:estimate} summarises the estimation for $\theta$ by means and RMSE, and tells us that although the large RMSE of $\hat{\rho}_{n}$ results in increase of RMSE of $\hat{\alpha}_{n}$, estimation by our method is substantially better than that by LGA of course.
\begin{table}[ht]
    \centering
    \begin{tabular}{c|c|c|c|c|c|c|c}
        \multicolumn{2}{c|}{} & \multicolumn{3}{c|}{the proposed method} & \multicolumn{3}{c}{LGA}\\\hline
        \multicolumn{2}{c|}{}& $\alpha$ & $\beta^{\left(1\right)}$ & $\beta^{\left(2\right)}$ & $\alpha$ & $\beta^{\left(1\right)}$ & $\beta^{\left(2\right)}$\\\hline
        \multicolumn{2}{c|}{true value}& $3.0$ & $-2.0$ & $1.0$ & $3.0$ & $-2.0$ & $1.0$\\\hline
        \multirow{2}{*}{$\rho=10$} & mean & $2.989$ & $-2.101$ & $1.030$ & $0.933$ & $-0.204$ & $0.0811$ \\
        & RMSE & $(0.0347)$ & $(0.323)$ & $(0.496)$ & $(2.067)$ & $(1.796)$ & $(0.920)$\\\hline
        \multirow{2}{*}{$\rho=15$} & mean & $2.996$ & $-2.095$ & $1.027$ & $0.765$ & $-0.138$ & $0.0473$ \\
        & RMSE & $(0.0475)$ & $(0.321)$ & $(0.495)$ & $(2.235)$ & $(1.862)$ & $(0.953)$ \\\hline
        \multirow{2}{*}{$\rho=20$} & mean & $2.977$ & $-2.090$ & $1.024$ & $0.664$ & $-0.104$ & $0.0302$ \\
        & RMSE & $(0.0526)$ & $(0.319)$ & $(0.493)$ & $(2.336)$ & $(1.896)$ & $(0.970)$ \\
    \end{tabular}
    \caption{Estimation of $\theta$ by the proposed method with large $\rho$}
    \label{tab:sim:1dim:large:theta:estimate}
\end{table}

\subsection{2-dimensional simulation}
We consider the following 2-dimensional stochastic differential equation whose solution is a 2-dimensional OU process:
\begin{align*}
    \mathrm{d}\left[\begin{matrix}
    X_{t}^{\left(1\right)}\\
    X_{t}^{\left(2\right)}
    \end{matrix}\right]
    &=
    \left(
    \left[\begin{matrix}
    \beta^{\left(1\right)} & \beta^{\left(2\right)}\\
    \beta^{\left(4\right)} & \beta^{\left(5\right)}
    \end{matrix}\right]
    \left[\begin{matrix}
    X_{t}^{\left(1\right)}\\
    X_{t}^{\left(2\right)}
    \end{matrix}\right]
    +
    \left[\begin{matrix}
    \beta^{\left(3\right)}\\
    \beta^{\left(6\right)}
    \end{matrix}\right]
    \right)\mathrm{d}t
    +
    \left[\begin{matrix}
    \alpha^{\left(1\right)} & \alpha^{\left(2\right)}\\
    \alpha^{\left(2\right)} & \alpha^{\left(3\right)}
    \end{matrix}\right]\mathrm{d}w_{t}, X_{-\lambda}=\left[\begin{matrix}
    0\\
    0
    \end{matrix}\right],
\end{align*}
$\lambda=10^{-7/3}$.
The simulation is conducted with the settings as follows: firstly, we iterate the OU process by Euler-Maruyama scheme with the simulation sample size $n_{\mathrm{sim}}=10^{5+m}$, $T_{\mathrm{sim}}=10^{5/3}$ and discretisation step $h_{\mathrm{sim}}=10^{-10/3-m}$, where $m=2$ is the precision parameter for approximation of convolution; in the second place, we approximate the convoluted process with summation such that
\begin{align*}
    \overline{X}_{ih_{n},n}^{\left(j\right)}\approxeq \begin{cases}\frac{1}{\left[10^{m}\rho^{\left(j\right)}\right]}\sum_{k=0}^{10^{m}-1}X_{ih_{n}-kh_{\mathrm{sim}}}^{\left(j\right)}&\text{ if }\left[10^{m}\rho^{\left(j\right)}\right]\ge1,\\
    X_{ih_{n}}^{\left(j\right)}&\text{ if } \left[10^{m}\rho^{\left(j\right)}\right]<0,
    \end{cases}
\end{align*}
where $i=0,\ldots,n$, $j=1,2,$ the sampling scheme for inference is defined as $n=10^{5}$ and $h_{n}=10^{-10/3}$; the true value of $\rho$, $\alpha$ and $\beta$ are set as $\rho_{\star}=\left[2,4\right]^{T}$, $\alpha_{\star}=\left[2,0,3\right]^{T}$, $\beta_{\star}=\left[-2, -0.4, 0, 0.1, -3, 5\right]^{T}$; the parameter spaces are defined as $\Theta_{\rho}=\left[0,10\right]^{2}$, $\Theta_{1}=\left[1+10^{-8},10\right]\times\left[-1+10^{-8},1-10^{-8}\right]\times\left[1+10^{-8},10\right]$, and $\Theta_{2}=\left[-10,10\right]^{6}$; the total iteration number is set to 1000.

Table \ref{tab:sim:2dim:rho:estimate} summarises the estimation for $\rho$ with the method proposed in Section 3 (the inverse of $r$ is computationally obtained) with empirical means and empirical RMSEs of $\hat{\rho}_{n}$ in 1000 iterations. We can see that $\hat{\rho}_{n}$ is sufficiently precise to estimate the true value of $\rho$ indeed in this result, which is significant to estimate the other parameters $\alpha$ and $\beta$.
\begin{table}[ht]
    \centering
    \begin{tabular}{c|c|c}
         & $\rho^{\left(1\right)}$ & $\rho^{\left(2\right)}$\\\hline
         true value & $2.0$ & $4.0$ \\\hline
         empirical mean & $1.988$ & $3.966$ \\\hline
         empirical RMSE & $(0.0207)$ & $(0.0514)$
    \end{tabular}
    \caption{summary for $\rho$ estimate}
    \label{tab:sim:2dim:rho:estimate}
\end{table}
We also note that the maximum values of the test statistics for smoothed observation proposed in Section 3.2 in 1000 iterations are $-17.947$ and $-33.159$ for each axis. The $p$-value for them are smaller than $10^{-16}$; therefore, we can conclude that it is possible to detect the smoothed observation with the proposed test statistic in the case $\rho^{\left(i\right)}=2.0$ if $d=2$ from this result.

With respect to the estimation for $\alpha$ and $\beta$, we compare the estimates by our proposal method with that by LGA which does not concern convolutional observation. Table \ref{tab:sim:2dim:alpha:estimate} is the summary for $\alpha$ estimate by both the methods: 
\begin{table}[ht]
    \centering
    \begin{tabular}{c|c|c|c|c}
         & & $\alpha^{\left(1\right)}$ & $\alpha^{\left(2\right)}$ & $\alpha^{\left(3\right)}$\\\hline
         & true value & $2.0$ & $0.0$ & $3.0$ \\\hline
         \multirow{2}{*}{Our proposal}& mean & $1.993$ &  $0.000256$ & $2.992$ \\
         & RMSE & $(0.0115)$ & $(0.00739)$ & $(0.0213)$\\\hline
         \multirow{2}{*}{LGA}& mean & $1.295$ & $-0.00320$ & $1.442$ \\
         & RMSE & $(0.705)$ & $(0.0154)$ & $(1.558)$\\
    \end{tabular}
    \caption{summary for $\alpha$ estimate}
    \label{tab:sim:2dim:alpha:estimate}
\end{table}
we can see that the estimation precision for $\alpha$ by our proposal outperforms those by LGA. This results support validity of our estimation method if we have convolutional observation for diffusion processes. Regarding $\beta$, the simulation result is summarised in Table \ref{tab:sim:2dim:beta:estimate}:
\begin{table}[ht]
    \centering
    \begin{tabular}{c|c|c|c|c|c|c|c}
         & & $\beta^{\left(1\right)}$ & $\beta^{\left(2\right)}$ & $\beta^{\left(3\right)}$ & $\beta^{\left(4\right)}$ & $\beta^{\left(5\right)}$ & $\beta^{\left(6\right)}$\\\hline
         & true value & $-2.0$ & $-0.4$ & $0.0$ & $0.1$ & $-3.0$ & $5.0$ \\\hline
         \multirow{2}{*}{Our proposal}& mean & $-2.137$ & $-0.408$ & $-0.0439$ & $0.0788$ & $-3.103$ & $5.091$ \\
         & RMSE & $(0.362)$ & $(0.252)$ & $(0.540)$ & $(0.473)$ & $(0.399)$ & $(0.777)$\\\hline
         \multirow{2}{*}{LGA}& mean & $-0.917$ & $0.340$ & $-0.326$ & $-0.696$ & $0.221$ & $1.243$\\
         & RMSE & $(1.093)$ & $(0.802)$ & $(0.386)$ & $(0.804)$ & $(3.242)$& $(3.765)$\\
    \end{tabular}
    \caption{summary for $\beta$ estimate}
    \label{tab:sim:2dim:beta:estimate}
\end{table}
though the estimation for $\beta^{\left(3\right)}$ by our method has the smaller bias in comparison to that by LGA, the RMSE of our method is larger than that of LGA; in the estimation for other parameters, our proposal method outperforms the method by LGA. We can conclude that our proposal for estimation of $\alpha$ and $\beta$ concerning convolutional observation performs better than that not considering this observation scheme.

\section{Real data analysis}

In this section, we analyse the EEG dataset named S02E.mat provided in ``2.\ Two class motor imagery (002-2014)" of \citet{BNCI-2014}. The datasets including S02E.mat are also studied by \citet{Steyrl-et-al-2016}.

\subsection{Estimation and test for the smoothing parameters} In the first place, we pick up the first 15 axes of the dataset and compute $\hat{\rho}_{n}$ and $\mathcal{T}_{n}$ proposed in Section 3.1 and 3.2 respectively. The results are shown in Table \ref{tab:rda:est_and_tst}.
\begin{table}[ht]
    \centering
    \begin{tabular}{c|c|c|c|c|c}
         & 1st axis & 2nd axis & 3rd axis & 4th axis & 5th axis\\\hline
        $\hat{\rho}_{n}$ & $0.449$ & $1.037$ & $0.894$ & $0.736$ & $0.937$ \\\hline
        $\mathcal{T}_{n}$ & $-20.398$ & $-58.631$ & $-46.649$ & $-35.201$ & $-49.741$
    \end{tabular}\\
    \begin{tabular}{c|c|c|c|c|c}
         & 6th axis & 7th axis & 8th axis & 9th axis & 10th axis\\\hline
        $\hat{\rho}_{n}$ & $0.951$ & $0.971$ & $1.017$ & $0.958$ & $0.967$ \\\hline
        $\mathcal{T}_{n}$ & $-51.392$ & $-52.607$ & $-55.455$ & $-51.221$ & $-51.797$
    \end{tabular}\\
    \begin{tabular}{c|c|c|c|c|c}
        & 11th axis & 12th axis & 13th axis & 14th axis & 15th axis\\\hline
        $\hat{\rho}_{n}$ & $0.949$ & $0.649$ & $0.952$ & $0.977$ & $0.932$ \\\hline
        $\mathcal{T}_{n}$ & $-50.457$ & $-30.094$ & $-50.633$ & $-50.978$ & $-48.842$
    \end{tabular}\\
    \caption{The values of $\hat{\rho}_{n}$ and $\mathcal{T}_{n}$ for the first 15 axes of S02.mat by \citet{BNCI-2014}.}
    \label{tab:rda:est_and_tst}
\end{table}
We can observe that all the 15 time series data have the smoothing parameter $\rho>0$ with statistical significance when we assume ordinary significance levels. These results motivate us to use our methods for parametric estimation proposed in Section 4 when we fit stochastic differential equations for these data.

\subsection{Parametric estimation for a diffusion process}
We fit a 1-dimensional OU process for the time series data in the 2nd column of the data file S02E.mat with 512Hz observation for 222 seconds (the plot of the path can be seen in Figure \ref{fig:BNCI:paths}), whose $\hat{\rho}_{n}=1.037$ is the largest among those for the 15 axes and it is larger than 0 with statistical significance. According to the simulation result shown in Section 5.1, this size of the smoothing parameter gives critical biases when we estimate $\alpha$ and $\beta$ with LGA method not concerning convolutional observation scheme.

The stochastic differential equation with parameters $\alpha\in\Theta_{1}:=\left[0.01,200\right]$ and $\beta\in\Theta_{2}:=\left[-100,-0.01\right]\times\left[-100,100\right]$ is defined as follows:
\begin{align*}
    \mathrm{d}X_{t}=\left(\beta^{\left(1\right)}X_{t}+\beta^{\left(2\right)}\right)\mathrm{d}t+\alpha\mathrm{d}w_{t},\ X_{-\lambda}=x_{-\lambda}.
\end{align*}
We set 5 seconds as the time unit: hence $n=113664$ and $h_{n}=1/\left(5\times512\right)$. If we fit the OU process with the LGA method, i.e., we do not concern convolutional observation scheme, we obtain the fitting result such that
\begin{align*}
    \mathrm{d}X_{t}=\left(\left(-17.378\right)X_{t}+\left(-1.091\right)\right)\mathrm{d}t+\left(122.892\right)\mathrm{d}w_{t},\ X_{-\lambda}=x_{-\lambda}.
\end{align*}
In the next place, we fit $\alpha$ and $\beta$ with the least square method proposed in Section 4, and then we have the next fitting result:
\begin{align*}
    \mathrm{d}X_{t}=\left(\left(-2.146\right)X_{t}+\left(0.552\right)\right)\mathrm{d}t+\left(151.919\right)\mathrm{d}w_{t},\ X_{-\lambda}=x_{-\lambda}.
\end{align*}
It is worth noting that this fitting result is substantially different to that by LGA as shown above: hence these results indicate the significance to examine if the observation is convoluted with the smoothing parameter $\rho>0$ and otherwise the estimation is strongly biased.

\section{Summary}

We have discussed the convolutional observation scheme which deals with the smoothness of observation in comparison to ordinary diffusion processes. The first contribution is to propose this new observation scheme with the statistical test to confirm whether this scheme is valid in real data. The second one is to prove consistency of the estimator $\hat{\rho}_{n}$ for the smoothing parameter $\rho$, those for parameters in diffusion and drift coefficient, i.e., $\alpha$ and $\beta$, of the latent diffusion process $\left\{X_{t}\right\}$. Thirdly, we have examined the performance of those estimators and the test statistics in computational simulation, and verified these statistics work well in realistic settings. In the fourth place, we have shown a real example of observation where $\rho\neq0$ holds with statistical significance.

These contributions, especially the third one, will cultivate the motivation to study statistical approaches for convolutionally observed diffusion processes furthermore, such as estimation of kernel function $V$ appearing in the convoluted diffusion $\overline{X}_{t}:=\left(V\ast X\right)\left(t\right)$, test theory for parameters $\alpha$ and $\beta$ as likelihood-ratio-type test statistics \citep[for example, see][]{Kitagawa-Uchida-2014, Nakakita-Uchida-2019b}, large deviation inequalities for quasi-likelihood functions and associated discussion of Bayes-type estimators \citep[e.g.,][]{Yoshida-2011, Ogihara-Yoshida-2011, Clinet-Yoshida-2017, Nakakita-Uchida-2018}. By these future works, it is expected that the applicability of stochastic differential equations in real data analysis and contributions to the areas with high frequency observation of phenomena such as EEG will be enhanced.

\bibliography{bib190620}

\begin{thebibliography}{}

\bibitem[A{\"{i}}t-Sahalia and Jacod, 2014]{Ait-Sahalia-Jacod-2014}
A{\"{i}}t-Sahalia, Y. and Jacod, J. (2014).
\newblock {\em High-frequency financial econometrics}.
\newblock Princeton University Press, Princeton.

\bibitem[Bibby and S\o{}rensen, 1995]{Bibby-Sorensen-1995}
Bibby, B.~M. and S\o{}rensen, M. (1995).
\newblock Martingale estimating functions for discretely observed diffusion
  processes.
\newblock {\em Bernoulli}, 1:17--39.

\bibitem[Bibinger et~al., 2014]{Bibinger-et-al-2014}
Bibinger, M., Hautsch, N., Malec, P., and Rei{\ss}, M. (2014).
\newblock Estimating the quadratic covariation matrix from noisy observations:
  {L}ocal method of moments and efficiency.
\newblock {\em The Annals of Statistics}, 42(4):1312--1346.

\bibitem[{BNCI Horizon 2020}, 2014]{BNCI-2014}
{BNCI Horizon 2020} (2014).
\newblock Two class motor imagery.
  http://bnci-horizon-2020.eu/database/data-sets.

\bibitem[Clinet and Yoshida, 2017]{Clinet-Yoshida-2017}
Clinet, S. and Yoshida, N. (2017).
\newblock Statistical inference for ergodic point processes and application to
  limit order book.
\newblock {\em Stochastic Processes and their Applications}, 127(6):1800--1839.

\bibitem[Ditlevsen and S\o{}rensen, 2004]{Ditlevsen-Sorensen-2004}
Ditlevsen, S. and S\o{}rensen, M. (2004).
\newblock Inference for observations of integrated diffusion processes.
\newblock {\em Scandinavian Journal of Statistics}, 31(3):417--429.

\bibitem[Florens-Zmirou, 1989]{Florens-Zmirou-1989}
Florens-Zmirou, D. (1989).
\newblock Approximate discrete time schemes for statistics of diffusion
  processes.
\newblock {\em Statistics}, 20(4):547--557.

\bibitem[Genon-Catalot and Jacod, 1993]{Genon-Catalot-Jacod-1993}
Genon-Catalot, V. and Jacod, J. (1993).
\newblock On the estimation of the diffusion coefficient for multidimensional
  diffusion processes.
\newblock {\em Annales de l'Institut Henri Poincar\'{e} Probabilit\'{e}s et
  statistiques}, 29:119--151.

\bibitem[Gloter, 2000]{Gloter-2000}
Gloter, A. (2000).
\newblock Discrete sampling of an integrated diffusion process and parameter
  estimation of the diffusion coefficient.
\newblock {\em ESAIM: Probability and Statistics}, 4:205--227.

\bibitem[Gloter, 2006]{Gloter-2006}
Gloter, A. (2006).
\newblock Parameter estimation for a discretely observed integrated diffusion
  process.
\newblock {\em Scandinavian Journal of Statistics}, 33(1):83--104.

\bibitem[Gloter and Gobet, 2008]{Gloter-Gobet-2008}
Gloter, A. and Gobet, E. (2008).
\newblock {LAMN} property for hidden processes: {The} case of integrated
  diffusions.
\newblock {\em Annales de l'Institut Henri Poincar\'{e} Probabilit\'{e}s et
  statistiques}, 44(1):104–--128.

\bibitem[Iacus, 2008]{Iacus-2008}
Iacus, S.~M. (2008).
\newblock {\em Simulation and inference for stochastic differential equations:
  with {R} examples}.
\newblock Springer-Verlag, New York.

\bibitem[Ibragimov and Has'minskii, 1981]{Ibragimov-Hasminskii-1981}
Ibragimov, I.~A. and Has'minskii, R.~Z. (1981).
\newblock {\em Statistical estimation}.
\newblock Springer Verlag, New York.

\bibitem[Jacod et~al., 2009]{Jacod-et-al-2009}
Jacod, J., Li, Y., Mykland, P.~A., Podolskij, M., and Vetter, M. (2009).
\newblock Microstructure noise in the continuous case: the pre-averaging
  approach.
\newblock {\em Stochastic Processes and their Applications}, 119(7):2249--2276.

\bibitem[Jacod et~al., 2010]{Jacod-et-al-2010}
Jacod, J., Podolskij, M., and Vetter, M. (2010).
\newblock Limit theorems for moving averages of discretized processes plus
  noise.
\newblock {\em The Annals of Statistics}, 38(3):1478--1545.

\bibitem[Kessler, 1997]{Kessler-1997}
Kessler, M. (1997).
\newblock Estimation of an ergodic diffusion from discrete observations.
\newblock {\em Scandinavian Journal of Statistics}, 24:211--229.

\bibitem[Kessler and S\o{}rensen, 1999]{Kessler-Sorensen-1999}
Kessler, M. and S\o{}rensen, M. (1999).
\newblock Estimating equations based on eigenfunctions for a discretely
  observed diffusion process.
\newblock {\em Bernoulli}, 5(2):299--314.

\bibitem[Kitagawa and Uchida, 2014]{Kitagawa-Uchida-2014}
Kitagawa, H. and Uchida, M. (2014).
\newblock Adaptive test statistics for ergodic diffusion processes sampled at
  discrete times.
\newblock {\em Journal of Statistical Planning and Inference}, 150:84--110.

\bibitem[Koike, 2016]{Koike-2016}
Koike, Y. (2016).
\newblock Quadratic covariation estimation of an irregularly observed
  semimartingale with jumps and noise.
\newblock {\em Bernoulli}, 22(2):1894–--1936.

\bibitem[Nakakita and Uchida, 2018]{Nakakita-Uchida-2018}
Nakakita, S.~H. and Uchida, M. (2018).
\newblock Quasi-likelihood analysis of an ergodic diffusion plus noise. arxiv:
  1806.09401.

\bibitem[Nakakita and Uchida, 2019a]{Nakakita-Uchida-2019a}
Nakakita, S.~H. and Uchida, M. (2019a).
\newblock {{\GG{201906}}Inference for ergodic diffusions plus noise}.
\newblock {\em Scandinavian Journal of Statistics}, 46(2):470--516.

\bibitem[Nakakita and Uchida, 2019b]{Nakakita-Uchida-2019b}
Nakakita, S.~H. and Uchida, M. (2019b).
\newblock {{\GG{201912}}Adaptive test for ergodic diffusions plus noise}.
\newblock {\em Journal of Statistical Planning and Inference}, 203:131--150.

\bibitem[Ogihara, 2018]{Ogihara-2018}
Ogihara, T. (2018).
\newblock Parametric inference for nonsynchronously observed diffusion
  processes in the presence of market microstructure noise.
\newblock {\em Bernoulli}, 24(4B):3318--3383.

\bibitem[Ogihara and Yoshida, 2011]{Ogihara-Yoshida-2011}
Ogihara, T. and Yoshida, N. (2011).
\newblock Quasi-likelihood analysis for the stochastic differential equation
  with jumps.
\newblock {\em Statistical inference for stochastic processes}, 14(3):189--229.

\bibitem[S\o{}rensen, 2011]{Sorensen-2011}
S\o{}rensen, M. (2011).
\newblock Prediction-based estimating functions: review and new developments.
\newblock {\em Brazilian Journal of Probability and Statistics},
  25(3):362--391.

\bibitem[Steyrl et~al., 2016]{Steyrl-et-al-2016}
Steyrl, D., Scherer, R., Faller, J., and M{\"{u}}ller-Putz, G.~R. (2016).
\newblock Random forests in non-invasive sensorimotor rhythm brain-computer
  interfaces: a practical and convenient non-linear classifier.
\newblock {\em Biomedizinische Technik}, 61:77--86.

\bibitem[Uchida and Yoshida, 2012]{Uchida-Yoshida-2012}
Uchida, M. and Yoshida, N. (2012).
\newblock Adaptive estimation of an ergodic diffusion process based on sampled
  data.
\newblock {\em Stochastic Processes and their Applications}, 122(8):2885--2924.

\bibitem[Uchida and Yoshida, 2014]{Uchida-Yoshida-2014}
Uchida, M. and Yoshida, N. (2014).
\newblock Adaptive {Bayes} type estimators of ergodic diffusion processes from
  discrete observations.
\newblock {\em Statistical Inference for Stochastic Processes}, 17(2):181--219.

\bibitem[Yoshida, 1992]{Yoshida-1992}
Yoshida, N. (1992).
\newblock Estimation for diffusion processes from discrete observation.
\newblock {\em Journal of Multivariate Analysis}, 41(2):220--242.

\bibitem[Yoshida, 2011]{Yoshida-2011}
Yoshida, N. (2011).
\newblock Polynomial type large deviation inequalities and quasi-likelihood
  analysis for stochastic differential equations.
\newblock {\em Annals of the Institute of Statistical Mathematics},
  63:431--479.

\bibitem[Zhang et~al., 2005]{Zhang-et-al-2005}
Zhang, L., Mykland, P.~A., and A{\"{i}}t-Sahalia, Y. (2005).
\newblock A tale of two time scales: Determining integrated volatility with
  noisy high-frequency data.
\newblock {\em Journal of the American Statistical Association},
  100(472):1394--1411.

\end{thebibliography}
\bibliographystyle{apalike}

\section*{Appendix A. Proofs for preliminary lemmas and main results}
\subsection*{Nonasymptotic results}
We assume $\Delta \le \lambda$,  $k\in\mathbf{N}, M>0$, and consider a class of $\overline{\mathbf{R}}^{k}\otimes\overline{\mathbf{R}}^{d}$-valued kernel functions on $\mathbf{R}$ denoted as $\mathcal{K}\left(\Delta,k,M\right)$ such that for all $\Phi_{\Delta}\in\mathcal{K}\left(\Delta,k,M\right)$, it holds:
\begin{align*}
    \text{(i) }&\text{supp}\Phi_{\Delta}\subset \left[0,\Delta\right],\\
    \text{(ii) }&\text{for all } f:\left[0,\Delta\right]\times\Omega\to\mathbf{R}^{k}, \omega\in\Omega,\left|\int_{0}^{\Delta}\Phi_{\Delta}\left(\Delta-s\right)f\left(s,\omega\right)\mathrm{d}s\right|\le M\sup_{s\in\left[0,\Delta\right]}\left|f\left(s,\omega\right)\right|\\
    \text{(iii) }&\text{for all }t_{0}\ge -\lambda, f:\mathbf{R}^{d}\to\mathbf{R}\text{\ which is continuous and at most polynomial growth},\\
    & \mathbf{E}\left[\left.\int_{t-\Delta}^{t}\Phi_{\Delta}\left(t-s\right)f\left(X_{s}\right)\mathrm{d}s\right|\mathcal{F}_{t_{0}}\right]=\int_{t-\Delta}^{t}\Phi_{\Delta}\left(t-s\right)\mathbf{E}\left[\left.f\left(X_{s}\right)\right|\mathcal{F}_{t_{0}}\right]\mathrm{d}s.
\end{align*}
\begin{remark}
    {Note one sufficient condition for $\Phi_{\Delta}\in\mathcal{K}\left(\Delta,k,M\right)$ is (i) $\Phi_{\Delta}:\mathbf{R}\to\mathbf{R}^{k}\otimes\mathbf{R}^{d}$, (ii) $\text{supp}\Phi_{\Delta}\subset \left[0,\Delta\right]$, (iii) $\int_{0}^{\Delta}\left\|\Phi_{\Delta}\left(\Delta-s\right)\right\|\mathrm{d}s\le M$ and (iv) $\mathcal{B}\left(\left[0,\Delta\right]\right)$-measurable since}
    \begin{align*}
        \left|\int_{0}^{\Delta}\Phi_{\Delta}\left(\Delta-s\right)f\left(s,\omega\right)\mathrm{d}s\right|&\le \int_{0}^{\Delta}\left\|\Phi_{\Delta}\left(\Delta-s\right)\right\|\left|f\left(s,\omega\right)\right|\mathrm{d}s\le M\sup_{s\in\left[0,\Delta\right]}\left|f\left(s,\omega\right)\right|
    \end{align*}
    for Cauchy-Schwarz inequality, and Fubini's theorem.
\end{remark}
It is easily checked that $V_{\rho,h_{n}}\in \mathcal{K}\left(\max_{i=1,\ldots,d}\rho_{i}h_{n},d, d\right)$.

The next theorem is a generalisation of Proposition 2.2 of \citet{Gloter-2000},  Theorem 1 of \citet{Gloter-2006} and Corollary 1 of \citet{Nakakita-Uchida-2019a}.
\begin{theorem}\label{thmAppox}
    Set $t\ge0$, $\Delta\in\left(0,\lambda\right]$, $k\in\mathbf{N}$, $M>0$, $\Phi_{\Delta},\Psi_{\Delta}\in\mathcal{K}\left(\Delta,k,M\right)$, and assume [A1].
    Then we have
    \begin{align*}
        \int_{t-\Delta}^{t}\Phi_{\Delta}\left(t-s\right)X_{s}\mathrm{d}s
        & = \left(\int_{0}^{\Delta}\Phi_{\Delta}\left(\Delta-s\right)\mathrm{d}s\right)X_{t-\Delta}+ \left(\int_{0}^{\Delta}\Phi_{\Delta}\left(\Delta-s\right)s\mathrm{d}s\right)b\left(X_{t-\Delta}\right)\\
        &\quad+\int_{t-\Delta}^{t}\Phi_{\Delta}\left(t-s_{1}\right)\left(\int_{t-\Delta}^{s_{1}}a\left(X_{t-\Delta}\right)\mathrm{d}w_{s_{2}}\right)\mathrm{d}s_{1}+e_{t-\Delta,\Delta},
    \end{align*}
    where $e_{t-\Delta,\Delta}$ is an $\mathbf{R}^{k}$-valued $\mathcal{F}_{t}$-measurable random variable such that
    \begin{align}
        \text{(i) }&\left|\mathbf{E}\left[e_{t-\Delta,\Delta}|\mathcal{F}_{t-\Delta}\right]\right|\le C\left(M\right)\Delta^{2} \left(1+\left|X_{t-\Delta}\right|\right)^{C\left(M\right)},\\
        \text{(ii) }& 
        \text{for all} \ m>0,\ \mathbf{E}\left[\left|e_{t-\Delta,\Delta}\right|^{m}|\mathcal{F}_{t-\Delta}\right]\le C\left(m,M\right)\Delta^{m} \left(1+\left|X_{t-\Delta}\right|\right)^{C\left(m,M\right)},\\
        \text{(iii) }&\left|\mathbf{E}\left[e_{t-\Delta,\Delta}\left[\int_{t-\Delta}^{t}\Psi_{\Delta}\left(t-s_{1}\right)\left(\int_{t-\Delta}^{s_{1}}a\left(X_{t-\Delta}\right)\mathrm{d}w_{s_{2}}\right)\mathrm{d}s_{1}\right]|\mathcal{F}_{t-\Delta}\right]\right|\notag\\
        &\quad\le C\left(M\right)\Delta^{2} \left(1+\left|X_{t-\Delta}\right|\right)^{C\left(M\right)}.
    \end{align}
\end{theorem}

The following corollaries correspond more directly to Proposition 2.2 and Theorem 2.3 of \citet{Gloter-2000}, or Proposition 1 and Theorem 1 of \citet{Gloter-2006}.

\begin{corollary}\label{corApprox1}
    Set $t\ge0$, $\Delta\in\left(0,\lambda\right]$, $M>0$ and $\Phi_{\Delta}\in\mathcal{K}\left(\Delta,d,M\right)$. We assume [A1] and $\int_{0}^{\Delta}\Phi_{\Delta}\left(\Delta-s\right)\mathrm{d}s=I_{d}$.
    Then we obtain
    \begin{align*}
        \int_{t-\Delta}^{t}\Phi_{\Delta}\left(t-s\right)X_{s}\mathrm{d}s
        & = X_{t-\Delta}+\int_{t-\Delta}^{t}\Phi_{\Delta}\left(t-s_{1}\right)\left(\int_{t-\Delta}^{s_{1}}a\left(X_{t-\Delta}\right)\mathrm{d}w_{s_{2}}\right)\mathrm{d}s_{1}+e_{t-\Delta,\Delta},
    \end{align*}
    where $e_{t-\Delta,\Delta}$ is an $\mathbf{R}^{d}$-valued $\mathcal{F}_{t}$-measurable random variable such that
    \begin{align}
        \text{(i) }&\left|\mathbf{E}\left[e_{t-\Delta,\Delta}|\mathcal{F}_{t-\Delta}\right]\right|\le C\left(M\right)\Delta \left(1+\left|X_{t-\Delta}\right|\right)^{C\left(M\right)},\\
        \text{(ii) }&\text{for all } m>0,\ \mathbf{E}\left[\left|e_{t-\Delta,\Delta}\right|^{m}|\mathcal{F}_{t-\Delta}\right]\le C\left(m,M\right)\Delta^{m} \left(1+\left|X_{t-\Delta}\right|\right)^{C\left(m,M\right)}.
    \end{align}
\end{corollary}

\begin{remark}
    This corollary leads to the same result as Proposition 2.2 of \citet{Gloter-2000} by setting $\Phi_{\Delta}\left(s\right)=I_{d}\left(\mathbf{1}_{\left[0,\Delta\right]}\left(s\right)/\Delta\right)$; we can see $\Phi_{\Delta}\in\mathcal{K}\left(\Delta,d,\sqrt{d}\right)$ because $\int_{0}^{\Delta}\left\|\Phi_{\Delta}\left(s\right)\right\|\mathrm{d}s=\sqrt{d}$.
    We have the following equalities
    \begin{align*}
        \int_{0}^{\Delta}\Phi_{\Delta}\left(s\right)\mathrm{d}s&=I_{d},\\
        \int_{t-\Delta}^{t}\Phi_{\Delta}\left(t-s_{1}\right)\left(\int_{t-\Delta}^{s_{1}}a\left(X_{t-\Delta}\right)\mathrm{d}w_{s_{2}}\right)\mathrm{d}s_{1}&=\int_{t-\Delta}^{t}\left(\int_{s_{1}}^{t}\Phi_{\Delta}\left(t-s_{2}\right)\mathrm{d}s_{2}\right)a\left(X_{t-\Delta}\right)\mathrm{d}w_{s_{1}}\\
        &=\frac{1}{\Delta}a\left(X_{t-\Delta}\right)\int_{t-\Delta}^{t}\left(t-s_{1}\right)\mathrm{d}w_{s_{1}}.
    \end{align*}
    Then we can see that this result is identical to that of \citet{Gloter-2000}.
\end{remark}

\begin{corollary}\label{corCorApprox1}
    Set $t\ge0$, $\Delta\in\left(0,\lambda\right]$, $M>0$,  $\Phi_{\Delta}\in\mathcal{K}\left(\Delta,d,M\right)$, and $f\left(x,\xi\right)$ is a real-valued function such that $f:\mathbf{R}^{d}\times\Xi\to\mathbf{R}$, and $f$, $\partial_{x}f$ and $\partial_{x}^{2}f$ are polynomial growth uniformly in $\xi\in\Xi$. We assume [A1] and $\int_{0}^{\Delta}\Phi_{\Delta}\left(\Delta-s\right)\mathrm{d}s=I_{d}$.
    Then we obtain
    \begin{align*}
        &\sup_{\xi\in\Xi}\left|\mathbf{E}\left[f\left(\int_{t-\Delta}^{t}\Phi_{\Delta}\left(t-s\right)X_{s}\mathrm{d}s,\xi\right)-f\left(X_{t-\Delta},\xi\right)|\mathcal{F}_{t-\Delta}\right]\right|\\
        &\quad\le C\left(M\right)\Delta\left(1+\left|X_{t-\Delta}\right|\right)^{C\left(M\right)},\\
        &\mathbf{E}\left[\sup_{\xi\in\Xi}\left|f\left(\int_{t-\Delta}^{t}\Phi_{\Delta}\left(t-s\right)X_{s}\mathrm{d}s,\xi\right)-f\left(X_{t-\Delta},\xi\right)\right|^{m}|\mathcal{F}_{t-\Delta}\right]\\
        &\quad\le C\left(m,M\right)\Delta^{m/2}\left(1+\left|X_{t-\Delta}\right|\right)^{C\left(m,M\right)}
    \end{align*}
    for all $m>0$.
\end{corollary}

\begin{remark}
    We see that this result generalises Proposition 4 of \citet{Nakakita-Uchida-2019a}; let us set $p\in\mathbf{N}$ and $h>0$ such that $ph\le \lambda$, and $\Phi_{ph}$ as follows:
    \begin{align*}
        \Phi_{ph}\left(s\right)=\frac{1}{p}\sum_{i=0}^{p-1}\delta\left(ih-s\right)I_{d}.
    \end{align*}
    It is obvious that {$\Phi_{ph}\in\mathcal{K}\left(ph,d,d\right)$} and $\int_{0}^{ph}\Phi_{ph}\left(ph-s\right)\mathrm{d}s=I_{d}$.
\end{remark}

\begin{corollary}\label{corApprox2}
    Set $t\ge0$, $\Delta\in\left(0,\lambda\right]$, $k\in\mathbf{N}$, $M>0$ and $\Phi_{\Delta}\in\mathcal{K}\left(\Delta,k,M\right)$. We assume that $\int_{0}^{\Delta}\Phi_{\Delta}\left(\Delta-s\right)\mathrm{d}s=O$.
    Then we obtain
    \begin{align*}
        \int_{t-\Delta}^{t}\Phi_{\Delta}\left(t-s\right)X_{s}\mathrm{d}s
        & = \left(\int_{0}^{\Delta}\Phi_{\Delta}\left(\Delta-s\right)s\mathrm{d}s\right)b\left(X_{t-\Delta}\right)\\
        &\quad+\int_{t-\Delta}^{t}\Phi_{\Delta}\left(t-s_{1}\right)\left(\int_{t-\Delta}^{s_{1}}a\left(X_{t-\Delta}\right)\mathrm{d}w_{s_{2}}\right)\mathrm{d}s_{1}+e_{t-\Delta,\Delta},
    \end{align*}
    where $e_{t-\Delta,\Delta}$ is an $\mathbf{R}^{k}$-valued $\mathcal{F}_{t}$-measurable random variable such that
    \begin{align}
        \text{(i) }&\left|\mathbf{E}\left[e_{t-\Delta,\Delta}|\mathcal{F}_{t-\Delta}\right]\right|\le C\left(M\right)\Delta^{2} \left(1+\left|X_{t-\Delta}\right|\right)^{C\left(M\right)},\\
        \text{(ii) }&\text{for all } m>0,\ \mathbf{E}\left[\left|e_{t-\Delta,\Delta}\right|^{m}|\mathcal{F}_{t-\Delta}\right]\le C\left(m,M\right)\Delta^{m} \left(1+\left|X_{t-\Delta}\right|\right)^{C\left(m,M\right)},\\
        \text{(iii) }&\left|\mathbf{E}\left[e_{t-\Delta,\Delta}\left[\int_{t-\Delta}^{t}\Phi_{\Delta}\left(t-s_{1}\right)\left(\int_{t-\Delta}^{s_{1}}a\left(X_{t-\Delta}\right)\mathrm{d}w_{s_{2}}\right)\mathrm{d}s_{1}\right]|\mathcal{F}_{t-\Delta}\right]\right|^{m}\notag\\
        &\quad\le C\left(M\right)\Delta^{2} \left(1+\left|X_{t-\Delta}\right|\right)^{C\left(M\right)}.
    \end{align}
\end{corollary}

\begin{remark}
    We can have a result similar to Theorem 3.2 of \citet{Gloter-2000} if we assume that $\Delta'=\Delta/2$ and  $\Phi_{\Delta}\left(s\right)=I_{d}\left(\mathbf{1}_{\left[0,\Delta'\right]}\left(s\right)-\mathbf{1}_{\left[\Delta',\Delta\right]}\left(s\right)\right)/\Delta'$ where $\Delta\le \lambda$ and $k=d$. {$\Phi_{\Delta}\left(s\right)\in\mathcal{K}\left(\Delta,d,2\sqrt{d}\right)$} because $\int_{0}^{\Delta}\left\|\Phi_{\Delta}\left(\Delta-s\right)\right\|\mathrm{d}s=2\sqrt{d}$.
    We have the following equalities
    \begin{align*}
        &\int_{0}^{\Delta}\Phi_{\Delta}\left(\Delta-s\right)\mathrm{d}s=O,\\
        &\int_{0}^{\Delta}\Phi_{\Delta}\left(\Delta-s\right)s\mathrm{d}s=I_{d}\left(\frac{1}{\Delta'}\int_{\Delta'}^{\Delta}s\mathrm{d}s-\frac{1}{\Delta'}\int_{0}^{\Delta'}s\mathrm{d}s\right)=I_{d}\left(\Delta-\frac{\Delta'}{2}-\frac{\Delta'}{2}\right)=\Delta' I_{d},\\
        &\int_{t-\Delta}^{t}\Phi_{\Delta}\left(t-s_{1}\right)\left(\int_{t-\Delta}^{s_{1}}a\left(X_{t-\Delta}\right)\mathrm{d}w_{s_{2}}\right)\mathrm{d}s_{1}\\
        &\quad=\int_{t-\Delta}^{t}\left(\int_{s_{1}}^{t}\Phi_{\Delta}\left(t-s_{2}\right)\mathrm{d}s_{2}\right)a\left(X_{t-\Delta}\right)\mathrm{d}w_{s_{1}}\\
        &\quad=\frac{1}{\Delta'}a\left(X_{t-\Delta}\right)\int_{t-\Delta}^{t}\left(\int_{s_{1}}^{t}\left(\mathbf{1}_{\left[0,\Delta'\right]}\left(t-s_{2}\right)-\mathbf{1}_{\left[\Delta',\Delta\right]}\left(t-s_{2}\right)\right)\mathrm{d}s_{2}\right)\mathrm{d}w_{s_{1}}\\
        &\quad=\frac{1}{\Delta'}a\left(X_{t-\Delta}\right)\int_{t-\Delta}^{t}\left(\Delta'\mathbf{1}_{\left[\Delta',\Delta\right]}\left(t-s_{1}\right)+\left(t-s_{1}\right)\mathbf{1}_{\left[0,\Delta'\right]}\left(t-s_{1}\right)\right)\mathrm{d}w_{s_{1}}\\
        &\qquad-\frac{1}{\Delta'}a\left(X_{t-\Delta}\right)\int_{t-\Delta}^{t}\left(t-\Delta'-s_{1}\right)\mathbf{1}_{\left[0,\Delta'\right]}\left(t-\Delta'-s_{1}\right)\mathrm{d}w_{s_{1}}\\
        &\quad=\frac{1}{\Delta'}a\left(X_{t-\Delta}\right)\left(\Delta'\int_{t-\Delta}^{t-\Delta'}\mathrm{d}w_{s}+\int_{t-\Delta'}^{t}\left(t-s_{1}\right)\mathrm{d}w_{s_{1}}\right)\\
        &\qquad-\frac{1}{\Delta'}a\left(X_{t-\Delta}\right)\int_{t-\Delta}^{t-\Delta'}\left(t-\Delta'-s_{1}\right)\mathrm{d}w_{s_{1}}\\
        &\quad=\frac{1}{\Delta'}a\left(X_{t-\Delta}\right)\left(\int_{t-\Delta}^{t-\Delta'}\left(s_{1}-\left(t-\Delta\right)\right)\mathrm{d}w_{s}+\int_{t-\Delta'}^{t}\left(t-s_{1}\right)\mathrm{d}w_{s_{1}}\right)
    \end{align*}
    because of Lemma \ref{LemmaExchangeIntegrals} in Appendix B. Hence this result and Corollary \ref{corCorApprox1} give the same evaluation as \citet{Gloter-2000}.
\end{remark}

\begin{remark}
This corollary 
also generalises Corollary 1 of \citet{Nakakita-Uchida-2019a} when we ignore noise term; let us set $p\in\mathbf{N}$ and $h>0$ such that $2ph\le \lambda$, and $\Phi_{2ph}$ as follows:
    \begin{align*}
        \Phi_{2ph}\left(s\right)=\frac{1}{p}\sum_{i=0}^{p-1}\delta\left(s-ih\right)I_{d}-\frac{1}{p}\sum_{i=0}^{p-1}\delta\left(s-\left(p+i\right)h\right)I_{d}.
    \end{align*}
    It is obvious that {$\Phi_{2ph}\in\mathcal{K}\left(2ph,d,2d\right)$} and $\int_{0}^{2ph}\Phi_{2ph}\left(2ph-s\right)\mathrm{d}s=O$. We can evaluate 
    \begin{align*}
        &\int_{0}^{2ph}\Phi_{2ph}\left(2ph-s\right)s\mathrm{d}s=\frac{1}{p}\sum_{i=0}^{p-1}\left(2ph-ih\right)I_{d}-\frac{1}{p}\sum_{i=0}^{p-1}\left(2ph-\left(p+i\right)h\right)I_{d}=phI_{d},\\
        &\int_{t-2ph}^{t}\Phi_{2ph}\left(t-s_{1}\right)\left(\int_{t-2ph}^{s_{1}}a\left(X_{t-2ph}\right)\mathrm{d}w_{s_{2}}\right)\mathrm{d}s_{1}\\
        &\quad=a\left(X_{t-2ph}\right)\left(\frac{1}{p}\sum_{i=0}^{p-1}\int_{t-2ph}^{t-ih}\mathrm{d}w_{s}-\frac{1}{p}\sum_{i=0}^{p-1}\int_{t-2ph}^{t-\left(p+i\right)h}\mathrm{d}w_{s}\right)\\
        &\quad=a\left(X_{t-2ph}\right)\left(\frac{1}{p}\sum_{i=0}^{p-1}\left(i+1\right)\int_{t-\left(2p-i\right)h}^{t-\left(2p-i+1\right)h}\mathrm{d}w_{s}+\frac{1}{p}\sum_{i=0}^{p-1}\left(p-1-i\right)\int_{t-\left(p-i\right)h}^{t-\left(p-i+1\right)h}\mathrm{d}w_{s}\right).
    \end{align*}
\end{remark}

\begin{proof}[Proof of Theorem \ref{thmAppox}]
We have
\begin{align*}
    & \int_{t-\Delta}^{t}\Phi_{\Delta}\left(t-s\right)X_{s}\mathrm{d}s \\
    & = \int_{t-\Delta}^{t}\Phi_{\Delta}\left(t-s_{1}\right)\left(X_{t-\Delta}+\int_{t-\Delta}^{s_{1}}b\left(X_{s_{2}}\right)\mathrm{d}s_{2}+\int_{t-\Delta}^{s_{1}}a\left(X_{s_{2}}\right)\mathrm{d}w_{s_{2}}\right)\mathrm{d}s_{1}\\
    & = \left(\int_{t-\Delta}^{t}\Phi_{\Delta}\left(t-s_{1}\right)\mathrm{d}s\right)X_{t-\Delta}+\int_{t-\Delta}^{t}\Phi_{\Delta}\left(t-s_{1}\right)\left(\int_{t-\Delta}^{s_{1}}b\left(X_{t-\Delta}\right)\mathrm{d}s_{2}\right)\mathrm{d}s_{1}\\
    &\quad+\int_{t-\Delta}^{t}\Phi_{\Delta}\left(t-s_{1}\right)\left(\int_{t-\Delta}^{s_{1}}a\left(X_{t-\Delta}\right)\mathrm{d}w_{s_{2}}\right)\mathrm{d}s_{1}\\
    &\quad+\int_{t-\Delta}^{t}\Phi_{\Delta}\left(t-s_{1}\right)\left(\int_{t-\Delta}^{s_{1}}\left(b\left(X_{s_{2}}\right)-b\left(X_{t-\Delta}\right)\right)\mathrm{d}s_{2}\right)\mathrm{d}s_{1}\\
    &\quad+\int_{t-\Delta}^{t}\Phi_{\Delta}\left(t-s_{1}\right)\left(\int_{t-\Delta}^{s_{1}}\left(a\left(X_{s_{2}}\right)-a\left(X_{t-\Delta}\right)\right)\mathrm{d}w_{s_{2}}\right)\mathrm{d}s_{1}.
\end{align*}
Set
$e_{t-\Delta,\Delta}:=e_{t-\Delta,\Delta,1}+e_{t-\Delta,\Delta,2}$ where
\begin{align*}
    e_{t-\Delta,\Delta,1}&:=\int_{t-\Delta}^{t}\Phi_{\Delta}\left(t-s_{1}\right)\left(\int_{t-\Delta}^{s_{1}}\left(b\left(X_{s_{2}}\right)-b\left(X_{t-\Delta}\right)\right)\mathrm{d}s_{2}\right)\mathrm{d}s_{1},\\
    e_{t-\Delta,\Delta,2}&:=\int_{t-\Delta}^{t}\Phi_{\Delta}\left(t-s_{1}\right)\left(\int_{t-\Delta}^{s_{1}}\left(a\left(X_{s_{2}}\right)-a\left(X_{t-\Delta}\right)\right)\mathrm{d}w_{s_{2}}\right)\mathrm{d}s_{1}.
\end{align*}
We examine that the properties (i)-(iii) hold for this $e_{t-\Delta,\Delta}$.
The assumption of $\Phi_{\Delta}$ and martingale property of stochastic integral verify
\begin{align*}
    \mathbf{E}\left[ e_{t-\Delta,\Delta,2}|\mathcal{F}_{t-\Delta}\right]=\mathbf{0}.
\end{align*}
Then, 
in order to show (i) and (ii),
it is sufficient to prove the following inequalities.
\begin{align*}
    \left|\mathbf{E}\left[e_{t-\Delta,\Delta,1}|\mathcal{F}_{t-\Delta}\right]\right|&\le C\left(M\right)\Delta^{2}\left(1+\left|X_{t-\Delta}\right|\right)^{C\left(M\right)},\\
    \mathbf{E}\left[\left|e_{t-\Delta,\Delta,1}\right|^{m}|\mathcal{F}_{t-\Delta}\right]
    &\le C\left(m,M\right)\Delta^{3m/2}\left(1+\left|X_{t-\Delta}\right|\right)^{C\left(m,M\right)},\\
    \mathbf{E}\left[\left| e_{t-\Delta,\Delta,2}\right|^{m}|\mathcal{F}_{t-\Delta}\right]
    &\le C\left(m,M\right)\Delta^{m}\left(1+\left|X_{t-\Delta}\right|\right)^{C\left(m,M\right)}.
\end{align*}
{In the first place, we have}
\begin{align*}
    &\left|\mathbf{E}\left[\int_{t-\Delta}^{t}\Phi_{\Delta}\left(t-s_{1}\right)\left(\int_{t-\Delta}^{s_{1}}\left(b\left(X_{s_{2}}\right)-b\left(X_{t-\Delta}\right)\right)\mathrm{d}s_{2}\right)\mathrm{d}s_{1}|\mathcal{F}_{t-\Delta}\right]\right|\\
    &=\left|\int_{t-\Delta}^{t}\mathbf{E}\left[\Phi_{\Delta}\left(t-s_{1}\right)\left(\int_{t-\Delta}^{s_{1}}\left(b\left(X_{s_{2}}\right)-b\left(X_{t-\Delta}\right)\right)\mathrm{d}s_{2}\right)|\mathcal{F}_{t-\Delta}\right]\mathrm{d}s_{1}\right|\\
    &=\left|\int_{t-\Delta}^{t}\Phi_{\Delta}\left(t-s_{1}\right)\mathbf{E}\left[\int_{t-\Delta}^{s_{1}}\left(b\left(X_{s_{2}}\right)-b\left(X_{t-\Delta}\right)\right)\mathrm{d}s_{2}|\mathcal{F}_{t-\Delta}\right]\mathrm{d}s_{1}\right|\\
    &=\left|\int_{t-\Delta}^{t}\Phi_{\Delta}\left(t-s_{1}\right)\left(\int_{t-\Delta}^{s_{1}}\mathbf{E}\left[b\left(X_{s_{2}}\right)-b\left(X_{t-\Delta}\right)|\mathcal{F}_{t-\Delta}\right]\mathrm{d}s_{2}\right)\mathrm{d}s_{1}\right|\\
    &\le C\left(M\right)\sup_{s_{1}\in\left[t-\Delta,t\right]}\left|\int_{t-\Delta}^{s_{1}}\mathbf{E}\left[b\left(X_{s_{2}}\right)-b\left(X_{t-\Delta}\right)|\mathcal{F}_{t-\Delta}\right]\mathrm{d}s_{2}\right|\\
    &\le C\left(M\right)\sup_{s_{1}\in\left[t-\Delta,t\right]}\int_{t-\Delta}^{s_{1}}\left|\mathbf{E}\left[b\left(X_{s_{2}}\right)-b\left(X_{t-\Delta}\right)|\mathcal{F}_{t-\Delta}\right]\right|\mathrm{d}s_{2}\\
    &\le C\left(M\right)\Delta\sup_{s\in\left[t-\Delta,t\right]}\left|\mathbf{E}\left[b\left(X_{s}\right)-b\left(X_{t-\Delta}\right)|\mathcal{F}_{t-\Delta}\right]\right|\\
    &\le C\left(M\right)\Delta^{2}\left(1+\left|X_{t-\Delta}\right|\right)^{C\left(M\right)}
\end{align*}
because of Proposition A of \citet{Gloter-2000}. Secondly, we obtain
\begin{align*}
    &\left|\int_{t-\Delta}^{t}\Phi_{\Delta}\left(t-s_{1}\right)\left(\int_{t-\Delta}^{s_{1}}\left(b\left(X_{s_{2}}\right)-b\left(X_{t-\Delta}\right)\right)\mathrm{d}s_{2}\right)\mathrm{d}s_{1}\right|\\
    &\le C\left(M\right)\sup_{s_{1}\in\left[t-\Delta,t\right]}\left|\int_{t-\Delta}^{s_{1}}\left(b\left(X_{s_{2}}\right)-b\left(X_{t-\Delta}\right)\right)\mathrm{d}s_{2}\right|\\
    &\le C\left(M\right)\sup_{s_{1}\in\left[t-\Delta,t\right]}\int_{t-\Delta}^{s_{1}}\left|b\left(X_{s_{2}}\right)-b\left(X_{t-\Delta}\right)\right|\mathrm{d}s_{2}\\
    &\le C\left(M\right)\Delta \sup_{s\in\left[t-\Delta,t\right]}\left|b\left(X_{s}\right)-b\left(X_{t-\Delta}\right)\right|
\end{align*}
by {$\Phi_{\Delta}\in\mathcal{K}\left(\Delta,k,M\right)$}; therefore, we obtain
\begin{align*}
    &\mathbf{E}\left[\left|\int_{t-\Delta}^{t}\Phi_{\Delta}\left(t-s_{1}\right)\left(\int_{t-\Delta}^{s_{1}}\left(b\left(X_{s_{2}}\right)-b\left(X_{t-\Delta}\right)\right)\mathrm{d}s_{2}\right)\mathrm{d}s_{1}\right|^{m}|\mathcal{F}_{t-\Delta}\right]\\
    &\le C\left(m,M\right)\Delta^{m}\mathbf{E}\left[\sup_{s\in\left[t-\Delta,t\right]}\left|b\left(X_{s}\right)-b\left(X_{t-\Delta}\right)\right|^{m}|\mathcal{F}_{t-\Delta}\right]\\
    &\le C\left(m,M\right)\Delta^{3m/2}\left(1+\left|X_{t-\Delta}\right|\right)^{C\left(m,M\right)}
\end{align*}
due to Proposition 5.1 of \citet{Gloter-2000}. Thirdly, by H\"{o}lder's inequality and the Burkholder-Davis-Gundy one, and Proposition A of \citet{Gloter-2000},
\begin{align*}
    &\mathbf{E}\left[\left|\int_{t-\Delta}^{t}\Phi_{\Delta}\left(t-s_{1}\right)\left(\int_{t-\Delta}^{s_{1}}\left(a\left(X_{s_{2}}\right)-a\left(X_{t-\Delta}\right)\right)\mathrm{d}w_{s_{2}}\right)\mathrm{d}s_{1}\right|^{m}|\mathcal{F}_{t-\Delta}\right]\\
    &\le C\left(m,M\right)\mathbf{E}\left[\left|\sup_{s_{1}\in\left[t-\Delta,t\right]}\int_{t-\Delta}^{s_{1}}\left(a\left(X_{s_{2}}\right)-a\left(X_{t-\Delta}\right)\right)\mathrm{d}w_{s_{2}}\right|^{m}|\mathcal{F}_{t-\Delta}\right]\\
    &\le C\left(m,M\right)\mathbf{E}\left[\left|\int_{t-\Delta}^{t}\left\|a\left(X_{s_{2}}\right)-a\left(X_{t-\Delta}\right)\right\|^{2}\mathrm{d}s_{2}\right|^{m/2}|\mathcal{F}_{t-\Delta}\right]\\
    &\le C\left(m,M\right)\Delta^{m/2}\sup_{s\in\left[t-\Delta,t\right]}\mathbf{E}\left[\left\|a\left(X_{s}\right)-a\left(X_{t-\Delta}\right)\right\|^{m}|\mathcal{F}_{t-\Delta}\right]\\
    &\le  C\left(m,M\right)\Delta^{m}\left(1+\left|X_{t-\Delta}\right|\right)^{C\left(m,M\right)}.
\end{align*}
To show (iii), it is obvious that
\begin{align*}
    &\left|\mathbf{E}\left[e_{t-\Delta,\Delta,1}\left[\int_{t-\Delta}^{t}\Phi_{\Delta}\left(t-s_{1}\right)\left(\int_{t-\Delta}^{s_{1}}a\left(X_{t-\Delta}\right)\mathrm{d}w_{s_{2}}\right)\mathrm{d}s_{1}\right]|\mathcal{F}_{t-\Delta}\right]\right|\\
    &\le C\left(M\right)\Delta^{2} \left(1+\left|X_{t-\Delta}\right|\right)^{C\left(M\right)}
\end{align*}
due to H\"{o}lder's inequality and the evaluation analogous to $\mathbf{E}\left[\left|e_{t-\Delta,\Delta,2}\right|^{m}|\mathcal{F}_{t-\Delta}\right]$ such that
    \begin{align*}
        &\mathbf{E}\left[\left|\int_{t-\Delta}^{t}\Phi_{\Delta}\left(t-s_{1}\right)\left(\int_{t-\Delta}^{s_{1}}a\left(X_{t-\Delta}\right)\mathrm{d}w_{s_{2}}\right)\mathrm{d}s_{1}\right|^{m}|\mathcal{F}_{t-\Delta}\right]\\
        &\le C\left(m,M\right)\Delta^{m/2}\left(1+\left|X_{t-\Delta}\right|\right)^{C\left(m,M\right)}.
    \end{align*}
Hence it is sufficient to show
\begin{align*}
    &\left|\mathbf{E}\left[e_{t-\Delta,\Delta,2}\left[\int_{t-\Delta}^{t}\Phi_{\Delta}\left(t-s_{1}\right)\left(\int_{t-\Delta}^{s_{1}}a\left(X_{t-\Delta}\right)\mathrm{d}w_{s_{2}}\right)\mathrm{d}s_{1}\right]|\mathcal{F}_{t-\Delta}\right]\right|\\
    &\le C\left(M\right)\Delta^{2} \left(1+\left|X_{t-\Delta}\right|\right)^{C\left(M\right)}.
\end{align*}
We can evaluate the left hand side such that 
\begin{align*}
    &\left|\mathbf{E}\left[\int_{t-\Delta}^{t}\Phi_{\Delta}\left(t-s_{1}\right)\left(\int_{t-\Delta}^{s_{1}}\left(a\left(X_{s_{2}}\right)-a\left(X_{t-\Delta}\right)\right)\mathrm{d}w_{s_{2}}\right)\mathrm{d}s_{1}\right.\right.\\
    &\hspace{5cm}\left.\left.\left[\int_{t-\Delta}^{t}\Psi_{\Delta}\left(t-s_{1}'\right)\left(\int_{t-\Delta}^{s_{1}'}a\left(X_{t-\Delta}\right)\mathrm{d}w_{s_{2}'}\right)\mathrm{d}s_{1}'\right]|\mathcal{F}_{t-\Delta}\right]\right|\\
    &=\left|\int_{t-\Delta}^{t}\Phi_{\Delta}\left(t-s_{1}\right)\mathbf{E}\left[\left(\int_{t-\Delta}^{s_{1}}\left(a\left(X_{s_{2}}\right)-a\left(X_{t-\Delta}\right)\right)\mathrm{d}w_{s_{2}}\right)\right.\right.\\
    &\hspace{5cm}\left.\left.\left[\int_{t-\Delta}^{t}\Psi_{\Delta}\left(t-s_{1}'\right)\left(\int_{t-\Delta}^{s_{1}'}a\left(X_{t-\Delta}\right)\mathrm{d}w_{s_{2}'}\right)\mathrm{d}s_{1}'\right]|\mathcal{F}_{t-\Delta}\right]\mathrm{d}s_{1}\right|\\
    &\le C\left(M\right)\sup_{s_{1}\in\left[t-\Delta,t\right]}\left|\mathbf{E}\left[\left(\int_{t-\Delta}^{s_{1}}\left(a\left(X_{s_{2}}\right)-a\left(X_{t-\Delta}\right)\right)\mathrm{d}w_{s_{2}}\right)\right.\right.\\
    &\hspace{5cm}\left.\left.\left[\int_{t-\Delta}^{t}\Psi_{\Delta}\left(t-s_{1}'\right)\left(\int_{t-\Delta}^{s_{1}'}a\left(X_{t-\Delta}\right)\mathrm{d}w_{s_{2}'}\right)\mathrm{d}s_{1}'\right]|\mathcal{F}_{t-\Delta}\right]\right|\\
    &= C\left(M\right)\sup_{s_{1}\in\left[t-\Delta,t\right]}\left|\int_{t-\Delta}^{t}\Psi_{\Delta}\left(t-s_{1}'\right)\mathbf{E}\left[\left(\int_{t-\Delta}^{s_{1}}\left(a\left(X_{s_{2}}\right)-a\left(X_{t-\Delta}\right)\right)\mathrm{d}w_{s_{2}}\right)\right.\right.\\
    &\hspace{7cm}\left.\left.\left[\int_{t-\Delta}^{s_{1}'}a\left(X_{t-\Delta}\right)\mathrm{d}w_{s_{2}'}\right]|\mathcal{F}_{t-\Delta}\right]\mathrm{d}s_{1}'\right|\\
    &\le C\left(M\right)\\
    &\quad\times\sup_{s_{1}\in\left[t-\Delta,t\right]}\sup_{s_{1}'\in\left[t-\Delta,t\right]}\left|\mathbf{E}\left[\int_{t-\Delta}^{s_{1}}\left(a\left(X_{s_{2}}\right)-a\left(X_{t-\Delta}\right)\right)\mathrm{d}w_{s_{2}}\left[\int_{t-\Delta}^{s_{1}'}a\left(X_{t-\Delta}\right)\mathrm{d}w_{s_{2}'}\right]|\mathcal{F}_{t-\Delta}\right]\right|\\
    &\le C\left(M\right)\sup_{s_{1}\in\left[t-\Delta,t\right]}\left|\mathbf{E}\left[\int_{t-\Delta}^{s_{1}}\left(a\left(X_{s_{2}}\right)-a\left(X_{t-\Delta}\right)\right)\mathrm{d}w_{s_{2}}\left[\int_{t-\Delta}^{s_{1}}a\left(X_{t-\Delta}\right)\mathrm{d}w_{s_{2}'}\right]|\mathcal{F}_{t-\Delta}\right]\right|\\
    &\le C\left(M\right)\sup_{s_{1}\in\left[t-\Delta,t\right]}\left|\mathbf{E}\left[\int_{t-\Delta}^{s_{1}}\left(a\left(X_{s_{2}}\right)-a\left(X_{t-\Delta}\right)\right)\mathrm{d}w_{s_{2}}\left[\int_{t-\Delta}^{s_{1}}a\left(X_{t-\Delta}\right)\mathrm{d}w_{s_{2}'}\right]|\mathcal{F}_{t-\Delta}\right]\right|\\
    &= C\left(M\right)\sup_{s_{1}\in\left[t-\Delta,t\right]}\left|\mathbf{E}\left[\int_{t-\Delta}^{s_{1}}\left(a\left(X_{s_{2}}\right)-a\left(X_{t-\Delta}\right)\right)\mathrm{d}s_{2}\left[a\left(X_{t-\Delta}\right)\right]|\mathcal{F}_{t-\Delta}\right]\right|\\
    &= C\left(M\right)\sup_{s_{1}\in\left[t-\Delta,t\right]}\left|\int_{t-\Delta}^{s_{1}}\mathbf{E}\left[\left(a\left(X_{s_{2}}\right)-a\left(X_{t-\Delta}\right)\right)|\mathcal{F}_{t-\Delta}\right]\mathrm{d}s_{2}\left[a\left(X_{t-\Delta}\right)\right]\right|\\
    &\le C\left(M\right)\Delta\sup_{s\in\left[t-\Delta,t\right]}\left\|\mathbf{E}\left[\left(a\left(X_{s}\right)-a\left(X_{t-\Delta}\right)\right)|\mathcal{F}_{t-\Delta}\right]\right\|\left\|a\left(X_{t-\Delta}\right)\right\|\\
    &\le C\left(M\right)\Delta^{2}\left(1+\left|X_{t-\Delta}\right|\right)^{C\left(M\right)}.
\end{align*}
Hence we obtain the proof of (iii).
\end{proof}

\begin{proof}[Proof of Corollary \ref{corCorApprox1}]
By Taylor's expansion,
\begin{align*}
    &f\left(\int_{t-\Delta}^{t}\Phi_{\Delta}\left(t-s\right)X_{s}\mathrm{d}s,\xi\right)-f\left(X_{t-\Delta},\xi\right)\\
    &=\partial_{x}f\left(X_{t-\Delta},\xi\right)\left(\int_{t-\Delta}^{t}\Phi_{\Delta}\left(t-s\right)X_{s}\mathrm{d}s-X_{t-\Delta}\right)\\
    &\quad+\int_{0}^{1}\left(1-u\right)\partial_{x}^{2}f\left(X_{t-\Delta}+u\left(\int_{t-\Delta}^{t}\Phi_{\Delta}\left(t-s\right)X_{s}\mathrm{d}s-X_{t-\Delta}\right),\xi\right)\mathrm{d}u\\
    &\hspace{2cm}\left[\left(\int_{t-\Delta}^{t}\Phi_{\Delta}\left(t-s\right)X_{s}\mathrm{d}s-X_{t-\Delta}\right)^{\otimes2}\right].
\end{align*}
It is obvious that
\begin{align*}
    &\sup_{\xi\in\Xi}\left|\mathbf{E}\left[\partial_{x}f\left(X_{t-\Delta},\xi\right)\left(\int_{t-\Delta}^{t}\Phi_{\Delta}\left(t-s\right)X_{s}\mathrm{d}s-X_{t-\Delta}\right)|\mathcal{F}_{t-\Delta}\right]\right|\\
    &\le \sup_{\xi\in\Xi}\left|\partial_{x}f\left(X_{t-\Delta},\xi\right)\right|\left|\mathbf{E}\left[\int_{t-\Delta}^{t}\Phi_{\Delta}\left(t-s\right)X_{s}\mathrm{d}s-X_{t-\Delta}|\mathcal{F}_{t-\Delta}\right]\right|\\
    &\le C\left(M\right)\left(1+\left|X_{t-\Delta}\right|\right)^{C\left(M\right)}\left|\mathbf{E}\left[e_{t-\Delta,\Delta}|\mathcal{F}_{t-\Delta}\right]\right|\\
    &\le C\left(M\right)\Delta\left(1+\left|X_{t-\Delta}\right|\right)^{C\left(M\right)}
\end{align*}
by Corollary \ref{corApprox1}. We also have 
\begin{align*}
    &\sup_{\xi\in\Xi}\left|\mathbf{E}\left[\int_{0}^{1}\left(1-u\right)\partial_{x}^{2}f\left(X_{t-\Delta}+u\left(\int_{t-\Delta}^{t}\Phi_{\Delta}\left(t-s\right)X_{s}\mathrm{d}s-X_{t-\Delta}\right),\xi\right)\mathrm{d}u\right.\right.\\
    &\hspace{7cm}\left.\left.\left[\left(\int_{t-\Delta}^{t}\Phi_{\Delta}\left(t-s\right)X_{s}\mathrm{d}s-X_{t-\Delta}\right)^{\otimes2}\right]|\mathcal{F}_{t-\Delta}\right]\right|\\
    &\le \sup_{\xi\in\Xi}\mathbf{E}\left[\left\|\int_{0}^{1}\left(1-u\right)\partial_{x}^{2}f\left(X_{t-\Delta}+u\left(\int_{t-\Delta}^{t}\Phi_{\Delta}\left(t-s\right)X_{s}\mathrm{d}s-X_{t-\Delta}\right),\xi\right)\mathrm{d}u\right\|\right.\\
    &\hspace{8cm}\left.\left|\int_{t-\Delta}^{t}\Phi_{\Delta}\left(t-s\right)X_{s}\mathrm{d}s-X_{t-\Delta}\right|^{2}|\mathcal{F}_{t-\Delta}\right]\\
    &\le \mathbf{E}\left[C\left(1+\left|X_{t-\Delta}\right|+\left|\int_{t-\Delta}^{t}\Phi_{\Delta}\left(t-s\right)X_{s}\mathrm{d}s\right|\right)^{C}|\mathcal{F}_{t-\Delta}\right]^{1/2}\\
    &\quad \times \mathbf{E}\left[\left|\int_{t-\Delta}^{t}\Phi_{\Delta}\left(t-s\right)X_{s}\mathrm{d}s-X_{t-\Delta}\right|^{4}|\mathcal{F}_{t-\Delta}\right]^{1/2}\\
    &\le \mathbf{E}\left[C\left(1+\sup_{s\in\left[t-\Delta,t\right]}\left|X_{s}\right|\right)^{C}|\mathcal{F}_{t-\Delta}\right]^{1/2} C\left(M\right)\Delta\left(1+\left|X_{t-\Delta}\right|\right)^{C\left(M\right)}\\
    &\le C\left(M\right)\Delta\left(1+\left|X_{t-\Delta}\right|\right)^{C\left(M\right)}
\end{align*}
because of Corollary \ref{corApprox1} and Proposition 5.1 of \citet{Gloter-2000}. Here we obtain the first evaluation. With respect to the second one, we can have the following evaluation as above:
\begin{align*}
    &\mathbf{E}\left[\sup_{\xi\in\Xi}\left|f\left(\int_{t-\Delta}^{t}\Phi_{\Delta}\left(t-s\right)X_{s}\mathrm{d}s,\xi\right)-f\left(X_{t-\Delta},\xi\right)\right|^{m}|\mathcal{F}_{t-\Delta}\right]\\
    &\le \mathbf{E}\left[C\left(m\right)\left(1+\sup_{s\in\left[t-\Delta,t\right]}\left|X_{s}\right|\right)^{C\left(m\right)}|\mathcal{F}_{t-\Delta}\right]^{\frac{1}{2}}\\
    &\hspace{6cm}\times\mathbf{E}\left[\left|\int_{t-\Delta}^{t}\Phi_{\Delta}\left(t-s\right)X_{s}\mathrm{d}s-X_{t-\Delta}\right|^{2m}|\mathcal{F}_{t-\Delta}\right]^{\frac{1}{2}}\\
    &\le C\left(m,M\right)\Delta^{m/2}\left(1+\left|X_{t-\Delta}\right|\right)^{C\left(m,M\right)}.
\end{align*}
Hence the proof is complete.
\end{proof}

\subsection*{Proofs of the results for some laws of large numbers} 

\subsubsection*{General results}
Let $p$ denote an integer such that $\sup_{n\in\mathbf{N}}ph_{n}\le \lambda$, $\Delta_{n}:=ph_{n}$. We set the sequence of the kernels $\left\{\Phi_{\Delta_{n},n}\right\}_{n\in\mathbf{N}}$ such that $\Phi_{\Delta_{n},n}\in\mathcal{K}\left(\Delta_{n},d,M\right)$ for some $M>0$, $\int_{0}^{\Delta_{n}}\Phi_{\Delta_{n},n}\mathrm{d}s=I_{d}$ and there exist a matrix $B\in\mathbf{R}^{d}\otimes\mathbf{R}^{d}$ such that
\begin{align*}
    \left\|\int_{0}^{\Delta_{n}+h_{n}}\left(\Phi_{\Delta_{n},n}\left(\left(\Delta_{n}+h_{n}\right)-s\right)-\Phi_{\Delta_{n},n}\left(\Delta_{n}-s\right)\right)s\mathrm{d}s - h_{n}B\right\|\le Ch_{n}^{2}\left(1+\left|x\right|\right)^{2}, 
\end{align*}
a set $\mathbb{L}\subset\left\{0,\ldots,p\right\}$ such that there exist functions $D_{\ell}:\mathbf{R}^{d}\to\mathbf{R}^{d}\otimes\mathbf{R}^{d}$ for $\ell\in\mathbb{L}$ such that
\begin{align*}
    &\left\|\mathbf{E}\left[\left(\int_{0}^{\Delta_{n}+\left(1+\ell\right)h_{n}}\Phi_{\Delta_{n},n}\left(\Delta_{n}-s_{1}\right)\left(\int_{0}^{s_{1}}a\left(x\right)\mathrm{d}w_{s_{2}}\right)\mathrm{d}s_{1}\right)\right.\right.\\
    &\hspace{1cm}\left(\int_{0}^{\Delta_{n}+\left(1+\ell\right)h_{n}}\left(\Phi_{\Delta_{n},n}\left(\left(\Delta_{n}+\left(1+\ell\right)h_{n}\right)-s_{1}\right)-\Phi_{\Delta_{n},n}\left(\Delta_{n}+\ell h_{n}-s_{1}\right)\right)\right.\\
    &\hspace{8cm}\left.\left.\left.\times\left(\int_{0}^{s_{1}}a\left(x\right)\mathrm{d}w_{s_{2}}\right)\mathrm{d}s_{1}\right)^{T}\right]- h_{n}D_{\ell}\left(x\right)\right\|\\
    &\le Ch_{n}^{2}\left(1+\left|x\right|\right)^{C},
\end{align*}
a function $G:\mathbf{R}^{d}\to\mathbf{R}^{d}\otimes\mathbf{R}^{d}$ such that
\begin{align*}
    &\left\|\mathbf{E}\left[\left(\int_{0}^{\Delta_{n}+h_{n}}\left(\Phi_{\Delta_{n},n}\left(\left(\Delta_{n}+h_{n}\right)-s_{1}\right)-\Phi_{\Delta_{n},n}\left(\Delta_{n}-s_{1}\right)\right)\left(\int_{0}^{s_{1}}a\left(x\right)\mathrm{d}w_{s_{2}}\right)\mathrm{d}s_{1}\right)\right.\right.\\
    &\quad\left.\left(\int_{0}^{\Delta_{n}+h_{n}}\left(\Phi_{\Delta_{n},n}\left(\left(\Delta_{n}+h_{n}\right)-s_{1}\right)-\Phi_{\Delta_{n},n}\left(\Delta_{n}-s_{1}\right)\right)\left(\int_{0}^{s_{1}}a\left(x\right)\mathrm{d}w_{s_{2}}\right)\mathrm{d}s_{1}\right)^{T}\right]\\
    &\quad\left.- h_{n}G\left(x\right)\right\|\\
    &\le Ch_{n}^{2}\left(1+\left|x\right|\right)^{C}.
\end{align*}
We define
\begin{align*}
    \overline{X}_{t,n}=\int_{t-\Delta_{n}}^{t}\Phi_{\Delta_{n},n}\left(t-s\right)X_{s}\mathrm{d}s,
\end{align*}
and the following random quantities such that
\begin{align*}
    \overline{\nu}_{n}\left(f\left(\cdot,\xi\right)\right)&:=\frac{1}{n}\sum_{i=1}^{n}f\left(\overline{X}_{ih_{n},n},\xi\right),\\
    \overline{I}_{\ell,n}\left(v\left(\cdot,\xi\right)\right)&:=\frac{1}{nh_{n}}\sum_{i=1+\ell}^{n}v\left(\overline{X}_{\left(i-1-\ell\right)h_{n},n},\xi\right)\left[\overline{X}_{ih_{n},n}-\overline{X}_{\left(i-1\right)h_{n},n}-\left(h_{n}B\right)b\left(\overline{X}_{\left(i-1-\ell\right)h_{n},n}\right)\right],\\
    \overline{Q}_{n}\left(M\left(\cdot,\xi\right)\right)&:=\frac{1}{nh_{n}}\sum_{i=1}^{n}M\left(\overline{X}_{\left(i-1\right)h_{n},n},\xi\right)\left[\left(\overline{X}_{ih_{n},n}-\overline{X}_{\left(i-1\right)h_{n},n}\right)^{\otimes2}\right],
\end{align*}
where $f:\mathbf{R}^{d}\times\Xi\to\mathbf{R}$, $v:\mathbf{R}^{d}\times\Xi\to\mathbf{R}^{d}$, $M:\mathbf{R}^{d}\times\Xi\to\mathbf{R}^{d}\otimes\mathbf{R}^{d}$ are in $\mathcal{C}^{2}$-class, and their first and second derivatives and themselves are at most polynomial growth uniformly in $\xi\in\Xi$.

\begin{proposition}\label{propEmpMean}
Under [A1], $\overline{\nu}_{n}\left(f\left(\cdot,\xi\right)\right)\to^{P}\nu_{0}\left(f\left(\cdot,\xi\right)\right)\text{ uniformly in }\xi\in\Xi.$
\end{proposition}

\begin{proposition}\label{propEmpI}
If $\ell\in\mathbb{L}$ and [A1] hold, $\overline{I}_{\ell,n}\left(v\left(\cdot,\xi\right)\right)\to^{P}\nu_{0}\left(\partial_{x}v\left[D_{\ell}^{T}\right]\left(\cdot,\xi\right)\right)\text{ uniformly in }\xi\in\Xi.$
\end{proposition}

\begin{remark}
    When $\Phi_{\Delta_{n},n}\left(s\right)=\frac{1}{h_{n}}\mathbf{1}_{\left[0,h_{n}\right]}\left(s\right)I_{d}$, $p=1$, $M=1$, then
    \begin{align*}
        \int_{0}^{2h_{n}}\left(\Phi_{\Delta_{n},n}\left(2h_{n}-s\right)-\Phi_{\Delta_{n},n}\left(h_{n}-s\right)\right)s\mathrm{d}s&=\frac{1}{h_{n}}I_{d}\int_{h_{n}}^{2h_{n}}s\mathrm{d}s-\frac{1}{h_{n}}I_{d}\int_{0}^{h_{n}}s\mathrm{d}s\\
        &=\frac{1}{h_{n}}I_{d}\left[\frac{\left(2h_{n}\right)^{2}-2\left(h_{n}\right)^{2}}{2}\right]\\
        &=h_{n}I_{d}
    \end{align*}
    and hence we obtain $B=I_{d}$; we also can evaluate $D_{0}\left(x\right)=\frac{1}{6}A\left(x\right)$, which coincides with that of \citet{Gloter-2006}.
\end{remark}

\begin{proposition}\label{propEmpQ}Under [A1], $
    \overline{Q}_{n}\left(M\left(\cdot,\xi\right)\right)\to^{P}\nu_{0}\left(M\left[G\right]\left(\cdot,\xi\right)\right)\text{ uniformly in }\xi\in\Xi.$
\end{proposition}

\begin{remark}
    As the previous remark, we can obtain $G\left(x\right)=\frac{2}{3}A\left(x\right)$ as shown in \citet{Gloter-2006}.
\end{remark}

\begin{proof}[Proof of Proposition \ref{propEmpMean}]
It is obvious by Corollary \ref{corCorApprox1} and the assumption for $f$.
\end{proof}

\begin{proof}[Proof of Proposition \ref{propEmpI}]
We decompose the summation as follows:
\begin{align*}
    &\overline{I}_{\ell,n}\left(v\left(\cdot,\xi\right)\right)\\
    &=\frac{1}{nh_{n}}\sum_{i=1+\ell}^{n}v\left(\overline{X}_{\left(i-1-\ell\right)h_{n},n},\xi\right)\left[\overline{X}_{ih_{n},n}-\overline{X}_{\left(i-1\right)h_{n},n}-\left(h_{n}B\right)b\left(\overline{X}_{\left(i-1-\ell\right)h_{n},n}\right)\right]\\
    &=\frac{1}{nh_{n}}\sum_{i=1+\ell}^{n}v\left(\overline{X}_{\left(i-1-\ell\right)h_{n},n},\xi\right)\left[\overline{X}_{ih_{n},n}-\overline{X}_{\left(i-1\right)h_{n},n}-\left(h_{n}B\right)b\left(X_{\left(i-2p-1\right)h_{n}}\right)\right]\\
    &\quad+\frac{1}{nh_{n}}\sum_{i=1+\ell}^{n}v\left(\overline{X}_{\left(i-1-\ell\right)h_{n},n},\xi\right)\left[\left(h_{n}B\right)b\left(X_{\left(i-2p-1\right)h_{n}}\right)-\left(h_{n}B\right)b\left(\overline{X}_{\left(i-1-\ell\right)h_{n},n}\right)\right]\\
    &=\frac{1}{nh_{n}}\sum_{i=1+\ell}^{n}v\left(X_{\left(i-2p-1\right)h_{n}},\xi\right)\left[\overline{X}_{ih_{n},n}-\overline{X}_{\left(i-1\right)h_{n},n}-\left(h_{n}B\right)b\left(X_{\left(i-2p-1\right)h_{n}}\right)\right]\\
    &\quad+\frac{1}{nh_{n}}\sum_{i=1+\ell}^{n}\partial_{x}v\left(X_{\left(i-2p-1\right)h_{n}},\xi\right)\left[\left(\overline{X}_{ih_{n},n}-\overline{X}_{\left(i-1\right)h_{n},n}-\left(h_{n}B\right)b\left(X_{\left(i-2p-1\right)h_{n}}\right)\right)\right.\\
    &\hspace{6cm}\left.\left(\overline{X}_{\left(i-1-\ell\right)h_{n},n}-X_{\left(i-2p-1\right)h_{n}}\right)^{T}\right]\\
    &\quad+\frac{1}{nh_{n}}\sum_{i=1+\ell}^{n}\sum_{j_{1}=1}^{d}\sum_{j_{2}=1}^{d}\\
    &\hspace{2cm}\int_{0}^{1}\left(1-s\right)\partial_{x^{\left(j_{1}\right)}}\partial_{x^{\left(j_{2}\right)}}v\left(X_{\left(i-p-1\right)h_{n}}+s\left(\overline{X}_{\left(i-1\right)h_{n},n}-X_{\left(i-p-1\right)h_{n}}\right),\xi\right)\mathrm{d}s\\
    &\hspace{2cm}\left(\overline{X}_{\left(i-1-\ell\right)h_{n},n}-X_{\left(i-2p-1\right)h_{n}}\right)^{\left(j_{1}\right)}\left(\overline{X}_{\left(i-1-\ell\right)h_{n},n}-X_{\left(i-2p-1\right)h_{n}}\right)^{\left(j_{2}\right)}\\
    &\hspace{2cm}\left[\overline{X}_{ih_{n},n}-\overline{X}_{\left(i-1\right)h_{n},n}-\left(h_{n}B\right)b\left(X_{\left(i-2p-1\right)h_{n}}\right)\right]\\
    &\quad+\frac{1}{nh_{n}}\sum_{i=1+\ell}^{n}v\left(\overline{X}_{\left(i-1-\ell\right)h_{n},n},\xi\right)\left[\left(h_{n}B\right)b\left(X_{\left(i-2p-1\right)h_{n}}\right)-\left(h_{n}B\right)b\left(\overline{X}_{\left(i-1-\ell\right)h_{n},n}\right)\right].
\end{align*}
Because of the evaluation such that
\begin{align*}
    &\frac{1}{nh_{n}}\sum_{1+\ell\le \left(2p+1\right)i\le n}\left|\mathbf{E}\left[\left.v\left(X_{\left(2p+1\right)\left(i-1\right)h_{n}},\xi\right)\right.\right.\right.\\
    &\hspace{2.5cm}\left.\left.\left.\left[\overline{X}_{\left(2p+1\right)ih_{n},n}-\overline{X}_{\left(\left(2p+1\right)i-1\right)h_{n},n}-\left(h_{n}B\right)b\left(X_{\left(2p+1\right)\left(i-1\right)h_{n}}\right)\right]\right|\mathcal{F}_{\left(2p+1\right)\left(i-1\right)h_{n}}\right]\right|\\
    &\to^{P}0,\\
    &\frac{1}{n^{2}h_{n}^{2}}\sum_{1+\ell\le \left(2p+1\right)i\le n}\mathbf{E}\left[\left|\left.v\left(X_{\left(2p+1\right)\left(i-1\right)h_{n}},\xi\right)\right.\right.\right.\\
    &\hspace{2.5cm}\left.\left.\left.\left[\overline{X}_{\left(2p+1\right)ih_{n},n}-\overline{X}_{\left(\left(2p+1\right)i-1\right)h_{n},n}-\left(h_{n}B\right)b\left(X_{\left(2p+1\right)\left(i-1\right)h_{n}}\right)\right]\right|^{2}\right|\mathcal{F}_{\left(2p+1\right)\left(i-1\right)h_{n}}\right]\\
    &\to^{P}0
\end{align*}
for all $\xi\in\Xi$ and Lemma 9 of \citet{Genon-Catalot-Jacod-1993}, we have
\begin{align*}
    \frac{1}{nh_{n}}\sum_{i=1+\ell}^{n}v\left(X_{\left(i-2p-1\right)h_{n}},\xi\right)\left[\overline{X}_{ih_{n},n}-\overline{X}_{\left(i-1\right)h_{n},n}-\left(h_{n}B\right)b\left(X_{\left(i-2p-1\right)h_{n}}\right)\right]\to^{P}0
\end{align*}
for all $\xi\in\Xi$. To verify the uniform convergence in probability of this summation, we show that the following inequalities
hold \citep{Ibragimov-Hasminskii-1981}: there exist $C>0$ and $k>\mathrm{dim}\Xi$ such that for all $n\in\mathbf{N}$ and $\xi,\xi'\in\Xi$,
\begin{align*}
    &\mathbf{E}\left[\left|\frac{1}{nh_{n}}\sum_{i=1+\ell}^{n}v\left(X_{\left(i-2p-1\right)h_{n}},\xi\right)\left[\overline{X}_{ih_{n},n}-\overline{X}_{\left(i-1\right)h_{n},n}-\left(h_{n}B\right)b\left(X_{\left(i-2p-1\right)h_{n}}\right)\right]\right|^{k}\right]\le C,\\
    &\mathbf{E}\left[\left|\frac{1}{nh_{n}}\sum_{i=1+\ell}^{n}v\left(X_{\left(i-2p-1\right)h_{n}},\xi\right)\left[\overline{X}_{ih_{n},n}-\overline{X}_{\left(i-1\right)h_{n},n}-\left(h_{n}B\right)b\left(X_{\left(i-2p-1\right)h_{n}}\right)\right]\right.\right.\\
    &\qquad\left.\left.-\frac{1}{nh_{n}}\sum_{i=1+\ell}^{n}v\left(X_{\left(i-2p-1\right)h_{n}},\xi'\right)\left[\overline{X}_{ih_{n},n}-\overline{X}_{\left(i-1\right)h_{n},n}-\left(h_{n}B\right)b\left(X_{\left(i-2p-1\right)h_{n}}\right)\right]\right|^{k}\right]\\
    &\le C\left|\xi-\xi'\right|^{k}.
\end{align*}
These evaluations can be led by the assumption of $\Phi_{\Delta}$ and Burkholder's inequality in a similar way to \citet{Nakakita-Uchida-2019a}.

With respect to the second summation, we can easily have the evaluation such that
\begin{align*}
    &\frac{1}{nh_{n}}\sum_{1+\ell\le \left(2p+1\right)i\le n}\partial_{x}v\left(X_{\left(2p+1\right)\left(i-1\right)h_{n}},\xi\right)\\
    &\hspace{3cm}\left[\left(\overline{X}_{\left(2p+1\right)ih_{n},n}-\overline{X}_{\left(\left(1p+1\right)i-1\right)h_{n},n}-\left(h_{n}B\right)b\left(X_{\left(2p+1\right)\left(i-1\right)h_{n}}\right)\right)\right.\\
    &\hspace{3cm}\left(\overline{X}_{\left(\left(2p+1\right)i-1-\ell\right)h_{n},n}-X_{\left(2p+1\right)\left(i-1\right)h_{n}}\right.\\
    &\hspace{4cm}\left.\left.-\left(\int_{0}^{\Delta_{n}}\Phi_{\Delta_{n},n}\left(\Delta_{n}-s\right)s\mathrm{d}s\right)b\left(X_{\left(p+1\right)\left(i-1\right)}\right)\right)^{T}\right]\\
    &+\frac{1}{nh_{n}}\sum_{1+\ell\le \left(2p+1\right)i\le n}\partial_{x}v\left(X_{\left(2p+1\right)\left(i-1\right)h_{n}},\xi\right)\\
    &\hspace{3cm}\left[\left(\overline{X}_{\left(2p+1\right)ih_{n},n}-\overline{X}_{\left(\left(2p+1\right)i-1\right)h_{n},n}-\left(h_{n}B\right)b\left(X_{\left(2p+1\right)\left(i-1\right)h_{n}}\right)\right)\right.\\
    &\hspace{3cm}\left.\left(\left(\int_{0}^{\Delta_{n}}\Phi_{\Delta_{n},n}\left(\Delta_{n}-s\right)s\mathrm{d}s\right)b\left(X_{\left(2p+1\right)\left(i-1\right)}\right)\right)^{T}\right]\\
    &\to ^{P} \frac{1}{2p+1}\nu_{0}\left(\partial_{x}v\left[D_{\ell}^{T}\right]\left(\cdot,\xi\right)\right)\text{ uniformly in }\xi\in\Xi
\end{align*}
by an analogous manner to \citet{Gloter-2006}. Hence we obtain
\begin{align*}
    &\frac{1}{nh_{n}}\sum_{i=1+\ell}^{n}\partial_{x}v\left(X_{\left(i-2p-1\right)h_{n}},\xi\right)\left[\left(\overline{X}_{ih_{n},n}-\overline{X}_{\left(i-1\right)h_{n},n}-\left(h_{n}B\right)b\left(X_{\left(i-2p-1\right)h_{n}}\right)\right)\right.\\
    &\hspace{6cm}\left.\left(\overline{X}_{\left(i-1-\ell\right)h_{n},n}-X_{\left(i-2p-1\right)h_{n}}\right)^{T}\right]\\
    &\to \nu_{0}\left(\partial_{x}v\left[D_{\ell}^{T}\right]\left(\cdot,\xi\right)\right)\text{ uniformly in }\xi\in\Xi.
\end{align*}
For the residual terms, it is obvious that they converge to zero in probability uniformly in $\xi\in\Xi$. Hence we complete the proof.
\end{proof}

\begin{proof}[Proof of Proposition \ref{propEmpQ}]
Because of the fact
\begin{align*}
    &\frac{1}{nh_{n}}\sum_{i=1}^{n}M\left(\overline{X}_{\left(i-1\right)h_{n},n},\xi\right)\left[\left(\overline{X}_{ih_{n},n}-\overline{X}_{\left(i-1\right)h_{n},n}\right)^{\otimes2}\right]\\
    &\quad-\frac{1}{nh_{n}}\sum_{i=1}^{n}M\left(X_{\left(i-p-1\right)h_{n}},\xi\right)\left[\left(\overline{X}_{ih_{n},n}-\overline{X}_{\left(i-1\right)h_{n},n}\right)^{\otimes2}\right]\\
    &\to^{P}0\text{ uniformly in }\xi\in\Xi
\end{align*}
which can be easily obtained, it is sufficient to evaluate
\begin{align*}
     \overline{Q}_{n}'\left(M\left(\cdot,\xi\right)\right)=\frac{1}{nh_{n}}\sum_{i=1}^{n}M\left(X_{\left(i-p-1\right)h_{n}},\xi\right)\left[\left(\overline{X}_{ih_{n},n}-\overline{X}_{\left(i-1\right)h_{n},n}\right)^{\otimes2}\right],
\end{align*}
and we can have an analogous result to \citet{Gloter-2006} such that
\begin{align*}
    \overline{Q}_{n}'\left(M\left(\cdot,\xi\right)\right)\to^{P}\nu_{0}\left(M\left[G\right]\left(\cdot,\xi\right)\right)\text{ uniformly in }\xi\in\Xi
\end{align*}
and hence obtain the proof.
\end{proof}

\subsubsection*{Some specific evaluation} 
We set $p=\left[\overline{\rho}\right]+1$, $\Delta_{n}=ph_{n}$ and show the evaluation of $B$, $D_{\ell}$ and $G$ when setting our kernel $\left\{\Phi_{\Delta,n}\right\}=\left\{V_{\rho, h_{n}}\right\}$ as follows: we have $\Delta_{n}=ph_{n}$, $B=I_{d}$, $D_{0}\left(x\right)=\left.\mathbb{D}_{0}\left(x|\rho\right)\right|_{\rho=\rho_{\star}}$, where $\mathbb{D}_{0}^{\left(i,j\right)}\left(x|\rho\right)=A^{\left(i,j\right)}\left(x\right)f_{\mathbb{D}_{0}}\left(\rho^{\left(i\right)},\rho^{\left(j\right)}\right)$,
\begin{align*}
    f_{\mathbb{D}_{0}}\left(\rho^{\left(i\right)},\rho^{\left(j\right)}\right)
    &:=\begin{cases}
    0 & \text{ if }\rho^{\left(j\right)}=0,\\
    \frac{\rho^{\left(j\right)}}{2} & \text{ if }\rho^{\left(i\right)}=0, \rho^{\left(j\right)}\in\left(0,1\right],\\
    \frac{2\rho^{\left(j\right)}-1}{2\rho^{\left(j\right)}}&\text{ if }\rho^{\left(i\right)}=0, \rho^{\left(j\right)}\in\left(1,\overline{\rho}\right],\\
    \frac{6\rho^{\left(i\right)}\rho^{\left(j\right)}-3\left(\rho^{\left(i\right)}\right)^{2}-3\rho^{\left(i\right)}}{6\rho^{\left(i\right)}\rho^{\left(j\right)}}&\text{ if }\rho^{\left(i\right)}>0,\rho^{\left(i\right)}+1<\rho^{\left(j\right)},\\
    \frac{\left(\rho^{\left(i\right)}-\rho^{\left(j\right)}\right)^{3}+3\left(\rho^{\left(j\right)}\right)^{2}-3\rho^{\left(j\right)}+1}{6\rho^{\left(i\right)}\rho^{\left(j\right)}}&\text{ if }\rho^{\left(j\right)}>1,\rho^{\left(i\right)}<\rho^{\left(j\right)}\le \rho^{\left(i\right)}+1,\\
    \frac{3\left(\rho^{\left(j\right)}\right)^{2}-3\rho^{\left(j\right)}+1}{6\rho^{\left(i\right)}\rho^{\left(j\right)}}&\text{ if }\rho^{\left(j\right)}>1,\rho^{\left(i\right)}\ge \rho^{\left(j\right)},\\
    \frac{\left(\rho^{\left(i\right)}-\rho^{\left(j\right)}\right)^{3}+\left(\rho^{\left(j\right)}\right)^{3}}{6\rho^{\left(i\right)}\rho^{\left(j\right)}}&\text{ if }\rho^{\left(j\right)}\in\left(0,1\right], \rho^{\left(i\right)}\in\left(0,\rho^{\left(j\right)}\right),\\
    \frac{\left(\rho^{\left(j\right)}\right)^{3}}{6\rho^{\left(i\right)}\rho^{\left(j\right)}}&\text{ if }\rho^{\left(j\right)}\in\left(0,1\right], \rho^{\left(i\right)}\ge\rho^{\left(j\right)},
    \end{cases}
\end{align*}
$D_{\ell}=O$ for $\ell\ge \left[\max_{i=1,\ldots,d}\rho_{\ast}^{\left(d\right)}\right]+1$ because of independent increments of the Wiener process,
and $G\left(x\right)=\left.\mathbb{G}\left(x|\rho\right)\right|_{\rho=\rho_{\star}}$ where $\mathbb{G}\left(x|\rho\right)=\left.\mathbb{G}\left(x,\alpha|\rho\right)\right|_{\alpha=\alpha_{\star}}$.
\begin{remark}
    Note that $\mathbb{D}_{0}\left(x|\rho\right)$ and $\mathbb{G}\left(x|\rho\right)$ is continuous w.r.t. $\rho$ for all fixed $x$ by Lemma \ref{LemmaFunctionD} and Lemma \ref{LemmaFunctionG} in Appendix B.
\end{remark}
For all $i=1,\ldots,d$, if $\rho^{\left(i\right)}=0$, then
\begin{align*}
     &\left[\int_{0}^{\Delta_{n}+h_{n}}\left(V_{\rho,h_{n}}\left(\left(\Delta_{n}+h_{n}\right)-s\right)-V_{\rho,h_{n}}\left(\Delta_{n}-s\right)\right)s\mathrm{d}s\right]^{\left(i,i\right)}\\
     &=
     \int_{0}^{\left(p+1\right)h_{n}}\left(\delta\left(\left(p+1\right)h_{n}-s\right)-\delta\left(ph_{n}-s\right)\right)s\mathrm{d}s\\
     &=\left(p+1\right)h_{n}-ph_{n}\\
     &=h_{n},
\end{align*}
and if $\rho^{\left(i\right)}\in\left(0,1\right]$, then
\begin{align*}
    &\left[\int_{0}^{\Delta_{n}+h_{n}}\left(V_{\rho,h_{n}}\left(\left(\Delta_{n}+h_{n}\right)-s\right)-V_{\rho,h_{n}}\left(\Delta_{n}-s\right)\right)s\mathrm{d}s\right]^{\left(i,i\right)}\\
    &=\int_{0}^{\left(p+1\right)h_{n}}\left(\rho^{\left(i\right)}h_{n}\right)^{-1}\left(\mathbf{1}_{\left[0,\rho^{\left(i\right)}h_{n}\right]}\left(\left(p+1\right)h_{n}-s\right)-\mathbf{1}_{\left[0,\rho^{\left(i\right)}h_{n}\right]}\left(ph_{n}-s\right)\right)s\mathrm{d}s\\
    &=\int_{\left(p+1-\rho^{\left(i\right)}\right)h_{n}}^{\left(p+1\right)h_{n}}\left(\rho^{\left(i\right)}h_{n}\right)^{-1}s\mathrm{d}s-\int_{\left(p-\rho^{\left(i\right)}\right)h_{n}}^{ph_{n}}\left(\rho^{\left(i\right)}h_{n}\right)^{-1}s\mathrm{d}s\\
    &=\frac{h_{n}}{2\rho^{\left(i\right)}}\left[\left(p+1\right)^{2}-\left(p+1-\rho^{\left(i\right)}\right)^{2}-p^{2}+\left(p-\rho^{\left(i\right)}\right)^{2}\right]\\
    &=h_{n},
\end{align*}
and if $\rho^{\left(i\right)}\in\left(1,\overline{\rho}\right]$, then
\begin{align*}
    &\left[\int_{0}^{\Delta_{n}+h_{n}}\left(V_{\rho,h_{n}}\left(\left(\Delta_{n}+h_{n}\right)-s\right)-V_{\rho,h_{n}}\left(\Delta_{n}-s\right)\right)s\mathrm{d}s\right]^{\left(i,i\right)}\\
    &=\int_{\left(p+1-\rho^{\left(i\right)}\right)h_{n}}^{\left(p+1\right)h_{n}}\left(\rho^{\left(i\right)}h_{n}\right)^{-1}s\mathrm{d}s-\int_{\left(p-\rho^{\left(i\right)}\right)h_{n}}^{ph_{n}}\left(\rho^{\left(i\right)}h_{n}\right)^{-1}s\mathrm{d}s\\
    &=\int_{ph_{n}}^{\left(p+1\right)h_{n}}\left(\rho^{\left(i\right)}h_{n}\right)^{-1}s\mathrm{d}s-\int_{\left(p-\rho^{\left(i\right)}\right)h_{n}}^{\left(p+1-\rho^{\left(i\right)}\right)h_{n}}\left(\rho^{\left(i\right)}h_{n}\right)^{-1}s\mathrm{d}s\\
    &=\frac{h_{n}}{2\rho^{\left(i\right)}}\left[\left(p+1\right)^{2}-p^{2}-\left(p+1-\rho^{\left(i\right)}\right)^{2}+\left(p-\rho^{\left(i\right)}\right)^{2}\right]\\
    &=h_{n}.
\end{align*}
It is obvious that if $i,j=1,\ldots d$ and $i\neq j$,
\begin{align*}
    \left[\int_{0}^{\Delta_{n}+h_{n}}\left(V_{\rho,h_{n}}\left(\left(\Delta_{n}+h_{n}\right)-s\right)-V_{\rho,h_{n}}\left(\Delta_{n}-s\right)\right)s\mathrm{d}s\right]^{\left(i,j\right)}=0;
\end{align*}
therefore it holds that $B=I_{d}$. Regarding to $D_{0}\left(x\right)$, for all $i,j=1,\ldots,d$, let us define
\begin{align*}
    &\mathbf{D}_{0}^{\left(i,j\right)}\left(x|\rho\right)\\
    &:={\frac{1}{h_{n}}}\mathbf{E}\left[\left(\int_{0}^{\left(p+1\right)h_{n}}V_{\rho,h_{n}}\left(ph_{n}-s_{1}\right)\left(\int_{0}^{s_{1}}a\left(x\right)\mathrm{d}w_{s_{2}}\right)\mathrm{d}s_{1}\right)\right.\\
    &\quad\left.\left(\int_{0}^{\left(p+1\right)h_{n}}\left(V_{\rho,h_{n}}\left(\left(p+1\right)h_{n}-s_{1}\right)-V_{\rho,h_{n}}\left(ph_{n}-s_{1}\right)\right)\left(\int_{0}^{s_{1}}a\left(x\right)\mathrm{d}w_{s_{2}}\right)\mathrm{d}s_{1}\right)^{T}\right]^{\left(i,j\right)}\\
    &={\frac{1}{h_{n}}}\mathbf{E}\left[\left(\int_{0}^{\left(p+1\right)h_{n}}V_{\rho,h_{n}}\left(ph_{n}-s_{1}\right)\left(\int_{0}^{s_{1}}a\left(x\right)\mathrm{d}w_{s_{2}}\right)\mathrm{d}s_{1}\right)^{\left(i\right)}\right.\\
    &\qquad\left.\left(\int_{0}^{\left(p+1\right)h_{n}}\left(V_{\rho,h_{n}}\left(\left(p+1\right)h_{n}-s_{1}\right)-V_{\rho,h_{n}}\left(ph_{n}-s_{1}\right)\right)\left(\int_{0}^{s_{1}}a\left(x\right)\mathrm{d}w_{s_{2}}\right)\mathrm{d}s_{1}\right)^{\left(j\right)}\right]\\
    &={\frac{1}{h_{n}}}\mathbf{E}\left[\left(\int_{0}^{\left(p+1\right)h_{n}}V_{\rho,h_{n}}^{\left(i,i\right)}\left(ph_{n}-s\right)\left(a\left(x\right)w_{s}\right)^{\left(i\right)}\mathrm{d}s\right)\right.\\
    &\qquad\left.\left(\int_{0}^{\left(p+1\right)h_{n}}\left(V_{\rho,h_{n}}^{\left(j,j\right)}\left(\left(p+1\right)h_{n}-s_{1}'\right)-V_{\rho,h_{n}}^{\left(j,j\right)}\left(ph_{n}-s'\right)\right)\left(a\left(x\right)w_{s'}\right)^{\left(j\right)}\mathrm{d}s'\right)\right]\\
    &={\frac{1}{h_{n}}}\int_{0}^{\left(p+1\right)h_{n}}\int_{0}^{\left(p+1\right)h_{n}}\mathbf{E}\left[\left(a\left(x\right)w_{s}\right)^{\left(i\right)}\left(a\left(x\right)w_{s'}\right)^{\left(j\right)}\right]\\
    &\qquad\left(V_{\rho,h_{n}}^{\left(i,i\right)}\left(ph_{n}-s\right)\left(V_{\rho,h_{n}}^{\left(j,j\right)}\left(\left(p+1\right)h_{n}-s'\right)-V_{\rho,h_{n}}^{\left(j,j\right)}\left(ph_{n}-s'\right)\right)\right)\mathrm{d}s'\mathrm{d}s\\
    &={\frac{1}{h_{n}}}\int_{0}^{\left(p+1\right)h_{n}}\int_{0}^{\left(p+1\right)h_{n}}A^{\left(i,j\right)}\left(x\right)\min\left\{s,s'\right\}\\
    &\qquad\left(V_{\rho,h_{n}}^{\left(i,i\right)}\left(ph_{n}-s\right)\left(V_{\rho,h_{n}}^{\left(j,j\right)}\left(\left(p+1\right)h_{n}-s'\right)-V_{\rho,h_{n}}^{\left(j,j\right)}\left(ph_{n}-s'\right)\right)\right)\mathrm{d}s'\mathrm{d}s,
\end{align*}
and if $\rho^{\left(i\right)}=\rho^{\left(j\right)}=0$,
\begin{align*}
    {h_{n}}\mathbf{D}_{0}^{\left(i,j\right)}\left(x|\rho\right) &= 
    \int_{0}^{\left(p+1\right)h_{n}}\int_{0}^{\left(p+1\right)h_{n}}A^{\left(i,j\right)}\left(x\right)\min\left\{s,s'\right\}\\
    &\qquad\left(V_{\rho,h_{n}}^{\left(i,i\right)}\left(ph_{n}-s\right)\left(V_{\rho,h_{n}}^{\left(j,j\right)}\left(\left(p+1\right)h_{n}-s'\right)-V_{\rho,h_{n}}^{\left(j,j\right)}\left(ph_{n}-s'\right)\right)\right)\mathrm{d}s'\mathrm{d}s\\
    &=0,
\end{align*}
and if $\rho^{\left(i\right)}=0$, $\rho^{\left(j\right)}\in\left(0,1\right]$,
\begin{align*}
    &{h_{n}}\mathbf{D}_{0}^{\left(i,j\right)}\left(x|\rho\right)\\
    &=\int_{0}^{\left(p+1\right)h_{n}}\int_{0}^{\left(p+1\right)h_{n}}A^{\left(i,j\right)}\left(x\right)\min\left\{s,s'\right\}\\
    &\qquad\left(V_{\rho,h_{n}}^{\left(i,i\right)}\left(ph_{n}-s\right)\left(V_{\rho,h_{n}}^{\left(j,j\right)}\left(\left(p+1\right)h_{n}-s'\right)-V_{\rho,h_{n}}^{\left(j,j\right)}\left(ph_{n}-s'\right)\right)\right)\mathrm{d}s'\mathrm{d}s\\
    &=\int_{0}^{\left(p+1\right)h_{n}}A^{\left(i,j\right)}\left(x\right)\min\left\{ph_{n},s'\right\}\\
    &\hspace{3cm}\left(\rho^{\left(j\right)}h_{n}\right)^{-1}\left(\mathbf{1}_{\left[0,\rho^{\left(j\right)}h_{n}\right]}\left(\left(p+1\right)h_{n}-s'\right)-\mathbf{1}_{\left[0,\rho^{\left(j\right)}h_{n}\right]}\left(ph_{n}-s'\right)\right)\mathrm{d}s'\\
    &=\int_{ph_{n}}^{\left(p+1\right)h_{n}}A^{\left(i,j\right)}\left(x\right)ph_{n}\\
    &\hspace{3cm}\left(\rho^{\left(j\right)}h_{n}\right)^{-1}\left(\mathbf{1}_{\left[0,\rho^{\left(j\right)}h_{n}\right]}\left(\left(p+1\right)h_{n}-s'\right)-\mathbf{1}_{\left[0,\rho^{\left(j\right)}h_{n}\right]}\left(ph_{n}-s'\right)\right)\mathrm{d}s'\\
    &\quad+\int_{0}^{ph_{n}}A^{\left(i,j\right)}\left(x\right)s'\\
    &\hspace{3cm}\left(\rho^{\left(j\right)}h_{n}\right)^{-1}\left(\mathbf{1}_{\left[0,\rho^{\left(j\right)}h_{n}\right]}\left(\left(p+1\right)h_{n}-s'\right)-\mathbf{1}_{\left[0,\rho^{\left(j\right)}h_{n}\right]}\left(ph_{n}-s'\right)\right)\mathrm{d}s'\\
    &=A^{\left(i,j\right)}\left(x\right)\left(ph_{n}-\frac{p^{2}h_{n}^{2}-\left(p-\rho^{\left(j\right)}\right)^{2}h_{n}^{2}}{2\rho^{\left(j\right)}h_{n}}\right)\\
    &=\frac{\rho^{\left(j\right)}h_{n}A^{\left(i,j\right)}\left(x\right)}{2},
\end{align*}
and if $\rho^{\left(i\right)}=0$, $\rho^{\left(j\right)}\in\left(1,\overline{\rho}\right]$,
\begin{align*}
    &{h_{n}}\mathbf{D}_{0}^{\left(i,j\right)}\left(x|\rho\right)\\
    &=
    \int_{0}^{\left(p+1\right)h_{n}}\int_{0}^{\left(p+1\right)h_{n}}A^{\left(i,j\right)}\left(x\right)\min\left\{s,s'\right\}\\
    &\qquad\left(V_{\rho,h_{n}}^{\left(i,i\right)}\left(ph_{n}-s\right)\left(V_{\rho,h_{n}}^{\left(j,j\right)}\left(\left(p+1\right)h_{n}-s'\right)-V_{\rho,h_{n}}^{\left(j,j\right)}\left(ph_{n}-s'\right)\right)\right)\mathrm{d}s'\mathrm{d}s\\
    &=\int_{ph_{n}}^{\left(p+1\right)h_{n}}A^{\left(i,j\right)}\left(x\right)ph_{n}\\
    &\hspace{3cm}\left(\rho^{\left(j\right)}h_{n}\right)^{-1}\left(\mathbf{1}_{\left[0,\rho^{\left(j\right)}h_{n}\right]}\left(\left(p+1\right)h_{n}-s'\right)-\mathbf{1}_{\left[0,\rho^{\left(j\right)}h_{n}\right]}\left(ph_{n}-s'\right)\right)\mathrm{d}s'\\
    &\quad+\int_{0}^{ph_{n}}A^{\left(i,j\right)}\left(x\right)s'\\
    &\hspace{3cm}\left(\rho^{\left(j\right)}h_{n}\right)^{-1}\left(\mathbf{1}_{\left[0,\rho^{\left(j\right)}h_{n}\right]}\left(\left(p+1\right)h_{n}-s'\right)-\mathbf{1}_{\left[0,\rho^{\left(j\right)}h_{n}\right]}\left(ph_{n}-s'\right)\right)\mathrm{d}s'\\
    &=A^{\left(i,j\right)}\left(x\right)\left(\frac{ph_{n}}{\rho^{\left(j\right)}}+\frac{p^{2}h_{n}^{2}-\left(p+1-\rho^{\left(j\right)}\right)^{2}h_{n}^{2}-p^{2}h_{n}^{2}+\left(p-\rho^{\left(j\right)}\right)^{2}h_{n}^{2}}{2\rho^{\left(j\right)}h_{n}}\right)\\
    &=A^{\left(i,j\right)}\left(x\right)\left(\frac{ph_{n}}{\rho^{\left(j\right)}}+\frac{-2\left(p-\rho^{\left(j\right)}\right)h_{n}^{2}-h_{n}^{2}}{2\rho^{\left(j\right)}h_{n}}\right)\\
    &=\frac{\left(2\rho^{\left(j\right)}-1\right)h_{n}A^{\left(i,j\right)}\left(x\right)}{2\rho^{\left(j\right)}},
\end{align*}
and if $\rho^{\left(i\right)}>0$, $\rho^{\left(j\right)}=0$,
\begin{align*}
    &{h_{n}}\mathbf{D}_{0}^{\left(i,j\right)}\left(x|\rho\right)\\
    &=
    \int_{0}^{\left(p+1\right)h_{n}}\int_{0}^{\left(p+1\right)h_{n}}A^{\left(i,j\right)}\left(x\right)\min\left\{s,s'\right\}\\
    &\qquad\left(V_{\rho,h_{n}}^{\left(i,i\right)}\left(ph_{n}-s\right)\left(V_{\rho,h_{n}}^{\left(j,j\right)}\left(\left(p+1\right)h_{n}-s'\right)-V_{\rho,h_{n}}^{\left(j,j\right)}\left(ph_{n}-s'\right)\right)\right)\mathrm{d}s'\mathrm{d}s\\
    &=\int_{0}^{\left(p+1\right)h_{n}}A^{\left(i,j\right)}\left(x\right)\min\left\{s,\left(p+1\right)h_{n}\right\}\left(\rho^{\left(i\right)}h_{n}\right)^{-1}\mathbf{1}_{\left[0,\rho^{\left(i\right)}h_{n}\right]}\left(ph_{n}-s\right)\mathrm{d}s\\
    &\quad-\int_{0}^{\left(p+1\right)h_{n}}A^{\left(i,j\right)}\left(x\right)\min\left\{s,ph_{n}\right\}\left(\rho^{\left(i\right)}h_{n}\right)^{-1}\mathbf{1}_{\left[0,\rho^{\left(i\right)}h_{n}\right]}\left(ph_{n}-s\right)\mathrm{d}s\\
    &=\int_{0}^{\left(p+1\right)h_{n}}A^{\left(i,j\right)}\left(x\right)s\left(\rho^{\left(i\right)}h_{n}\right)^{-1}\mathbf{1}_{\left[0,\rho^{\left(i\right)}h_{n}\right]}\left(ph_{n}-s\right)\mathrm{d}s\\
    &\quad-\int_{ph_{n}}^{\left(p+1\right)h_{n}}A^{\left(i,j\right)}\left(x\right)ph_{n}\left(\rho^{\left(i\right)}h_{n}\right)^{-1}\mathbf{1}_{\left[0,\rho^{\left(i\right)}h_{n}\right]}\left(ph_{n}-s\right)\mathrm{d}s\\
    &\quad-\int_{0}^{ph_{n}}A^{\left(i,j\right)}\left(x\right)s\left(\rho^{\left(i\right)}h_{n}\right)^{-1}\mathbf{1}_{\left[0,\rho^{\left(i\right)}h_{n}\right]}\left(ph_{n}-s\right)\mathrm{d}s\\
    &=0,
\end{align*}
and if $\rho^{\left(i\right)}>0$, $\rho^{\left(j\right)}>0$,
\begin{align*}
    &{h_{n}}\mathbf{D}_{0}^{\left(i,j\right)}\left(x|\rho\right)\\
    &= \int_{0}^{\left(p+1\right)h_{n}}\int_{0}^{\left(p+1\right)h_{n}}A^{\left(i,j\right)}\left(x\right)\min\left\{s,s'\right\}\\
    &\qquad\left(V_{\rho,h_{n}}^{\left(i,i\right)}\left(ph_{n}-s\right)\left(V_{\rho,h_{n}}^{\left(j,j\right)}\left(\left(p+1\right)h_{n}-s'\right)-V_{\rho,h_{n}}^{\left(j,j\right)}\left(ph_{n}-s'\right)\right)\right)\mathrm{d}s'\mathrm{d}s\\
    &=\frac{A^{\left(i,j\right)}\left(x\right)}{\rho^{\left(i\right)}\rho^{\left(j\right)}h_{n}^{2}}\int_{0}^{\left(p+1\right)h_{n}}\int_{0}^{\left(p+1\right)h_{n}}\min\left\{s,s'\right\}\mathbf{1}_{\left[0,\rho^{\left(i\right)}h_{n}\right]}\left(ph_{n}-s\right)\\
    &\qquad\left(\mathbf{1}_{\left[0,\rho^{\left(j\right)}h_{n}\right]}\left(\left(p+1\right)h_{n}-s'\right)-\mathbf{1}_{\left[0,\rho^{\left(j\right)}h_{n}\right]}\left(ph_{n}-s'\right)\right)\mathrm{d}s'\mathrm{d}s\\
    &=\frac{A^{\left(i,j\right)}\left(x\right)}{\rho^{\left(i\right)}\rho^{\left(j\right)}h_{n}^{2}}\int_{0}^{\left(p+1\right)h_{n}}\int_{s}^{\left(p+1\right)h_{n}}s\mathbf{1}_{\left[0,\rho^{\left(i\right)}h_{n}\right]}\left(ph_{n}-s\right)\\
    &\hspace{3cm}\left(\mathbf{1}_{\left[0,\rho^{\left(j\right)}h_{n}\right]}\left(\left(p+1\right)h_{n}-s'\right)-\mathbf{1}_{\left[0,\rho^{\left(j\right)}h_{n}\right]}\left(ph_{n}-s'\right)\right)\mathrm{d}s'\mathrm{d}s\\
    &\quad+\frac{A^{\left(i,j\right)}\left(x\right)}{\rho^{\left(i\right)}\rho^{\left(j\right)}h_{n}^{2}}\int_{0}^{\left(p+1\right)h_{n}}\int_{0}^{s}s'\mathbf{1}_{\left[0,\rho^{\left(i\right)}h_{n}\right]}\left(ph_{n}-s\right)\\
    &\hspace{4cm}\left(\mathbf{1}_{\left[0,\rho^{\left(j\right)}h_{n}\right]}\left(\left(p+1\right)h_{n}-s'\right)-\mathbf{1}_{\left[0,\rho^{\left(j\right)}h_{n}\right]}\left(ph_{n}-s'\right)\right)\mathrm{d}s'\mathrm{d}s\\
    &=\frac{A^{\left(i,j\right)}\left(x\right)}{\rho^{\left(i\right)}\rho^{\left(j\right)}h_{n}^{2}}\int_{0}^{\left(p+1\right)h_{n}}s\mathbf{1}_{\left[0,\rho^{\left(i\right)}h_{n}\right]}\left(ph_{n}-s\right)\\
    &\hspace{3cm}\int_{s}^{\left(p+1\right)h_{n}}\left(\mathbf{1}_{\left[0,\rho^{\left(j\right)}h_{n}\right]}\left(\left(p+1\right)h_{n}-s'\right)-\mathbf{1}_{\left[0,\rho^{\left(j\right)}h_{n}\right]}\left(ph_{n}-s'\right)\right)\mathrm{d}s'\mathrm{d}s\\
    &\quad+\frac{A^{\left(i,j\right)}\left(x\right)}{\rho^{\left(i\right)}\rho^{\left(j\right)}h_{n}^{2}}\int_{0}^{\left(p+1\right)h_{n}}\mathbf{1}_{\left[0,\rho^{\left(i\right)}h_{n}\right]}\left(ph_{n}-s\right)\\
    &\hspace{4cm}\int_{0}^{s}s'\left(\mathbf{1}_{\left[0,\rho^{\left(j\right)}h_{n}\right]}\left(\left(p+1\right)h_{n}-s'\right)-\mathbf{1}_{\left[0,\rho^{\left(j\right)}h_{n}\right]}\left(ph_{n}-s'\right)\right)\mathrm{d}s'\mathrm{d}s\\
    &=\frac{A^{\left(i,j\right)}\left(x\right)}{\rho^{\left(i\right)}\rho^{\left(j\right)}h_{n}^{2}}\int_{0}^{\left(p+1\right)h_{n}}s\mathbf{1}_{\left[0,\rho^{\left(i\right)}h_{n}\right]}\left(ph_{n}-s\right)\\
    &\hspace{3cm}\left[\left(p+1\right)h_{n}-\max\left\{s,\left(p+1-\rho^{\left(j\right)}\right)h_{n}\right\}\right.\\
    &\hspace{4cm}\left.-\mathbf{1}_{\left[0,ph_{n}\right]}\left(s\right)\left(ph_{n}-\max\left\{s,\left(p-\rho^{\left(j\right)}\right)h_{n}\right\}\right)\right]\mathrm{d}s\\
    &\quad+\frac{A^{\left(i,j\right)}\left(x\right)}{\rho^{\left(i\right)}\rho^{\left(j\right)}h_{n}^{2}}\int_{0}^{\left(p+1\right)h_{n}}\mathbf{1}_{\left[0,\rho^{\left(i\right)}h_{n}\right]}\left(ph_{n}-s\right)\\
    &\hspace{3cm}\left[\mathbf{1}_{\left[\left(p+1-\rho^{\left(j\right)}\right)h_{n},\left(p+1\right)h_{n}\right]}\left(s\right)\frac{1}{2}\left(s^{2}-\left(p+1-\rho^{\left(j\right)}\right)^{2}h_{n}^{2}\right)\right.\\
    &\hspace{4cm}\left.-\mathbf{1}_{\left[\left(p-\rho^{\left(j\right)}\right)h_{n},\left(p+1\right)h_{n}\right]}\left(s\right)\frac{1}{2}\left(\min\left\{s^{2},p^{2}h_{n}^{2}\right\}-\left(p-\rho^{\left(j\right)}\right)^{2}h_{n}^{2}\right)\right]\mathrm{d}s\\
    &=\frac{A^{\left(i,j\right)}\left(x\right)\left(p+1\right)}{\rho^{\left(i\right)}\rho^{\left(j\right)}h_{n}}\int_{0}^{\left(p+1\right)h_{n}}s\mathbf{1}_{\left[0,\rho^{\left(i\right)}h_{n}\right]}\left(ph_{n}-s\right)\mathrm{d}s\\
    &\quad-\frac{A^{\left(i,j\right)}\left(x\right)}{\rho^{\left(i\right)}\rho^{\left(j\right)}h_{n}^{2}}\int_{0}^{\left(p+1\right)h_{n}}s\mathbf{1}_{\left[0,\rho^{\left(i\right)}h_{n}\right]}\left(ph_{n}-s\right)\max\left\{s,\left(p+1-\rho^{\left(j\right)}\right)h_{n}\right\}\mathrm{d}s\\
    &\quad-\frac{A^{\left(i,j\right)}\left(x\right)p}{\rho^{\left(i\right)}\rho^{\left(j\right)}h_{n}}\int_{0}^{\left(p+1\right)h_{n}}s\mathbf{1}_{\left[0,\rho^{\left(i\right)}h_{n}\right]}\left(ph_{n}-s\right)\mathbf{1}_{\left[0,ph_{n}\right]}\left(s\right)\mathrm{d}s\\
    &\quad+\frac{A^{\left(i,j\right)}\left(x\right)}{\rho^{\left(i\right)}\rho^{\left(j\right)}h_{n}^{2}}\int_{0}^{\left(p+1\right)h_{n}}s\mathbf{1}_{\left[0,\rho^{\left(i\right)}h_{n}\right]}\left(ph_{n}-s\right)\mathbf{1}_{\left[0,ph_{n}\right]}\left(s\right)\max\left\{s,\left(p-\rho^{\left(j\right)}\right)h_{n}\right\}\mathrm{d}s\\
    &\quad+\frac{A^{\left(i,j\right)}\left(x\right)}{2\rho^{\left(i\right)}\rho^{\left(j\right)}h_{n}^{2}}\int_{0}^{\left(p+1\right)h_{n}}\mathbf{1}_{\left[0,\rho^{\left(i\right)}h_{n}\right]}\left(ph_{n}-s\right)\mathbf{1}_{\left[\left(p+1-\rho^{\left(j\right)}\right)h_{n},\left(p+1\right)h_{n}\right]}\left(s\right)s^{2}\mathrm{d}s\\
    &\quad-\frac{A^{\left(i,j\right)}\left(x\right)\left(p+1-\rho^{\left(j\right)}\right)^{2}}{2\rho^{\left(i\right)}\rho^{\left(j\right)}}\int_{0}^{\left(p+1\right)h_{n}}\mathbf{1}_{\left[0,\rho^{\left(i\right)}h_{n}\right]}\left(ph_{n}-s\right)\mathbf{1}_{\left[\left(p+1-\rho^{\left(j\right)}\right)h_{n},\left(p+1\right)h_{n}\right]}\left(s\right)\mathrm{d}s\\
    &\quad-\frac{A^{\left(i,j\right)}\left(x\right)}{2\rho^{\left(i\right)}\rho^{\left(j\right)}h_{n}^{2}}\int_{0}^{\left(p+1\right)h_{n}}\mathbf{1}_{\left[0,\rho^{\left(i\right)}h_{n}\right]}\left(ph_{n}-s\right)\mathbf{1}_{\left[\left(p-\rho^{\left(j\right)}\right)h_{n},\left(p+1\right)h_{n}\right]}\left(s\right)\min\left\{s^{2},p^{2}h_{n}^{2}\right\}\mathrm{d}s\\
    &\quad+\frac{A^{\left(i,j\right)}\left(x\right)\left(p-\rho^{\left(j\right)}\right)^{2}}{2\rho^{\left(i\right)}\rho^{\left(j\right)}}\int_{0}^{\left(p+1\right)h_{n}}\mathbf{1}_{\left[0,\rho^{\left(i\right)}h_{n}\right]}\left(ph_{n}-s\right)\mathbf{1}_{\left[\left(p-\rho^{\left(j\right)}\right)h_{n},\left(p+1\right)h_{n}\right]}\left(s\right)\mathrm{d}s\\
    &=\frac{A^{\left(i,j\right)}\left(x\right)\left(p+1\right)}{\rho^{\left(i\right)}\rho^{\left(j\right)}h_{n}}\int_{\left(p-\rho^{\left(i\right)}\right)h_{n}}^{ph_{n}}s\mathrm{d}s\\
    &\quad-\frac{A^{\left(i,j\right)}\left(x\right)}{\rho^{\left(i\right)}\rho^{\left(j\right)}h_{n}^{2}}\int_{\left(p+1-\rho^{\left(j\right)}\right)h_{n}}^{\left(p+1\right)h_{n}}s^{2}\mathbf{1}_{\left[0,\rho^{\left(i\right)}h_{n}\right]}\left(ph_{n}-s\right)\mathrm{d}s\\
    &\quad-\frac{A^{\left(i,j\right)}\left(x\right)\left(p+1-\rho^{\left(j\right)}\right)}{\rho^{\left(i\right)}\rho^{\left(j\right)}h_{n}}\int_{0}^{\left(p+1-\rho^{\left(j\right)}\right)h_{n}}s\mathbf{1}_{\left[0,\rho^{\left(i\right)}h_{n}\right]}\left(ph_{n}-s\right)\mathrm{d}s\\
    &\quad-\frac{A^{\left(i,j\right)}\left(x\right)p}{\rho^{\left(i\right)}\rho^{\left(j\right)}h_{n}}\int_{\left(p-\rho^{\left(i\right)}\right)h_{n}}^{ph_{n}}s\mathrm{d}s\\
    &\quad+\frac{A^{\left(i,j\right)}\left(x\right)}{\rho^{\left(i\right)}\rho^{\left(j\right)}h_{n}^{2}}\int_{\left(p-\rho^{\left(j\right)}\right)h_{n}}^{\left(p+1\right)h_{n}}s^{2}\mathbf{1}_{\left[0,\rho^{\left(i\right)}h_{n}\right]}\left(ph_{n}-s\right)\mathbf{1}_{\left[0,ph_{n}\right]}\left(s\right)\mathrm{d}s\\
    &\quad+\frac{A^{\left(i,j\right)}\left(x\right)\left(p-\rho^{\left(j\right)}\right)}{\rho^{\left(i\right)}\rho^{\left(j\right)}h_{n}}\int_{0}^{\left(p-\rho^{\left(j\right)}\right)h_{n}}s\mathbf{1}_{\left[0,\rho^{\left(i\right)}h_{n}\right]}\left(ph_{n}-s\right)\mathbf{1}_{\left[0,ph_{n}\right]}\left(s\right)\mathrm{d}s\\
    &\quad+\frac{A^{\left(i,j\right)}\left(x\right)}{2\rho^{\left(i\right)}\rho^{\left(j\right)}h_{n}^{2}}\int_{\left(p-\rho^{\left(i\right)}\right)h_{n}}^{ph_{n}}\mathbf{1}_{\left[\left(p+1-\rho^{\left(j\right)}\right)h_{n},\left(p+1\right)h_{n}\right]}\left(s\right)s^{2}\mathrm{d}s\\
    &\quad-\frac{A^{\left(i,j\right)}\left(x\right)\left(p+1-\rho^{\left(j\right)}\right)^{2}}{2\rho^{\left(i\right)}\rho^{\left(j\right)}}\int_{\left(p-\rho^{\left(i\right)}\right)h_{n}}^{ph_{n}}\mathbf{1}_{\left[\left(p+1-\rho^{\left(j\right)}\right)h_{n},\left(p+1\right)h_{n}\right]}\left(s\right)\mathrm{d}s\\
    &\quad-\frac{A^{\left(i,j\right)}\left(x\right)}{2\rho^{\left(i\right)}\rho^{\left(j\right)}h_{n}^{2}}\int_{\left(p-\rho^{\left(i\right)}\right)h_{n}}^{ph_{n}}\mathbf{1}_{\left[\left(p-\rho^{\left(j\right)}\right)h_{n},\left(p+1\right)h_{n}\right]}\left(s\right)s^{2}\mathrm{d}s\\
    &\quad+\frac{A^{\left(i,j\right)}\left(x\right)\left(p-\rho^{\left(j\right)}\right)^{2}}{2\rho^{\left(i\right)}\rho^{\left(j\right)}}\int_{\left(p-\rho^{\left(i\right)}\right)h_{n}}^{ph_{n}}\mathbf{1}_{\left[\left(p-\rho^{\left(j\right)}\right)h_{n},\left(p+1\right)h_{n}\right]}\left(s\right)\mathrm{d}s\\
    &=\frac{A^{\left(i,j\right)}\left(x\right)\left(p+1\right)}{2\rho^{\left(i\right)}\rho^{\left(j\right)}h_{n}}\left(p^{2}h_{n}^2-\left(p-\rho^{\left(i\right)}\right)^{2}h_{n}^{2}\right)\\
    &\quad-\frac{A^{\left(i,j\right)}\left(x\right)\mathbf{1}_{\left(1,\overline{\rho}\right]}\left(\rho^{\left(j\right)}\right)}{\rho^{\left(i\right)}\rho^{\left(j\right)}h_{n}^{2}}\int_{\max\left\{\left(p-\rho^{\left(i\right)}\right),\left(p+1-\rho^{\left(j\right)}\right)\right\}h_{n}}^{ph_{n}}s^{2}\mathrm{d}s\\
    &\quad-\frac{A^{\left(i,j\right)}\left(x\right)\left(p+1-\rho^{\left(j\right)}\right)\mathbf{1}_{\left(0,\rho^{\left(i\right)}+1\right]}\left(\rho^{\left(j\right)}\right)}{\rho^{\left(i\right)}\rho^{\left(j\right)}h_{n}}\int_{\left(p-\rho^{\left(i\right)}\right)h_{n}}^{\min\left\{p,\left(p+1-\rho^{\left(j\right)}\right)\right\}h_{n}}s\mathrm{d}s\\
    &\quad-\frac{A^{\left(i,j\right)}\left(x\right)p}{2\rho^{\left(i\right)}\rho^{\left(j\right)}h_{n}}\left(p^{2}h_{n}^{2}-\left(p-\rho^{\left(i\right)}\right)^{2}h_{n}^{2}\right)\\
    &\quad+\frac{A^{\left(i,j\right)}\left(x\right)}{\rho^{\left(i\right)}\rho^{\left(j\right)}h_{n}^{2}}\int_{\left(p-\rho^{\left(j\right)}\right)h_{n}}^{ph_{n}}s^{2}\mathbf{1}_{\left[0,\rho^{\left(i\right)}h_{n}\right]}\left(ph_{n}-s\right)\mathrm{d}s\\
    &\quad+\frac{A^{\left(i,j\right)}\left(x\right)\left(p-\rho^{\left(j\right)}\right)}{\rho^{\left(i\right)}\rho^{\left(j\right)}h_{n}}\int_{0}^{\left(p-\rho^{\left(j\right)}\right)h_{n}}s\mathbf{1}_{\left[0,\rho^{\left(i\right)}h_{n}\right]}\left(ph_{n}-s\right)\mathrm{d}s\\
    &\quad+\frac{A^{\left(i,j\right)}\left(x\right)\mathbf{1}_{\left(1,\overline{\rho}\right]}\left(\rho^{\left(j\right)}\right)}{2\rho^{\left(i\right)}\rho^{\left(j\right)}h_{n}^{2}}\int_{\max\left\{\left(p-\rho^{\left(i\right)}\right),\left(p+1-\rho^{\left(j\right)}\right)\right\}h_{n}}^{ph_{n}}s^{2}\mathrm{d}s\\
    &\quad-\frac{A^{\left(i,j\right)}\left(x\right)\left(p+1-\rho^{\left(j\right)}\right)^{2}\mathbf{1}_{\left(1,\overline{\rho}\right]}\left(\rho^{\left(j\right)}\right)}{2\rho^{\left(i\right)}\rho^{\left(j\right)}}\int_{\max\left\{\left(p-\rho^{\left(i\right)}\right),\left(p+1-\rho^{\left(j\right)}\right)\right\}h_{n}}^{ph_{n}}\mathrm{d}s\\
    &\quad-\frac{A^{\left(i,j\right)}\left(x\right)}{2\rho^{\left(i\right)}\rho^{\left(j\right)}h_{n}^{2}}\int_{\max\left\{\left(p-\rho^{\left(i\right)}\right),\left(p-\rho^{\left(j\right)}\right)\right\}h_{n}}^{ph_{n}}s^{2}\mathrm{d}s\\
    &\quad+\frac{A^{\left(i,j\right)}\left(x\right)\left(p-\rho^{\left(j\right)}\right)^{2}}{2\rho^{\left(i\right)}\rho^{\left(j\right)}}\int_{\max\left\{\left(p-\rho^{\left(i\right)}\right),\left(p-\rho^{\left(j\right)}\right)\right\}h_{n}}^{ph_{n}}\mathrm{d}s\\
    &=\frac{A^{\left(i,j\right)}\left(x\right)\left(p+1\right)h_{n}}{2\rho^{\left(i\right)}\rho^{\left(j\right)}}\left(p^{2}-\left(p-\rho^{\left(i\right)}\right)^{2}\right)\\
    &\quad-\frac{A^{\left(i,j\right)}\left(x\right)\mathbf{1}_{\left(1,\overline{\rho}\right]}\left(\rho^{\left(j\right)}\right)h_{n}}{3\rho^{\left(i\right)}\rho^{\left(j\right)}}\left(p^{3}-\max\left\{\left(p-\rho^{\left(i\right)}\right)^{3},\left(p+1-\rho^{\left(j\right)}\right)^{3}\right\}\right)\\
    &\quad-\frac{A^{\left(i,j\right)}\left(x\right)\left(p+1-\rho^{\left(j\right)}\right)\mathbf{1}_{\left(0,\rho^{\left(i\right)}+1\right]}\left(\rho^{\left(j\right)}\right)h_{n}}{2\rho^{\left(i\right)}\rho^{\left(j\right)}}\left(\min\left\{p^{2},\left(p+1-\rho^{\left(j\right)}\right)^{2}\right\}-\left(p-\rho^{\left(i\right)}\right)^{2}\right)\\
    &\quad-\frac{A^{\left(i,j\right)}\left(x\right)ph_{n}}{2\rho^{\left(i\right)}\rho^{\left(j\right)}}\left(p^{2}-\left(p-\rho^{\left(i\right)}\right)^{2}\right)\\
    &\quad+\frac{A^{\left(i,j\right)}\left(x\right)}{\rho^{\left(i\right)}\rho^{\left(j\right)}h_{n}^{2}}\int_{\max\left\{\left(p-\rho^{\left(i\right)}\right),\left(p-\rho^{\left(j\right)}\right)\right\}h_{n}}^{ph_{n}}s^{2}\mathrm{d}s\\
    &\quad+\frac{A^{\left(i,j\right)}\left(x\right)\left(p-\rho^{\left(j\right)}\right)\mathbf{1}_{\left(0,\rho^{\left(i\right)}\right]}\left(\rho^{\left(j\right)}\right)}{\rho^{\left(i\right)}\rho^{\left(j\right)}h_{n}}\int_{\left(p-\rho^{\left(i\right)}\right)h_{n}}^{\left(p-\rho^{\left(j\right)}\right)h_{n}}s\mathrm{d}s\\
    &\quad+\frac{A^{\left(i,j\right)}\left(x\right)\mathbf{1}_{\left(1,\overline{\rho}\right]}\left(\rho^{\left(j\right)}\right)h_{n}}{6\rho^{\left(i\right)}\rho^{\left(j\right)}}\left(p^{3}-\max\left\{\left(p-\rho^{\left(i\right)}\right)^{3},\left(p+1-\rho^{\left(j\right)}\right)^{3}\right\}\right)\\
    &\quad-\frac{A^{\left(i,j\right)}\left(x\right)\left(p+1-\rho^{\left(j\right)}\right)^{2}\mathbf{1}_{\left(1,\overline{\rho}\right]}\left(\rho^{\left(j\right)}\right)h_{n}}{2\rho^{\left(i\right)}\rho^{\left(j\right)}}\left(p-\max\left\{\left(p-\rho^{\left(i\right)}\right),\left(p+1-\rho^{\left(j\right)}\right)\right\}\right)\\
    &\quad-\frac{A^{\left(i,j\right)}\left(x\right)h_{n}}{6\rho^{\left(i\right)}\rho^{\left(j\right)}}\left(p^{3}-\max\left\{\left(p-\rho^{\left(i\right)}\right)^{3},\left(p-\rho^{\left(j\right)}\right)^{3}\right\}\right)\\
    &\quad+\frac{A^{\left(i,j\right)}\left(x\right)\left(p-\rho^{\left(j\right)}\right)^{2}h_{n}}{2\rho^{\left(i\right)}\rho^{\left(j\right)}}\left(p-\max\left\{\left(p-\rho^{\left(i\right)}\right),\left(p-\rho^{\left(j\right)}\right)\right\}\right)\\
    &=\frac{A^{\left(i,j\right)}\left(x\right)h_{n}}{2\rho^{\left(i\right)}\rho^{\left(j\right)}}\left(p^{2}-\left(p-\rho^{\left(i\right)}\right)^{2}\right)\\
    &\quad-\frac{A^{\left(i,j\right)}\left(x\right)\mathbf{1}_{\left(1,\overline{\rho}\right]}\left(\rho^{\left(j\right)}\right)h_{n}}{6\rho^{\left(i\right)}\rho^{\left(j\right)}}\left(p^{3}-\max\left\{\left(p-\rho^{\left(i\right)}\right)^{3},\left(p+1-\rho^{\left(j\right)}\right)^{3}\right\}\right)\\
    &\quad-\frac{A^{\left(i,j\right)}\left(x\right)\left(p+1-\rho^{\left(j\right)}\right)\mathbf{1}_{\left(0,\rho^{\left(i\right)}+1\right]}\left(\rho^{\left(j\right)}\right)h_{n}}{2\rho^{\left(i\right)}\rho^{\left(j\right)}}\left(\min\left\{p^{2},\left(p+1-\rho^{\left(j\right)}\right)^{2}\right\}-\left(p-\rho^{\left(i\right)}\right)^{2}\right)\\
    &\quad+\frac{A^{\left(i,j\right)}\left(x\right)h_{n}}{6\rho^{\left(i\right)}\rho^{\left(j\right)}}\left(p^{3}-\max\left\{\left(p-\rho^{\left(i\right)}\right)^{3},\left(p-\rho^{\left(j\right)}\right)^{3}\right\}\right)\\
    &\quad+\frac{A^{\left(i,j\right)}\left(x\right)\left(p-\rho^{\left(j\right)}\right)\mathbf{1}_{\left(0,\rho^{\left(i\right)}\right]}\left(\rho^{\left(j\right)}\right)h_{n}}{2\rho^{\left(i\right)}\rho^{\left(j\right)}}\left(\left(p-\rho^{\left(j\right)}\right)^{2}-\left(p-\rho^{\left(i\right)}\right)^{2}\right)\\
    &\quad-\frac{A^{\left(i,j\right)}\left(x\right)\left(p+1-\rho^{\left(j\right)}\right)^{2}\mathbf{1}_{\left(1,\overline{\rho}\right]}\left(\rho^{\left(j\right)}\right)h_{n}}{2\rho^{\left(i\right)}\rho^{\left(j\right)}}\left(p-\max\left\{\left(p-\rho^{\left(i\right)}\right),\left(p+1-\rho^{\left(j\right)}\right)\right\}\right)\\
    &\quad+\frac{A^{\left(i,j\right)}\left(x\right)\left(p-\rho^{\left(j\right)}\right)^{2}h_{n}}{2\rho^{\left(i\right)}\rho^{\left(j\right)}}\left(p-\max\left\{\left(p-\rho^{\left(i\right)}\right),\left(p-\rho^{\left(j\right)}\right)\right\}\right)
\end{align*}
and then we should consider five cases as follows: (i) $\rho^{\left(i\right)}+1<\rho^{\left(j\right)}$; (ii) $\rho^{\left(j\right)}> 1$, $ \rho^{\left(i\right)}<\rho^{\left(j\right)}\le \rho^{\left(i\right)}+1$; (iii) $\rho^{\left(j\right)}> 1$, $ \rho^{\left(i\right)}\ge \rho^{\left(j\right)}$; (iv) $\rho^{\left(j\right)}\le 1$, $\rho^{\left(i\right)}<\rho^{\left(j\right)}$; (v) $\rho^{\left(j\right)}\le 1$, $\rho^{\left(i\right)}\ge\rho^{\left(j\right)}$, and for the case (i), we have 
\begin{align*}
    &{h_{n}}\mathbf{D}_{0}^{\left(i,j\right)}\left(x|\rho\right)\\
    &= \int_{0}^{\left(p+1\right)h_{n}}\int_{0}^{\left(p+1\right)h_{n}}A^{\left(i,j\right)}\left(x\right)\min\left\{s,s'\right\}\\
    &\qquad\left(V_{\rho,h_{n}}^{\left(i,i\right)}\left(ph_{n}-s\right)\left(V_{\rho,h_{n}}^{\left(j,j\right)}\left(\left(p+1\right)h_{n}-s'\right)-V_{\rho,h_{n}}^{\left(j,j\right)}\left(ph_{n}-s'\right)\right)\right)\mathrm{d}s'\mathrm{d}s\\
    &=\frac{A^{\left(i,j\right)}\left(x\right)h_{n}}{2\rho^{\left(i\right)}\rho^{\left(j\right)}}\left(p^{2}-\left(p-\rho^{\left(i\right)}\right)^{2}\right)-\frac{A^{\left(i,j\right)}\left(x\right)h_{n}}{6\rho^{\left(i\right)}\rho^{\left(j\right)}}\left(p^{3}-\left(p-\rho^{\left(i\right)}\right)^{3}\right)\\
    &\quad+\frac{A^{\left(i,j\right)}\left(x\right)h_{n}}{6\rho^{\left(i\right)}\rho^{\left(j\right)}}\left(p^{3}-\left(p-\rho^{\left(i\right)}\right)^{3}\right)-\frac{A^{\left(i,j\right)}\left(x\right)\left(p+1-\rho^{\left(j\right)}\right)^{2}h_{n}}{2\rho^{\left(i\right)}\rho^{\left(j\right)}}\left(p-\left(p-\rho^{\left(i\right)}\right)\right)\\
    &\quad+\frac{A^{\left(i,j\right)}\left(x\right)\left(p-\rho^{\left(j\right)}\right)^{2}h_{n}}{2\rho^{\left(i\right)}\rho^{\left(j\right)}}\left(p-\left(p-\rho^{\left(i\right)}\right)\right)\\
    &=\frac{A^{\left(i,j\right)}\left(x\right)h_{n}}{2\rho^{\left(i\right)}\rho^{\left(j\right)}}\left(p^{2}-\left(p-\rho^{\left(i\right)}\right)^{2}\right)\\
    &\quad-\frac{A^{\left(i,j\right)}\left(x\right)\left(p+1-\rho^{\left(j\right)}\right)^{2}\rho^{\left(i\right)}h_{n}}{2\rho^{\left(i\right)}\rho^{\left(j\right)}}+\frac{A^{\left(i,j\right)}\left(x\right)\left(p-\rho^{\left(j\right)}\right)^{2}\rho^{\left(i\right)}h_{n}}{2\rho^{\left(i\right)}\rho^{\left(j\right)}}\\
    &=\frac{A^{\left(i,j\right)}\left(x\right)h_{n}}{2\rho^{\left(i\right)}\rho^{\left(j\right)}}\left(p^{2}-\left(p-\rho^{\left(i\right)}\right)^{2}-2\left(p-\rho^{\left(j\right)}\right)\rho^{\left(i\right)}-\rho^{\left(i\right)}\right)\\
    &=\frac{A^{\left(i,j\right)}\left(x\right)h_{n}}{2\rho^{\left(i\right)}\rho^{\left(j\right)}}\left(p^{2}-p^{2}+2p\rho^{\left(i\right)}-\left(\rho^{\left(i\right)}\right)^{2}-2p\rho^{\left(i\right)}+2\rho^{\left(i\right)}\rho^{\left(j\right)}-\rho^{\left(i\right)}\right)\\
    &=\frac{A^{\left(i,j\right)}\left(x\right)h_{n}}{6\rho^{\left(i\right)}\rho^{\left(j\right)}}\left(6\rho^{\left(i\right)}\rho^{\left(j\right)}-3\left(\rho^{\left(i\right)}\right)^{2}-3\rho^{\left(i\right)}\right),
\end{align*}
and for the case (ii),
\begin{align*}
    &{h_{n}}\mathbf{D}_{0}^{\left(i,j\right)}\left(x|\rho\right)\\
    &=
    \int_{0}^{\left(p+1\right)h_{n}}\int_{0}^{\left(p+1\right)h_{n}}A^{\left(i,j\right)}\left(x\right)\min\left\{s,s'\right\}\\
    &\qquad\left(V_{\rho,h_{n}}^{\left(i,i\right)}\left(ph_{n}-s\right)\left(V_{\rho,h_{n}}^{\left(j,j\right)}\left(\left(p+1\right)h_{n}-s'\right)-V_{\rho,h_{n}}^{\left(j,j\right)}\left(ph_{n}-s'\right)\right)\right)\mathrm{d}s'\mathrm{d}s\\
    &=\frac{A^{\left(i,j\right)}\left(x\right)h_{n}}{2\rho^{\left(i\right)}\rho^{\left(j\right)}}\left(p^{2}-\left(p-\rho^{\left(i\right)}\right)^{2}\right)-\frac{A^{\left(i,j\right)}\left(x\right)h_{n}}{6\rho^{\left(i\right)}\rho^{\left(j\right)}}\left(p^{3}-\left(p+1-\rho^{\left(j\right)}\right)^{3}\right)\\
    &\quad-\frac{A^{\left(i,j\right)}\left(x\right)\left(p+1-\rho^{\left(j\right)}\right)h_{n}}{2\rho^{\left(i\right)}\rho^{\left(j\right)}}\left(\left(p+1-\rho^{\left(j\right)}\right)^{2}-\left(p-\rho^{\left(i\right)}\right)^{2}\right)\\
    &\quad+\frac{A^{\left(i,j\right)}\left(x\right)h_{n}}{6\rho^{\left(i\right)}\rho^{\left(j\right)}}\left(p^{3}-\left(p-\rho^{\left(i\right)}\right)^{3}\right)\\
    &\quad-\frac{A^{\left(i,j\right)}\left(x\right)\left(p+1-\rho^{\left(j\right)}\right)^{2}h_{n}}{2\rho^{\left(i\right)}\rho^{\left(j\right)}}\left(p-\left(p+1-\rho^{\left(j\right)}\right)\right)\\
    &\quad+\frac{A^{\left(i,j\right)}\left(x\right)\left(p-\rho^{\left(j\right)}\right)^{2}h_{n}}{2\rho^{\left(i\right)}\rho^{\left(j\right)}}\left(p-\left(p-\rho^{\left(i\right)}\right)\right)\\
    &=\frac{A^{\left(i,j\right)}\left(x\right)h_{n}}{6\rho^{\left(i\right)}\rho^{\left(j\right)}}\left[3\left(p^{2}-\left(p-\rho^{\left(i\right)}\right)^{2}\right)-\left(p^{3}-\left(p+1-\rho^{\left(j\right)}\right)^{3}\right)\right.\\
    &\hspace{3cm}-3\left(p+1-\rho^{\left(j\right)}\right)\left(\left(p+1-\rho^{\left(j\right)}\right)^{2}-\left(p-\rho^{\left(i\right)}\right)^{2}\right)\\
    &\hspace{3cm}+\left(p^{3}-\left(p-\rho^{\left(i\right)}\right)^{3}\right)\\
    &\hspace{3cm}-3\left(p+1-\rho^{\left(j\right)}\right)^{2}\left(p-\left(p+1-\rho^{\left(j\right)}\right)\right)\\
    &\hspace{3cm}+3\left(p-\rho^{\left(j\right)}\right)^{2}\left(p-\left(p-\rho^{\left(i\right)}\right)\right)\\
    &=\frac{A^{\left(i,j\right)}\left(x\right)h_{n}}{6\rho^{\left(i\right)}\rho^{\left(j\right)}}\left[\left(\rho^{\left(i\right)}-\rho^{\left(j\right)}\right)^{3}+3\left(\rho^{\left(j\right)}\right)^{2}-3\rho^{\left(j\right)}+1\right],
\end{align*}
and for the case (iii),
\begin{align*}
    &{h_{n}}\mathbf{D}_{0}^{\left(i,j\right)}\left(x|\rho\right)\\
    &= \int_{0}^{\left(p+1\right)h_{n}}\int_{0}^{\left(p+1\right)h_{n}}A^{\left(i,j\right)}\left(x\right)\min\left\{s,s'\right\}\\
    &\qquad\left(V_{\rho,h_{n}}^{\left(i,i\right)}\left(ph_{n}-s\right)\left(V_{\rho,h_{n}}^{\left(j,j\right)}\left(\left(p+1\right)h_{n}-s'\right)-V_{\rho,h_{n}}^{\left(j,j\right)}\left(ph_{n}-s'\right)\right)\right)\mathrm{d}s'\mathrm{d}s\\
    &=\frac{A^{\left(i,j\right)}\left(x\right)h_{n}}{2\rho^{\left(i\right)}\rho^{\left(j\right)}}\left(p^{2}-\left(p-\rho^{\left(i\right)}\right)^{2}\right)-\frac{A^{\left(i,j\right)}\left(x\right)h_{n}}{6\rho^{\left(i\right)}\rho^{\left(j\right)}}\left(p^{3}-\left(p+1-\rho^{\left(j\right)}\right)^{3}\right)\\
    &\quad-\frac{A^{\left(i,j\right)}\left(x\right)\left(p+1-\rho^{\left(j\right)}\right)h_{n}}{2\rho^{\left(i\right)}\rho^{\left(j\right)}}\left(\left(p+1-\rho^{\left(j\right)}\right)^{2}-\left(p-\rho^{\left(i\right)}\right)^{2}\right)\\
    &\quad+\frac{A^{\left(i,j\right)}\left(x\right)h_{n}}{6\rho^{\left(i\right)}\rho^{\left(j\right)}}\left(p^{3}-\left(p-\rho^{\left(j\right)}\right)^{3}\right)\\
    &\quad+\frac{A^{\left(i,j\right)}\left(x\right)\left(p-\rho^{\left(j\right)}\right)h_{n}}{2\rho^{\left(i\right)}\rho^{\left(j\right)}}\left(\left(p-\rho^{\left(j\right)}\right)^{2}-\left(p-\rho^{\left(i\right)}\right)^{2}\right)\\
    &\quad-\frac{A^{\left(i,j\right)}\left(x\right)\left(p+1-\rho^{\left(j\right)}\right)^{2}h_{n}}{2\rho^{\left(i\right)}\rho^{\left(j\right)}}\left(p-\left(p+1-\rho^{\left(j\right)}\right)\right)\\
    &\quad+\frac{A^{\left(i,j\right)}\left(x\right)\left(p-\rho^{\left(j\right)}\right)^{2}h_{n}}{2\rho^{\left(i\right)}\rho^{\left(j\right)}}\left(p-\left(p-\rho^{\left(j\right)}\right)\right)\\
    &=\frac{A^{\left(i,j\right)}\left(x\right)h_{n}}{6\rho^{\left(i\right)}\rho^{\left(j\right)}}\left[3\left(p^{2}-\left(p-\rho^{\left(i\right)}\right)^{2}\right)-\left(p^{3}-\left(p+1-\rho^{\left(j\right)}\right)^{3}\right)\right.\\
    &\hspace{3cm}-3\left(p+1-\rho^{\left(j\right)}\right)\left(\left(p+1-\rho^{\left(j\right)}\right)^{2}-\left(p-\rho^{\left(i\right)}\right)^{2}\right)\\
    &\hspace{3cm}+\left(p^{3}-\left(p-\rho^{\left(j\right)}\right)^{3}\right)+3\left(p-\rho^{\left(j\right)}\right)\left(\left(p-\rho^{\left(j\right)}\right)^{2}-\left(p-\rho^{\left(i\right)}\right)^{2}\right)\\
    &\hspace{3cm}\left.-3\left(p+1-\rho^{\left(j\right)}\right)^{2}\left(p-\left(p+1-\rho^{\left(j\right)}\right)\right)+3\left(p-\rho^{\left(j\right)}\right)^{2}\rho^{\left(j\right)}\right]\\
    &=\frac{A^{\left(i,j\right)}\left(x\right)h_{n}}{6\rho^{\left(i\right)}\rho^{\left(j\right)}}\left[3\left(\rho^{\left(j\right)}\right)^{2}-3\rho^{\left(j\right)}+1\right]
\end{align*}
and for the case (iv),
\begin{align*}
    {h_{n}}\mathbf{D}_{0}^{\left(i,j\right)}\left(x|\rho\right) &=
    \int_{0}^{\left(p+1\right)h_{n}}\int_{0}^{\left(p+1\right)h_{n}}A^{\left(i,j\right)}\left(x\right)\min\left\{s,s'\right\}\\
    &\qquad\left(V_{\rho,h_{n}}^{\left(i,i\right)}\left(ph_{n}-s\right)\left(V_{\rho,h_{n}}^{\left(j,j\right)}\left(\left(p+1\right)h_{n}-s'\right)-V_{\rho,h_{n}}^{\left(j,j\right)}\left(ph_{n}-s'\right)\right)\right)\mathrm{d}s'\mathrm{d}s\\
    &=\frac{A^{\left(i,j\right)}\left(x\right)h_{n}}{2\rho^{\left(i\right)}\rho^{\left(j\right)}}\left(p^{2}-\left(p-\rho^{\left(i\right)}\right)^{2}\right)\\
    &\quad-\frac{A^{\left(i,j\right)}\left(x\right)\left(p+1-\rho^{\left(j\right)}\right)h_{n}}{2\rho^{\left(i\right)}\rho^{\left(j\right)}}\left(p^{2}-\left(p-\rho^{\left(i\right)}\right)^{2}\right)\\
    &\quad+\frac{A^{\left(i,j\right)}\left(x\right)h_{n}}{6\rho^{\left(i\right)}\rho^{\left(j\right)}}\left(p^{3}-\left(p-\rho^{\left(i\right)}\right)^{3}\right)\\
    &\quad+\frac{A^{\left(i,j\right)}\left(x\right)\left(p-\rho^{\left(j\right)}\right)^{2}h_{n}}{2\rho^{\left(i\right)}\rho^{\left(j\right)}}\left(p-\left(p-\rho^{\left(i\right)}\right)\right)\\
    &=\frac{A^{\left(i,j\right)}\left(x\right)h_{n}}{6\rho^{\left(i\right)}\rho^{\left(j\right)}}\left[3\left(p^{2}-\left(p-\rho^{\left(i\right)}\right)^{2}\right)-3\left(p+1-\rho^{\left(j\right)}\right)\left(p^{2}-\left(p-\rho^{\left(i\right)}\right)^{2}\right)\right.\\
    &\hspace{3cm}\left.+\left(p^{3}-\left(p-\rho^{\left(i\right)}\right)^{3}\right)+3\left(p-\rho^{\left(j\right)}\right)^{2}\left(p-\left(p-\rho^{\left(i\right)}\right)\right)\right]\\
    &=\frac{A^{\left(i,j\right)}\left(x\right)h_{n}}{6\rho^{\left(i\right)}\rho^{\left(j\right)}}\left[\left(\rho^{\left(i\right)}-\rho^{\left(j\right)}\right)^{3}+\left(\rho^{\left(j\right)}\right)^{3}\right]
\end{align*}
and for the case (v),
\begin{align*}
    &{h_{n}}\mathbf{D}_{0}^{\left(i,j\right)}\left(x|\rho\right)\\
    &=
    \int_{0}^{\left(p+1\right)h_{n}}\int_{0}^{\left(p+1\right)h_{n}}A^{\left(i,j\right)}\left(x\right)\min\left\{s,s'\right\}\\
    &\qquad\left(V_{\rho,h_{n}}^{\left(i,i\right)}\left(ph_{n}-s\right)\left(V_{\rho,h_{n}}^{\left(j,j\right)}\left(\left(p+1\right)h_{n}-s'\right)-V_{\rho,h_{n}}^{\left(j,j\right)}\left(ph_{n}-s'\right)\right)\right)\mathrm{d}s'\mathrm{d}s\\
    &=\frac{A^{\left(i,j\right)}\left(x\right)h_{n}}{2\rho^{\left(i\right)}\rho^{\left(j\right)}}\left(p^{2}-\left(p-\rho^{\left(i\right)}\right)^{2}\right)\\
    &\quad-\frac{A^{\left(i,j\right)}\left(x\right)\left(p+1-\rho^{\left(j\right)}\right)h_{n}}{2\rho^{\left(i\right)}\rho^{\left(j\right)}}\left(p^{2}-\left(p-\rho^{\left(i\right)}\right)^{2}\right)\\
    &\quad+\frac{A^{\left(i,j\right)}\left(x\right)h_{n}}{6\rho^{\left(i\right)}\rho^{\left(j\right)}}\left(p^{3}-\left(p-\rho^{\left(j\right)}\right)^{3}\right)\\
    &\quad+\frac{A^{\left(i,j\right)}\left(x\right)\left(p-\rho^{\left(j\right)}\right)h_{n}}{2\rho^{\left(i\right)}\rho^{\left(j\right)}}\left(\left(p-\rho^{\left(j\right)}\right)^{2}-\left(p-\rho^{\left(i\right)}\right)^{2}\right)\\
    &\quad+\frac{A^{\left(i,j\right)}\left(x\right)\left(p-\rho^{\left(j\right)}\right)^{2}h_{n}}{2\rho^{\left(i\right)}\rho^{\left(j\right)}}\left(p-\left(p-\rho^{\left(j\right)}\right)\right)\\
    &=\frac{A^{\left(i,j\right)}\left(x\right)h_{n}}{6\rho^{\left(i\right)}\rho^{\left(j\right)}}\left[3\left(p^{2}-\left(p-\rho^{\left(i\right)}\right)^{2}\right)\right.\\
    &\hspace{3cm}-3\left(p+1-\rho^{\left(j\right)}\right)\left(p^{2}-\left(p-\rho^{\left(i\right)}\right)^{2}\right)+\left(p^{3}-\left(p-\rho^{\left(j\right)}\right)^{3}\right)\\
    &\hspace{3cm}\left.+3\left(p-\rho^{\left(j\right)}\right)\left(\left(p-\rho^{\left(j\right)}\right)^{2}-\left(p-\rho^{\left(i\right)}\right)^{2}\right)+3\left(p-\rho^{\left(j\right)}\right)^{2}\left(p-\left(p-\rho^{\left(j\right)}\right)\right)\right]\\
    &=\frac{A^{\left(i,j\right)}\left(x\right)h_{n}}{6\rho^{\left(i\right)}\rho^{\left(j\right)}}\left(\rho^{\left(j\right)}\right)^{3}.
\end{align*}
Hence, 
it follows from (i)-(v) that
 $\mathbf{D}_{0}\left(x|\rho\right)=\mathbb{D}_{0}\left(x|\rho\right)$. Regarding $D_{\ell}\left(x\right)$ where $\ell\ge \left[\max_{i=1,\ldots,n}\rho^{\left(i\right)}\right]+1$, it is obvious that
\begin{align*}
    &\mathbf{E}\left[\left(\int_{0}^{\left(p+1+\ell\right)h_{n}}V_{\rho,h_{n}}\left(ph_{n}-s_{1}\right)\left(\int_{0}^{s_{1}}a\left(x\right)\mathrm{d}w_{s_{2}}\right)\mathrm{d}s_{1}\right)\right.\\
    &\hspace{1cm}\left(\int_{0}^{\left(p+1+\ell\right)h_{n}}\left(V_{\rho,h_{n}}\left(\left(p+1+\ell\right)h_{n}-s_{1}\right)-V_{\rho,h_{n}}\left(\left(p+\ell\right) h_{n}-s_{1}\right)\right)\right.\\
    &\hspace{9cm}\left.\left.\times\left(\int_{0}^{s_{1}}a\left(x\right)\mathrm{d}w_{s_{2}}\right)\mathrm{d}s_{1}\right)^{T}\right]=O
\end{align*}
With respect to $G\left(x\right)$, for all $i,j=1,\ldots,d$, let us define $\mathbf{G}^{\left(i,j\right)}\left(x|\rho\right)$ such that
\begin{align*}
    &\mathbf{G}^{\left(i,j\right)}\left(x|\rho\right)\\
    &:={\frac{1}{h_{n}}}\mathbf{E}\left[\left(\int_{0}^{\Delta_{n}+h_{n}}\left(\Phi_{\Delta_{n},n}\left(\left(\Delta_{n}+h_{n}\right)-s_{1}\right)-\Phi_{\Delta_{n},n}\left(\Delta_{n}-s_{1}\right)\right)\left(\int_{0}^{s_{1}}a\left(x\right)\mathrm{d}w_{s_{2}}\right)\mathrm{d}s_{1}\right)\right.\\
    &\quad\left.\left(\int_{0}^{\Delta_{n}+h_{n}}\left(\Phi_{\Delta_{n},n}\left(\left(\Delta_{n}+h_{n}\right)-s_{1}\right)-\Phi_{\Delta_{n},n}\left(\Delta_{n}-s_{1}\right)\right)\left(\int_{0}^{s_{1}}a\left(x\right)\mathrm{d}w_{s_{2}}\right)\mathrm{d}s_{1}\right)^{T}\right]^{\left(i,j\right)}\\
    &={\frac{1}{h_{n}}}\int_{0}^{\left(p+1\right)h_{n}}\int_{0}^{\left(p+1\right)h_{n}}A^{\left(i,j\right)}\left(x\right)\min\left\{s,s'\right\}\left(V_{\rho,h_{n}}^{\left(i,i\right)}\left(\left(p+1\right)h_{n}-s\right)-V_{\rho,h_{n}}^{\left(i,i\right)}\left(ph_{n}-s\right)\right)\\
    &\hspace{3cm}\left(V_{\rho,h_{n}}^{\left(j,j\right)}\left(\left(p+1\right)h_{n}-s'\right)-V_{\rho,h_{n}}^{\left(j,j\right)}\left(ph_{n}-s'\right)\right)\mathrm{d}s'\mathrm{d}s\\
    &={\frac{1}{h_{n}}}\int_{0}^{\left(p+1\right)h_{n}}\int_{0}^{\left(p+1\right)h_{n}}A^{\left(i,j\right)}\left(x\right)\min\left\{s,s'\right\}\left(V_{\rho,h_{n}}^{\left(i,i\right)}\left(\left(p+1\right)h_{n}-s\right)\right)\\
    &\hspace{3cm}\left(V_{\rho,h_{n}}^{\left(j,j\right)}\left(\left(p+1\right)h_{n}-s'\right)-V_{\rho,h_{n}}^{\left(j,j\right)}\left(ph_{n}-s'\right)\right)\mathrm{d}s'\mathrm{d}s-{\mathbf{D}_{0}^{\left(i,j\right)}\left(x|\rho\right)},
\end{align*}
{and $\mathbf{K}^{\left(i,j\right)}\left(x|\rho\right):= \mathbf{G}^{\left(i,j\right)}\left(x|\rho\right)+\mathbf{D}_{0}^{\left(i,j\right)}\left(x|\rho\right)$. If} $\rho^{\left(i\right)}=0$, as {evaluation of} $B$,
\begin{align*}
    &{h_{n}\mathbf{K}^{\left(i,j\right)}\left(x|\rho\right)}\\
    &=
    \int_{0}^{\left(p+1\right)h_{n}}\int_{0}^{\left(p+1\right)h_{n}}A^{\left(i,j\right)}\left(x\right)\min\left\{s,s'\right\}\left(V_{\rho,h_{n}}^{\left(i,i\right)}\left(\left(p+1\right)h_{n}-s\right)\right)\\
    &\hspace{3cm}\left(V_{\rho,h_{n}}^{\left(j,j\right)}\left(\left(p+1\right)h_{n}-s'\right)-V_{\rho,h_{n}}^{\left(j,j\right)}\left(ph_{n}-s'\right)\right)\mathrm{d}s'\mathrm{d}s\\
    &=A^{\left(i,j\right)}\left(x\right)\int_{0}^{\left(p+1\right)h_{n}}\left(V_{\rho,h_{n}}^{\left(j,j\right)}\left(\left(p+1\right)h_{n}-s'\right)-V_{\rho,h_{n}}^{\left(j,j\right)}\left(ph_{n}-s'\right)\right)s'\mathrm{d}s'\\
    &=h_{n}A^{\left(i,j\right)}\left(x\right),
\end{align*}
and if $\rho^{\left(i\right)}\in\left(0,1\right]$ and $\rho^{\left(j\right)}=0$,
\begin{align*}
    &{h_{n}
	\mathbf{K}^{\left(i,j\right)}\left(x|\rho\right)}\\
	&= \int_{0}^{\left(p+1\right)h_{n}}\int_{0}^{\left(p+1\right)h_{n}}A^{\left(i,j\right)}\left(x\right)\min\left\{s,s'\right\}\left(V_{\rho,h_{n}}^{\left(i,i\right)}\left(\left(p+1\right)h_{n}-s\right)\right)\\
    &\hspace{3cm}\left(V_{\rho,h_{n}}^{\left(j,j\right)}\left(\left(p+1\right)h_{n}-s'\right)-V_{\rho,h_{n}}^{\left(j,j\right)}\left(ph_{n}-s'\right)\right)\mathrm{d}s'\mathrm{d}s\\
    &=\int_{0}^{\left(p+1\right)h_{n}}\int_{0}^{\left(p+1\right)h_{n}}A^{\left(i,j\right)}\left(x\right)\min\left\{s,s'\right\}\left(V_{\rho,h_{n}}^{\left(i,i\right)}\left(\left(p+1\right)h_{n}-s\right)\right)\\
    &\hspace{3cm}\times  V_{\rho,h_{n}}^{\left(j,j\right)}\left(\left(p+1\right)h_{n}-s'\right)\mathrm{d}s'\mathrm{d}s\\
    &\quad-\int_{0}^{\left(p+1\right)h_{n}}\int_{0}^{\left(p+1\right)h_{n}}A^{\left(i,j\right)}\left(x\right)\min\left\{s,s'\right\}\left(V_{\rho,h_{n}}^{\left(i,i\right)}\left(\left(p+1\right)h_{n}-s\right)\right)\\
    &\hspace{3cm}\times V_{\rho,h_{n}}^{\left(j,j\right)}\left(ph_{n}-s'\right)\mathrm{d}s'\mathrm{d}s\\
    &=A^{\left(i,j\right)}\left(x\right)\int_{0}^{\left(p+1\right)h_{n}}s\left(V_{\rho,h_{n}}^{\left(i,i\right)}\left(\left(p+1\right)h_{n}-s\right)\right)\mathrm{d}s\\
    &\quad-A^{\left(i,j\right)}\left(x\right)\int_{0}^{\left(p+1\right)h_{n}}\min\left\{s,ph_{n}\right\}\left(V_{\rho,h_{n}}^{\left(i,i\right)}\left(\left(p+1\right)h_{n}-s\right)\right)\mathrm{d}s\\
    &=\frac{A^{\left(i,j\right)}\left(x\right)}{\rho^{\left(i\right)}h_{n}}\left(\int_{\left(p+1-\rho^{\left(i\right)}\right)h_{n}}^{\left(p+1\right)h_{n}}s\mathrm{d}s-\int_{\left(p+1-\rho^{\left(i\right)}\right)h_{n}}^{\left(p+1\right)h_{n}}\min\left\{s,ph_{n}\right\}\mathrm{d}s\right)\\
    &=\frac{A^{\left(i,j\right)}\left(x\right)}{\rho^{\left(i\right)}h_{n}}\left(\int_{\left(p+1-\rho^{\left(i\right)}\right)h_{n}}^{\left(p+1\right)h_{n}}s\mathrm{d}s-\int_{\left(p+1-\rho^{\left(i\right)}\right)h_{n}}^{\left(p+1\right)h_{n}}ph_{n}\mathrm{d}s\right)\\
    &=\frac{A^{\left(i,j\right)}\left(x\right)}{\rho^{\left(i\right)}h_{n}}\left(\frac{1}{2}\left(p+1\right)^{2}h_{n}^{2}-\frac{1}{2}\left(p+1-\rho^{\left(i\right)}\right)^{2}h_{n}^{2}-p\left(p+1\right)h_{n}^{2}+p\left(p+1-\rho^{\left(i\right)}\right)h_{n}^{2}\right)\\
    &=h_{n}A^{\left(i,j\right)}\left(x\right)\left(1-\frac{\rho^{\left(i\right)}}{2}\right)
\end{align*}
and if $\rho^{\left(i\right)}\in\left(1,\overline{\rho}\right]$ and $\rho^{\left(j\right)}=0$,
\begin{align*}
    &{h_{n}
	\mathbf{K}^{\left(i,j\right)}\left(x|\rho\right)}\\&=
    \int_{0}^{\left(p+1\right)h_{n}}\int_{0}^{\left(p+1\right)h_{n}}A^{\left(i,j\right)}\left(x\right)\min\left\{s,s'\right\}\left(V_{\rho,h_{n}}^{\left(i,i\right)}\left(\left(p+1\right)h_{n}-s\right)\right)\\
    &\hspace{3cm}\left(V_{\rho,h_{n}}^{\left(j,j\right)}\left(\left(p+1\right)h_{n}-s'\right)-V_{\rho,h_{n}}^{\left(j,j\right)}\left(ph_{n}-s'\right)\right)\mathrm{d}s'\mathrm{d}s\\
    &=\frac{A^{\left(i,j\right)}\left(x\right)}{\rho^{\left(i\right)}h_{n}}\left(\int_{\left(p+1-\rho^{\left(i\right)}\right)h_{n}}^{\left(p+1\right)h_{n}}s\mathrm{d}s-\int_{\left(p+1-\rho^{\left(i\right)}\right)h_{n}}^{\left(p+1\right)h_{n}}\min\left\{s,ph_{n}\right\}\mathrm{d}s\right)\\
    &=\frac{A^{\left(i,j\right)}\left(x\right)}{\rho^{\left(i\right)}h_{n}}\left(\int_{\left(p+1-\rho^{\left(i\right)}\right)h_{n}}^{\left(p+1\right)h_{n}}s\mathrm{d}s-\int_{ph_{n}}^{\left(p+1\right)h_{n}}ph_{n}\mathrm{d}s-\int_{\left(p+1-\rho^{\left(i\right)}\right)h_{n}}^{ph_{n}}s\mathrm{d}s\right)\\
    &=\frac{A^{\left(i,j\right)}\left(x\right)}{\rho^{\left(i\right)}h_{n}}\left(\frac{1}{2}\left(p+1\right)^{2}h_{n}^{2}-\frac{1}{2}\left(p+1-\rho^{\left(i\right)}\right)^{2}h_{n}^{2}-p\left(p+1\right)h_{n}^{2}\right.\\
    &\hspace{3cm}\left.+p^{2}h_{n}^{2}-\frac{1}{2}p^{2}h_{n}^{2}+\frac{1}{2}\left(p+1-\rho^{\left(i\right)}\right)^{2}h_{n}^{2}\right)\\
    &=\frac{h_{n}A^{\left(i,j\right)}\left(x\right)}{2\rho^{\left(i\right)}}
\end{align*}
and if $\rho^{\left(i\right)}>0$ and $\rho^{\left(j\right)}>0$,
\begin{align*}
    &{h_{n}
	\mathbf{K}^{\left(i,j\right)}\left(x|\rho\right)}\\
    &=
    \int_{0}^{\left(p+1\right)h_{n}}\int_{0}^{\left(p+1\right)h_{n}}A^{\left(i,j\right)}\left(x\right)\min\left\{s,s'\right\}\left(V_{\rho,h_{n}}^{\left(i,i\right)}\left(\left(p+1\right)h_{n}-s\right)\right)\\
    &\hspace{3cm}\left(V_{\rho,h_{n}}^{\left(j,j\right)}\left(\left(p+1\right)h_{n}-s'\right)-V_{\rho,h_{n}}^{\left(j,j\right)}\left(ph_{n}-s'\right)\right)\mathrm{d}s'\mathrm{d}s\\
    &=\frac{A^{\left(i,j\right)}\left(x\right)}{\rho^{\left(i\right)}\rho^{\left(j\right)}h_{n}^{2}}\int_{0}^{\left(p+1\right)h_{n}}\int_{0}^{\left(p+1\right)h_{n}}\min\left\{s,s'\right\}\left(\mathbf{1}_{\left[0,\rho^{\left(i\right)}h_{n}\right]}\left(\left(p+1\right)h_{n}-s\right)\right)\\
    &\hspace{5cm}\left(\mathbf{1}_{\left[0,\rho^{\left(j\right)}h_{n}\right]}\left(\left(p+1\right)h_{n}-s'\right)-\mathbf{1}_{\left[0,\rho^{\left(j\right)}h_{n}\right]}\left(ph_{n}-s'\right)\right)\mathrm{d}s'\mathrm{d}s\\
    &=\frac{A^{\left(i,j\right)}\left(x\right)}{\rho^{\left(i\right)}\rho^{\left(j\right)}h_{n}^{2}}\int_{0}^{\left(p+1\right)h_{n}}\int_{s}^{\left(p+1\right)h_{n}}s\left(\mathbf{1}_{\left[0,\rho^{\left(i\right)}h_{n}\right]}\left(\left(p+1\right)h_{n}-s\right)\right)\\
    &\hspace{5cm}\left(\mathbf{1}_{\left[0,\rho^{\left(j\right)}h_{n}\right]}\left(\left(p+1\right)h_{n}-s'\right)-\mathbf{1}_{\left[0,\rho^{\left(j\right)}h_{n}\right]}\left(ph_{n}-s'\right)\right)\mathrm{d}s'\mathrm{d}s\\
    &\quad+\frac{A^{\left(i,j\right)}\left(x\right)}{\rho^{\left(i\right)}\rho^{\left(j\right)}h_{n}^{2}}\int_{0}^{\left(p+1\right)h_{n}}\int_{0}^{s}s'\left(\mathbf{1}_{\left[0,\rho^{\left(i\right)}h_{n}\right]}\left(\left(p+1\right)h_{n}-s\right)\right)\\
    &\hspace{5cm}\left(\mathbf{1}_{\left[0,\rho^{\left(j\right)}h_{n}\right]}\left(\left(p+1\right)h_{n}-s'\right)-\mathbf{1}_{\left[0,\rho^{\left(j\right)}h_{n}\right]}\left(ph_{n}-s'\right)\right)\mathrm{d}s'\mathrm{d}s\\
    &=\frac{A^{\left(i,j\right)}\left(x\right)}{\rho^{\left(i\right)}\rho^{\left(j\right)}h_{n}^{2}}\int_{0}^{\left(p+1\right)h_{n}}s\left(\mathbf{1}_{\left[0,\rho^{\left(i\right)}h_{n}\right]}\left(\left(p+1\right)h_{n}-s\right)\right)\\
    &\hspace{5cm}\int_{s}^{\left(p+1\right)h_{n}}\left(\mathbf{1}_{\left[0,\rho^{\left(j\right)}h_{n}\right]}\left(\left(p+1\right)h_{n}-s'\right)\right)\mathrm{d}s'\mathrm{d}s\\
    &\quad-\frac{A^{\left(i,j\right)}\left(x\right)}{\rho^{\left(i\right)}\rho^{\left(j\right)}h_{n}^{2}}\int_{0}^{\left(p+1\right)h_{n}}s\left(\mathbf{1}_{\left[0,\rho^{\left(i\right)}h_{n}\right]}\left(\left(p+1\right)h_{n}-s\right)\right)\\
    &\hspace{5cm}\int_{s}^{\left(p+1\right)h_{n}}\left(\mathbf{1}_{\left[0,\rho^{\left(j\right)}h_{n}\right]}\left(ph_{n}-s'\right)\right)\mathrm{d}s'\mathrm{d}s\\
    &\quad+\frac{A^{\left(i,j\right)}\left(x\right)}{\rho^{\left(i\right)}\rho^{\left(j\right)}h_{n}^{2}}\int_{0}^{\left(p+1\right)h_{n}}\left(\mathbf{1}_{\left[0,\rho^{\left(i\right)}h_{n}\right]}\left(\left(p+1\right)h_{n}-s\right)\right)\\
    &\hspace{5cm}\int_{0}^{s}s'\left(\mathbf{1}_{\left[0,\rho^{\left(j\right)}h_{n}\right]}\left(\left(p+1\right)h_{n}-s'\right)\right)\mathrm{d}s'\mathrm{d}s\\
    &\quad-\frac{A^{\left(i,j\right)}\left(x\right)}{\rho^{\left(i\right)}\rho^{\left(j\right)}h_{n}^{2}}\int_{0}^{\left(p+1\right)h_{n}}\left(\mathbf{1}_{\left[0,\rho^{\left(i\right)}h_{n}\right]}\left(\left(p+1\right)h_{n}-s\right)\right)\\
    &\hspace{5cm}\int_{0}^{s}s'\left(\mathbf{1}_{\left[0,\rho^{\left(j\right)}h_{n}\right]}\left(ph_{n}-s'\right)\right)\mathrm{d}s'\mathrm{d}s\\
    &=\frac{A^{\left(i,j\right)}\left(x\right)}{\rho^{\left(i\right)}\rho^{\left(j\right)}h_{n}^{2}}\int_{0}^{\left(p+1\right)h_{n}}s\left(\mathbf{1}_{\left[0,\rho^{\left(i\right)}h_{n}\right]}\left(\left(p+1\right)h_{n}-s\right)\right)\int_{\max\left\{s,\left(p+1-\rho^{\left(j\right)}\right)h_{n}\right\}}^{\left(p+1\right)h_{n}}\mathrm{d}s'\mathrm{d}s\\
    &\quad-\frac{A^{\left(i,j\right)}\left(x\right)}{\rho^{\left(i\right)}\rho^{\left(j\right)}h_{n}^{2}}\int_{0}^{\left(p+1\right)h_{n}}s\left(\mathbf{1}_{\left[0,\rho^{\left(i\right)}h_{n}\right]}\left(\left(p+1\right)h_{n}-s\right)\right)\\
    &\hspace{5cm}\left(\mathbf{1}_{\left[0,ph_{n}\right]}\left(s\right)\right)\int_{\max\left\{s,\left(p-\rho^{\left(j\right)}\right)h_{n}\right\}}^{ph_{n}}\mathrm{d}s'\mathrm{d}s\\
    &\quad+\frac{A^{\left(i,j\right)}\left(x\right)}{\rho^{\left(i\right)}\rho^{\left(j\right)}h_{n}^{2}}\int_{0}^{\left(p+1\right)h_{n}}\left(\mathbf{1}_{\left[0,\rho^{\left(i\right)}h_{n}\right]}\left(\left(p+1\right)h_{n}-s\right)\right)\left(\mathbf{1}_{\left[\left(p+1-\rho^{\left(j\right)}\right)h_{n},\left(p+1\right)h_{n}\right]}\left(s\right)\right)\\
    &\hspace{5cm}\int_{\left(p+1-\rho^{\left(j\right)}\right)h_{n}}^{s}s'\mathrm{d}s'\mathrm{d}s\\
    &\quad-\frac{A^{\left(i,j\right)}\left(x\right)}{\rho^{\left(i\right)}\rho^{\left(j\right)}h_{n}^{2}}\int_{0}^{\left(p+1\right)h_{n}}\left(\mathbf{1}_{\left[0,\rho^{\left(i\right)}h_{n}\right]}\left(\left(p+1\right)h_{n}-s\right)\right)\left(\mathbf{1}_{\left[\left(p-\rho^{\left(j\right)}\right)h_{n},\left(p+1\right)h_{n}\right]}\left(s\right)\right)\\
    &\hspace{5cm}\int_{\left(p-\rho^{\left(j\right)}\right)h_{n}}^{\min\left\{s,ph_{n}\right\}}s'\mathrm{d}s'\mathrm{d}s\\
    &=\frac{A^{\left(i,j\right)}\left(x\right)}{\rho^{\left(i\right)}\rho^{\left(j\right)}h_{n}^{2}}\int_{\left(p+1-\rho^{\left(i\right)}\right)h_{n}}^{\left(p+1\right)h_{n}}s\left(\left(p+1\right)h_{n}-\max\left\{s,\left(p+1-\rho^{\left(j\right)}\right)h_{n}\right\}\right)\mathrm{d}s\\
    &\quad-\frac{A^{\left(i,j\right)}\left(x\right)}{\rho^{\left(i\right)}\rho^{\left(j\right)}h_{n}^{2}}\mathbf{1}_{\left(1,\overline{\rho}\right]}\left(\rho^{\left(i\right)}\right)\int_{\left(p+1-\rho^{\left(i\right)}\right)h_{n}}^{ph_{n}}s\left(ph_{n}-\max\left\{s,\left(p-\rho^{\left(j\right)}\right)h_{n}\right\}\right)\mathrm{d}s\\
    &\quad+\frac{A^{\left(i,j\right)}\left(x\right)}{\rho^{\left(i\right)}\rho^{\left(j\right)}h_{n}^{2}}\int_{\max\left\{\left(p+1-\rho^{\left(i\right)}\right)h_{n},\left(p+1-\rho^{\left(j\right)}\right)h_{n}\right\}}^{\left(p+1\right)h_{n}}\left(\frac{s^{2}}{2}-\frac{\left(p+1-\rho^{\left(j\right)}\right)^{2}h_{n}^{2}}{2}\right)\mathrm{d}s\\
    &\quad-\frac{A^{\left(i,j\right)}\left(x\right)}{\rho^{\left(i\right)}\rho^{\left(j\right)}h_{n}^{2}}\int_{\max\left\{\left(p+1-\rho^{\left(i\right)}\right)h_{n},\left(p-\rho^{\left(j\right)}\right)h_{n}\right\}}^{\left(p+1\right)h_{n}}\left(\frac{\min\left\{s^{2},p^{2}h_{n}^{2}\right\}}{2}-\frac{\left(p-\rho^{\left(j\right)}\right)^{2}h_{n}^{2}}{2}\right)\mathrm{d}s\\
    &=\frac{A^{\left(i,j\right)}\left(x\right)}{\rho^{\left(i\right)}\rho^{\left(j\right)}h_{n}^{2}}\int_{\max\left\{\left(p+1-\rho^{\left(i\right)}\right)h_{n},\left(p+1-\rho^{\left(j\right)}\right)h_{n}\right\}}^{\left(p+1\right)h_{n}}s\left(\left(p+1\right)h_{n}-s\right)\mathrm{d}s\\
    &\quad+\frac{A^{\left(i,j\right)}\left(x\right)}{\rho^{\left(i\right)}\rho^{\left(j\right)}h_{n}^{2}}\mathbf{1}_{\left(\rho^{\left(j\right)},\overline{\rho}\right]}\left(\rho^{\left(i\right)}\right)\int_{\left(p+1-\rho^{\left(i\right)}\right)h_{n}}^{\left(p+1-\rho^{\left(j\right)}\right)h_{n}}s\left(\left(p+1\right)h_{n}-\left(p+1-\rho^{\left(j\right)}\right)h_{n}\right)\mathrm{d}s\\
    &\quad-\frac{A^{\left(i,j\right)}\left(x\right)}{\rho^{\left(i\right)}\rho^{\left(j\right)}h_{n}^{2}}\mathbf{1}_{\left(1,\overline{\rho}\right]}\left(\rho^{\left(i\right)}\right)\int_{\max\left\{\left(p+1-\rho^{\left(i\right)}\right)h_{n},\left(p-\rho^{\left(j\right)}\right)h_{n}\right\}}^{ph_{n}}s\left(ph_{n}-s\right)\mathrm{d}s\\
    &\quad-\frac{A^{\left(i,j\right)}\left(x\right)}{\rho^{\left(i\right)}\rho^{\left(j\right)}h_{n}^{2}}\mathbf{1}_{\left(\rho^{\left(j\right)}+1,\overline{\rho}\right]}\left(\rho^{\left(i\right)}\right)\int_{\left(p+1-\rho^{\left(i\right)}\right)h_{n}}^{\left(p-\rho^{\left(j\right)}\right)h_{n}}s\left(ph_{n}-\left(p-\rho^{\left(j\right)}\right)h_{n}\right)\mathrm{d}s\\
    &\quad+\frac{A^{\left(i,j\right)}\left(x\right)}{\rho^{\left(i\right)}\rho^{\left(j\right)}h_{n}^{2}}\int_{\max\left\{\left(p+1-\rho^{\left(i\right)}\right)h_{n},\left(p+1-\rho^{\left(j\right)}\right)h_{n}\right\}}^{\left(p+1\right)h_{n}}\left(\frac{s^{2}}{2}-\frac{\left(p+1-\rho^{\left(j\right)}\right)^{2}h_{n}^{2}}{2}\right)\mathrm{d}s\\
    &\quad-\frac{A^{\left(i,j\right)}\left(x\right)}{\rho^{\left(i\right)}\rho^{\left(j\right)}h_{n}^{2}}\mathbf{1}_{\left(1,\overline{\rho}\right]}\left(\rho^{\left(i\right)}\right)\int_{ph_{n}}^{\left(p+1\right)h_{n}}\left(\frac{p^{2}h_{n}^{2}}{2}-\frac{\left(p-\rho^{\left(j\right)}\right)^{2}h_{n}^{2}}{2}\right)\mathrm{d}s\\
    &\quad-\frac{A^{\left(i,j\right)}\left(x\right)}{\rho^{\left(i\right)}\rho^{\left(j\right)}h_{n}^{2}}\mathbf{1}_{\left(1,\overline{\rho}\right]}\left(\rho^{\left(i\right)}\right)\int_{\max\left\{\left(p+1-\rho^{\left(i\right)}\right)h_{n},\left(p-\rho^{\left(j\right)}\right)h_{n}\right\}}^{ph_{n}}\left(\frac{s^{2}}{2}-\frac{\left(p-\rho^{\left(j\right)}\right)^{2}h_{n}^{2}}{2}\right)\mathrm{d}s\\
    &\quad-\frac{A^{\left(i,j\right)}\left(x\right)}{\rho^{\left(i\right)}\rho^{\left(j\right)}h_{n}^{2}}\mathbf{1}_{\left(0,1\right]}\left(\rho^{\left(i\right)}\right)\int_{\left(p+1-\rho^{\left(i\right)}\right)h_{n}}^{\left(p+1\right)h_{n}}\left(\frac{p^{2}h_{n}^{2}}{2}-\frac{\left(p-\rho^{\left(j\right)}\right)^{2}h_{n}^{2}}{2}\right)\mathrm{d}s\\
    &=\frac{A^{\left(i,j\right)}\left(x\right)}{\rho^{\left(i\right)}\rho^{\left(j\right)}h_{n}^{2}}\left(\frac{\left(p+1\right)^{2}h_{n}^{2}}{2}-\frac{\max\left\{\left(p+1-\rho^{\left(i\right)}\right)^{2}h_{n}^{2},\left(p+1-\rho^{\left(j\right)}\right)^{2}h_{n}^{2}\right\}}{2}\right)\left(\left(p+1\right)h_{n}\right)\\
    &\quad-\frac{A^{\left(i,j\right)}\left(x\right)}{\rho^{\left(i\right)}\rho^{\left(j\right)}h_{n}^{2}}\left(\frac{\left(p+1\right)^{3}h_{n}^{3}}{3}-\frac{\max\left\{\left(p+1-\rho^{\left(i\right)}\right)^{3}h_{n}^{3},\left(p+1-\rho^{\left(j\right)}\right)^{3}h_{n}^{3}\right\}}{3}\right)\\
    &\quad+\frac{A^{\left(i,j\right)}\left(x\right)}{\rho^{\left(i\right)}\rho^{\left(j\right)}h_{n}^{2}}\mathbf{1}_{\left(\rho^{\left(j\right)},\overline{\rho}\right]}\left(\rho^{\left(i\right)}\right)\left(\frac{\left(p+1-\rho^{\left(j\right)}\right)^{2}h_{n}^{2}}{2}-\frac{\left(p+1-\rho^{\left(i\right)}\right)^{2}h_{n}^{2}}{2}\right)\left(\rho^{\left(j\right)}h_{n}\right)\\
    &\quad-\frac{A^{\left(i,j\right)}\left(x\right)}{\rho^{\left(i\right)}\rho^{\left(j\right)}h_{n}^{2}}\mathbf{1}_{\left(1,\overline{\rho}\right]}\left(\rho^{\left(i\right)}\right)\left(\frac{p^{2}h_{n}^{2}}{2}-\frac{\max\left\{\left(p+1-\rho^{\left(i\right)}\right)^{2}h_{n}^{2},\left(p-\rho^{\left(j\right)}\right)^{2}h_{n}^{2}\right\}}{2}\right)\left(ph_{n}\right)\\
    &\quad+\frac{A^{\left(i,j\right)}\left(x\right)}{\rho^{\left(i\right)}\rho^{\left(j\right)}h_{n}^{2}}\mathbf{1}_{\left(1,\overline{\rho}\right]}\left(\rho^{\left(i\right)}\right)\left(\frac{p^{3}h_{n}^{3}}{3}-\frac{\max\left\{\left(p+1-\rho^{\left(i\right)}\right)^{3}h_{n}^{3},\left(p-\rho^{\left(j\right)}\right)^{3}h_{n}^{3}\right\}}{3}\right)\\
    &\quad-\frac{A^{\left(i,j\right)}\left(x\right)}{\rho^{\left(i\right)}\rho^{\left(j\right)}h_{n}^{2}}\mathbf{1}_{\left(\rho^{\left(j\right)}+1,\overline{\rho}\right]}\left(\rho^{\left(i\right)}\right)\left(\frac{\left(p-\rho^{\left(j\right)}\right)^{2}h_{n}^{2}}{2}-\frac{\left(p+1-\rho^{\left(i\right)}\right)^{2}h_{n}^{2}}{2}\right)\left(\rho^{\left(j\right)}h_{n}\right)\\
    &\quad+\frac{A^{\left(i,j\right)}\left(x\right)}{\rho^{\left(i\right)}\rho^{\left(j\right)}h_{n}^{2}}\left(\frac{\left(p+1\right)^{3}h_{n}^{3}}{6}-\frac{\max\left\{\left(p+1-\rho^{\left(i\right)}\right)^{3}h_{n}^{3},\left(p+1-\rho^{\left(j\right)}\right)^{3}h_{n}^{3}\right\}}{6}\right)\\
    &\quad-\frac{A^{\left(i,j\right)}\left(x\right)}{\rho^{\left(i\right)}\rho^{\left(j\right)}h_{n}^{2}}\left(\min\left\{\rho^{\left(i\right)}h_{n},\rho^{\left(j\right)}h_{n}\right\}\right)\left(\frac{\left(p+1-\rho^{\left(j\right)}\right)^{2}h_{n}^{2}}{2}\right)\\
    &\quad-\frac{A^{\left(i,j\right)}\left(x\right)}{\rho^{\left(i\right)}\rho^{\left(j\right)}h_{n}^{2}}\mathbf{1}_{\left(1,\overline{\rho}\right]}\left(\rho^{\left(i\right)}\right)\left(\frac{p^{2}h_{n}^{2}}{2}-\frac{\left(p-\rho^{\left(j\right)}\right)^{2}h_{n}^{2}}{2}\right)h_{n}\\
    &\quad-\frac{A^{\left(i,j\right)}\left(x\right)}{\rho^{\left(i\right)}\rho^{\left(j\right)}h_{n}^{2}}\mathbf{1}_{\left(1,\overline{\rho}\right]}\left(\rho^{\left(i\right)}\right)\left(\frac{p^{3}h_{n}^{3}}{6}-\frac{\max\left\{\left(p+1-\rho^{\left(i\right)}\right)^{3}h_{n}^{3},\left(p-\rho^{\left(j\right)}\right)^{3}h_{n}^{3}\right\}}{6}\right)\\
    &\quad+\frac{A^{\left(i,j\right)}\left(x\right)}{\rho^{\left(i\right)}\rho^{\left(j\right)}h_{n}^{2}}\mathbf{1}_{\left(1,\overline{\rho}\right]}\left(\rho^{\left(i\right)}\right)\left(ph_{n}-\max\left\{\left(p+1-\rho^{\left(i\right)}\right)h_{n},\left(p-\rho^{\left(j\right)}\right)h_{n}\right\}\right)\left(\frac{\left(p-\rho^{\left(j\right)}\right)^{2}h_{n}^{2}}{2}\right)\\
    &\quad-\frac{A^{\left(i,j\right)}\left(x\right)}{\rho^{\left(i\right)}\rho^{\left(j\right)}h_{n}^{2}}\mathbf{1}_{\left(0,1\right]}\left(\rho^{\left(i\right)}\right)\left(\frac{p^{2}h_{n}^{2}}{2}-\frac{\left(p-\rho^{\left(j\right)}\right)^{2}h_{n}^{2}}{2}\right)\left(\rho^{\left(i\right)}h_{n}\right),
\end{align*}
and we consider the following cases: (i) $\rho^{\left(i\right)}>1$ and $\rho^{\left(i\right)}>\rho^{\left(j\right)}+1$; (ii) $\rho^{\left(i\right)}>1$ and $\rho^{\left(j\right)}<\rho^{\left(i\right)}\le \rho^{\left(j\right)}+1$;  (iii) $\rho^{\left(i\right)}>1$ and $\rho^{\left(i\right)}\le \rho^{\left(j\right)}$; (iv) $\rho^{\left(i\right)}\le 1$ and $\rho^{\left(j\right)}<\rho^{\left(i\right)}$; (v) $\rho^{\left(i\right)}\le 1$ and $\rho^{\left(i\right)}\le \rho^{\left(j\right)}$, and for the case (i),
\begin{align*}
    &{h_{n}
	\mathbf{K}^{\left(i,j\right)}\left(x|\rho\right)}\\
	&=
    \int_{0}^{\left(p+1\right)h_{n}}\int_{0}^{\left(p+1\right)h_{n}}A^{\left(i,j\right)}\left(x\right)\min\left\{s,s'\right\}\left(V_{\rho,h_{n}}^{\left(i,i\right)}\left(\left(p+1\right)h_{n}-s\right)\right)\\
    &\hspace{3cm}\left(V_{\rho,h_{n}}^{\left(j,j\right)}\left(\left(p+1\right)h_{n}-s'\right)-V_{\rho,h_{n}}^{\left(j,j\right)}\left(ph_{n}-s'\right)\right)\mathrm{d}s'\mathrm{d}s\\
    &=\frac{A^{\left(i,j\right)}\left(x\right)}{\rho^{\left(i\right)}\rho^{\left(j\right)}h_{n}^{2}}\left(\frac{\left(p+1\right)^{2}h_{n}^{2}}{2}-\frac{\left(p+1-\rho^{\left(j\right)}\right)^{2}h_{n}^{2}}{2}\right)\left(\left(p+1\right)h_{n}\right)\\
    &\quad-\frac{A^{\left(i,j\right)}\left(x\right)}{\rho^{\left(i\right)}\rho^{\left(j\right)}h_{n}^{2}}\left(\frac{\left(p+1\right)^{3}h_{n}^{3}}{3}-\frac{\left(p+1-\rho^{\left(j\right)}\right)^{3}h_{n}^{3}}{3}\right)\\
    &\quad+\frac{A^{\left(i,j\right)}\left(x\right)}{\rho^{\left(i\right)}\rho^{\left(j\right)}h_{n}^{2}}\mathbf{1}_{\left(\rho^{\left(j\right)},\overline{\rho}\right]}\left(\rho^{\left(i\right)}\right)\left(\frac{\left(p+1-\rho^{\left(j\right)}\right)^{2}h_{n}^{2}}{2}-\frac{\left(p+1-\rho^{\left(i\right)}\right)^{2}h_{n}^{2}}{2}\right)\left(\rho^{\left(j\right)}h_{n}\right)\\
    &\quad-\frac{A^{\left(i,j\right)}\left(x\right)}{\rho^{\left(i\right)}\rho^{\left(j\right)}h_{n}^{2}}\mathbf{1}_{\left(1,\overline{\rho}\right]}\left(\rho^{\left(i\right)}\right)\left(\frac{p^{2}h_{n}^{2}}{2}-\frac{\left(p-\rho^{\left(j\right)}\right)^{2}h_{n}^{2}}{2}\right)\left(ph_{n}\right)\\
    &\quad+\frac{A^{\left(i,j\right)}\left(x\right)}{\rho^{\left(i\right)}\rho^{\left(j\right)}h_{n}^{2}}\mathbf{1}_{\left(1,\overline{\rho}\right]}\left(\rho^{\left(i\right)}\right)\left(\frac{p^{3}h_{n}^{3}}{3}-\frac{\left(p-\rho^{\left(j\right)}\right)^{3}h_{n}^{3}}{3}\right)\\
    &\quad-\frac{A^{\left(i,j\right)}\left(x\right)}{\rho^{\left(i\right)}\rho^{\left(j\right)}h_{n}^{2}}\mathbf{1}_{\left(\rho^{\left(j\right)}+1,\overline{\rho}\right]}\left(\rho^{\left(i\right)}\right)\left(\frac{\left(p-\rho^{\left(j\right)}\right)^{2}h_{n}^{2}}{2}-\frac{\left(p+1-\rho^{\left(i\right)}\right)^{2}h_{n}^{2}}{2}\right)\left(\rho^{\left(j\right)}h_{n}\right)\\
    &\quad+\frac{A^{\left(i,j\right)}\left(x\right)}{\rho^{\left(i\right)}\rho^{\left(j\right)}h_{n}^{2}}\left(\frac{\left(p+1\right)^{3}h_{n}^{3}}{6}-\frac{\left(p+1-\rho^{\left(j\right)}\right)^{3}h_{n}^{3}}{6}\right)\\
    &\quad-\frac{A^{\left(i,j\right)}\left(x\right)}{\rho^{\left(i\right)}\rho^{\left(j\right)}h_{n}^{2}}\left(\rho^{\left(j\right)}h_{n}\right)\left(\frac{\left(p+1-\rho^{\left(j\right)}\right)^{2}h_{n}^{2}}{2}\right)\\
    &\quad-\frac{A^{\left(i,j\right)}\left(x\right)}{\rho^{\left(i\right)}\rho^{\left(j\right)}h_{n}^{2}}\mathbf{1}_{\left(1,\overline{\rho}\right]}\left(\rho^{\left(i\right)}\right)\left(\frac{p^{2}h_{n}^{2}}{2}-\frac{\left(p-\rho^{\left(j\right)}\right)^{2}h_{n}^{2}}{2}\right)h_{n}\\
    &\quad-\frac{A^{\left(i,j\right)}\left(x\right)}{\rho^{\left(i\right)}\rho^{\left(j\right)}h_{n}^{2}}\mathbf{1}_{\left(1,\overline{\rho}\right]}\left(\rho^{\left(i\right)}\right)\left(\frac{p^{3}h_{n}^{3}}{6}-\frac{\left(p-\rho^{\left(j\right)}\right)^{3}h_{n}^{3}}{6}\right)\\
    &\quad+\frac{A^{\left(i,j\right)}\left(x\right)}{\rho^{\left(i\right)}\rho^{\left(j\right)}h_{n}^{2}}\mathbf{1}_{\left(1,\overline{\rho}\right]}\left(\rho^{\left(i\right)}\right)\left(\rho^{\left(j\right)}h_{n}\right)\left(\frac{\left(p-\rho^{\left(j\right)}\right)^{2}h_{n}^{2}}{2}\right)\\
    &=\frac{h_{n}A^{\left(i,j\right)}\left(x\right)}{\rho^{\left(i\right)}\rho^{\left(j\right)}}\left(\frac{\left(\rho^{\left(j\right)}\right)^{2}}{2}+\frac{\rho^{\left(j\right)}}{2}\right),
\end{align*}
and for the case (ii),
\begin{align*}
    &{h_{n}
	\mathbf{K}^{\left(i,j\right)}\left(x|\rho\right)}\\
    &=
    \int_{0}^{\left(p+1\right)h_{n}}\int_{0}^{\left(p+1\right)h_{n}}A^{\left(i,j\right)}\left(x\right)\min\left\{s,s'\right\}\left(V_{\rho,h_{n}}^{\left(i,i\right)}\left(\left(p+1\right)h_{n}-s\right)\right)\\
    &\hspace{3cm}\left(V_{\rho,h_{n}}^{\left(j,j\right)}\left(\left(p+1\right)h_{n}-s'\right)-V_{\rho,h_{n}}^{\left(j,j\right)}\left(ph_{n}-s'\right)\right)\mathrm{d}s'\mathrm{d}s\\
    &=\frac{A^{\left(i,j\right)}\left(x\right)}{\rho^{\left(i\right)}\rho^{\left(j\right)}h_{n}^{2}}\left(\frac{\left(p+1\right)^{2}h_{n}^{2}}{2}-\frac{\left(p+1-\rho^{\left(j\right)}\right)^{2}h_{n}^{2}}{2}\right)\left(\left(p+1\right)h_{n}\right)\\
    &\quad-\frac{A^{\left(i,j\right)}\left(x\right)}{\rho^{\left(i\right)}\rho^{\left(j\right)}h_{n}^{2}}\left(\frac{\left(p+1\right)^{3}h_{n}^{3}}{3}-\frac{\left(p+1-\rho^{\left(j\right)}\right)^{3}h_{n}^{3}}{3}\right)\\
    &\quad+\frac{A^{\left(i,j\right)}\left(x\right)}{\rho^{\left(i\right)}\rho^{\left(j\right)}h_{n}^{2}}\mathbf{1}_{\left(\rho^{\left(j\right)},\overline{\rho}\right]}\left(\rho^{\left(i\right)}\right)\left(\frac{\left(p+1-\rho^{\left(j\right)}\right)^{2}h_{n}^{2}}{2}-\frac{\left(p+1-\rho^{\left(i\right)}\right)^{2}h_{n}^{2}}{2}\right)\left(\rho^{\left(j\right)}h_{n}\right)\\
    &\quad-\frac{A^{\left(i,j\right)}\left(x\right)}{\rho^{\left(i\right)}\rho^{\left(j\right)}h_{n}^{2}}\mathbf{1}_{\left(1,\overline{\rho}\right]}\left(\rho^{\left(i\right)}\right)\left(\frac{p^{2}h_{n}^{2}}{2}-\frac{\left(p+1-\rho^{\left(i\right)}\right)^{2}h_{n}^{2}}{2}\right)\left(ph_{n}\right)\\
    &\quad+\frac{A^{\left(i,j\right)}\left(x\right)}{\rho^{\left(i\right)}\rho^{\left(j\right)}h_{n}^{2}}\mathbf{1}_{\left(1,\overline{\rho}\right]}\left(\rho^{\left(i\right)}\right)\left(\frac{p^{3}h_{n}^{3}}{3}-\frac{\left(p+1-\rho^{\left(i\right)}\right)^{3}h_{n}^{3}}{3}\right)\\
    &\quad+\frac{A^{\left(i,j\right)}\left(x\right)}{\rho^{\left(i\right)}\rho^{\left(j\right)}h_{n}^{2}}\left(\frac{\left(p+1\right)^{3}h_{n}^{3}}{6}-\frac{\left(p+1-\rho^{\left(j\right)}\right)^{3}h_{n}^{3}}{6}\right)\\
    &\quad-\frac{A^{\left(i,j\right)}\left(x\right)}{\rho^{\left(i\right)}\rho^{\left(j\right)}h_{n}^{2}}\left(\rho^{\left(j\right)}h_{n}\right)\left(\frac{\left(p+1-\rho^{\left(j\right)}\right)^{2}h_{n}^{2}}{2}\right)\\
    &\quad-\frac{A^{\left(i,j\right)}\left(x\right)}{\rho^{\left(i\right)}\rho^{\left(j\right)}h_{n}^{2}}\mathbf{1}_{\left(1,\overline{\rho}\right]}\left(\rho^{\left(i\right)}\right)\left(\frac{p^{2}h_{n}^{2}}{2}-\frac{\left(p-\rho^{\left(j\right)}\right)^{2}h_{n}^{2}}{2}\right)h_{n}\\
    &\quad-\frac{A^{\left(i,j\right)}\left(x\right)}{\rho^{\left(i\right)}\rho^{\left(j\right)}h_{n}^{2}}\mathbf{1}_{\left(1,\overline{\rho}\right]}\left(\rho^{\left(i\right)}\right)\left(\frac{p^{3}h_{n}^{3}}{6}-\frac{\left(p+1-\rho^{\left(i\right)}\right)^{3}h_{n}^{3}}{6}\right)\\
    &\quad+\frac{A^{\left(i,j\right)}\left(x\right)}{\rho^{\left(i\right)}\rho^{\left(j\right)}h_{n}^{2}}\mathbf{1}_{\left(1,\overline{\rho}\right]}\left(\rho^{\left(i\right)}\right)\left(\left(\rho^{\left(i\right)}-1\right)h_{n}\right)\left(\frac{\left(p-\rho^{\left(j\right)}\right)^{2}h_{n}^{2}}{2}\right)\\
    &=\frac{h_{n}A^{\left(i,j\right)}\left(x\right)}{6\rho^{\left(i\right)}\rho^{\left(j\right)}}\\
    &\quad\times\left(\left(\rho^{\left(i\right)}\right)^3 - 3 \left(\rho^{\left(i\right)}\right)^2\rho^{\left(j\right)} - 3\left(\rho^{\left(i\right)}\right)^2 + 3 \rho^{\left(i\right)}\left(\rho^{\left(j\right)}\right)^2 + 6\rho^{\left(i\right)}\rho^{\left(j\right)} + 3 \rho^{\left(i\right)} - \left(\rho^{\left(j\right)}\right)^3 - 1\right)
\end{align*}
and for the case (iii),
\begin{align*}
    &{h_{n}
	\mathbf{K}^{\left(i,j\right)}\left(x|\rho\right)}\\
	&=\int_{0}^{\left(p+1\right)h_{n}}\int_{0}^{\left(p+1\right)h_{n}}A^{\left(i,j\right)}\left(x\right)\min\left\{s,s'\right\}\left(V_{\rho,h_{n}}^{\left(i,i\right)}\left(\left(p+1\right)h_{n}-s\right)\right)\\
    &\hspace{3cm}\left(V_{\rho,h_{n}}^{\left(j,j\right)}\left(\left(p+1\right)h_{n}-s'\right)-V_{\rho,h_{n}}^{\left(j,j\right)}\left(ph_{n}-s'\right)\right)\mathrm{d}s'\mathrm{d}s\\
    &=\frac{A^{\left(i,j\right)}\left(x\right)}{\rho^{\left(i\right)}\rho^{\left(j\right)}h_{n}^{2}}\left(\frac{\left(p+1\right)^{2}h_{n}^{2}}{2}-\frac{\left(p+1-\rho^{\left(i\right)}\right)^{2}h_{n}^{2}}{2}\right)\left(\left(p+1\right)h_{n}\right)\\
    &\quad-\frac{A^{\left(i,j\right)}\left(x\right)}{\rho^{\left(i\right)}\rho^{\left(j\right)}h_{n}^{2}}\left(\frac{\left(p+1\right)^{3}h_{n}^{3}}{3}-\frac{\left(p+1-\rho^{\left(i\right)}\right)^{3}h_{n}^{3}}{3}\right)\\
    &\quad-\frac{A^{\left(i,j\right)}\left(x\right)}{\rho^{\left(i\right)}\rho^{\left(j\right)}h_{n}^{2}}\mathbf{1}_{\left(1,\overline{\rho}\right]}\left(\rho^{\left(i\right)}\right)\left(\frac{p^{2}h_{n}^{2}}{2}-\frac{\left(p+1-\rho^{\left(i\right)}\right)^{2}h_{n}^{2}}{2}\right)\left(ph_{n}\right)\\
    &\quad+\frac{A^{\left(i,j\right)}\left(x\right)}{\rho^{\left(i\right)}\rho^{\left(j\right)}h_{n}^{2}}\mathbf{1}_{\left(1,\overline{\rho}\right]}\left(\rho^{\left(i\right)}\right)\left(\frac{p^{3}h_{n}^{3}}{3}-\frac{\left(p+1-\rho^{\left(i\right)}\right)^{3}h_{n}^{3}}{3}\right)\\
    &\quad+\frac{A^{\left(i,j\right)}\left(x\right)}{\rho^{\left(i\right)}\rho^{\left(j\right)}h_{n}^{2}}\left(\frac{\left(p+1\right)^{3}h_{n}^{3}}{6}-\frac{\left(p+1-\rho^{\left(i\right)}\right)^{3}h_{n}^{3}}{6}\right)\\
    &\quad-\frac{A^{\left(i,j\right)}\left(x\right)}{\rho^{\left(i\right)}\rho^{\left(j\right)}h_{n}^{2}}\left(\rho^{\left(i\right)}h_{n}\right)\left(\frac{\left(p+1-\rho^{\left(j\right)}\right)^{2}h_{n}^{2}}{2}\right)\\
    &\quad-\frac{A^{\left(i,j\right)}\left(x\right)}{\rho^{\left(i\right)}\rho^{\left(j\right)}h_{n}^{2}}\mathbf{1}_{\left(1,\overline{\rho}\right]}\left(\rho^{\left(i\right)}\right)\left(\frac{p^{2}h_{n}^{2}}{2}-\frac{\left(p-\rho^{\left(j\right)}\right)^{2}h_{n}^{2}}{2}\right)h_{n}\\
    &\quad-\frac{A^{\left(i,j\right)}\left(x\right)}{\rho^{\left(i\right)}\rho^{\left(j\right)}h_{n}^{2}}\mathbf{1}_{\left(1,\overline{\rho}\right]}\left(\rho^{\left(i\right)}\right)\left(\frac{p^{3}h_{n}^{3}}{6}-\frac{\left(p+1-\rho^{\left(i\right)}\right)^{3}h_{n}^{3}}{6}\right)\\
    &\quad+\frac{A^{\left(i,j\right)}\left(x\right)}{\rho^{\left(i\right)}\rho^{\left(j\right)}h_{n}^{2}}\mathbf{1}_{\left(1,\overline{\rho}\right]}\left(\rho^{\left(i\right)}\right)\left(\left(\rho^{\left(i\right)}-1\right)h_{n}\right)\left(\frac{\left(p-\rho^{\left(j\right)}\right)^{2}h_{n}^{2}}{2}\right)\\
    &=\frac{h_{n}A^{\left(i,j\right)}\left(x\right)}{\rho^{\left(i\right)}\rho^{\left(j\right)}}\left(-\frac{\left(\rho^{\left(i\right)}\right)^{2}}{2}+\rho^{\left(i\right)}\rho^{\left(j\right)}+\frac{\rho^{\left(i\right)}}{2}-\frac{1}{6}\right),
\end{align*}
and for the case (iv),
\begin{align*}
    &{h_{n}
	\mathbf{K}^{\left(i,j\right)}\left(x|\rho\right)}\\
	&=\int_{0}^{\left(p+1\right)h_{n}}\int_{0}^{\left(p+1\right)h_{n}}A^{\left(i,j\right)}\left(x\right)\min\left\{s,s'\right\}\left(V_{\rho,h_{n}}^{\left(i,i\right)}\left(\left(p+1\right)h_{n}-s\right)\right)\\
    &\hspace{3cm}\left(V_{\rho,h_{n}}^{\left(j,j\right)}\left(\left(p+1\right)h_{n}-s'\right)-V_{\rho,h_{n}}^{\left(j,j\right)}\left(ph_{n}-s'\right)\right)\mathrm{d}s'\mathrm{d}s\\
    &=\frac{A^{\left(i,j\right)}\left(x\right)}{\rho^{\left(i\right)}\rho^{\left(j\right)}h_{n}^{2}}\left(\frac{\left(p+1\right)^{2}h_{n}^{2}}{2}-\frac{\left(p+1-\rho^{\left(j\right)}\right)^{2}h_{n}^{2}}{2}\right)\left(\left(p+1\right)h_{n}\right)\\
    &\quad-\frac{A^{\left(i,j\right)}\left(x\right)}{\rho^{\left(i\right)}\rho^{\left(j\right)}h_{n}^{2}}\left(\frac{\left(p+1\right)^{3}h_{n}^{3}}{3}-\frac{\left(p+1-\rho^{\left(j\right)}\right)^{3}h_{n}^{3}}{3}\right)\\
    &\quad+\frac{A^{\left(i,j\right)}\left(x\right)}{\rho^{\left(i\right)}\rho^{\left(j\right)}h_{n}^{2}}\mathbf{1}_{\left(\rho^{\left(j\right)},\overline{\rho}\right]}\left(\rho^{\left(i\right)}\right)\left(\frac{\left(p+1-\rho^{\left(j\right)}\right)^{2}h_{n}^{2}}{2}-\frac{\left(p+1-\rho^{\left(i\right)}\right)^{2}h_{n}^{2}}{2}\right)\left(\rho^{\left(j\right)}h_{n}\right)\\
    &\quad+\frac{A^{\left(i,j\right)}\left(x\right)}{\rho^{\left(i\right)}\rho^{\left(j\right)}h_{n}^{2}}\left(\frac{\left(p+1\right)^{3}h_{n}^{3}}{6}-\frac{\left(p+1-\rho^{\left(j\right)}\right)^{3}h_{n}^{3}}{6}\right)\\
    &\quad-\frac{A^{\left(i,j\right)}\left(x\right)}{\rho^{\left(i\right)}\rho^{\left(j\right)}h_{n}^{2}}\left(\rho^{\left(j\right)}h_{n}\right)\left(\frac{\left(p+1-\rho^{\left(j\right)}\right)^{2}h_{n}^{2}}{2}\right)\\
    &\quad-\frac{A^{\left(i,j\right)}\left(x\right)}{\rho^{\left(i\right)}\rho^{\left(j\right)}h_{n}^{2}}\mathbf{1}_{\left(0,1\right]}\left(\rho^{\left(i\right)}\right)\left(\frac{p^{2}h_{n}^{2}}{2}-\frac{\left(p-\rho^{\left(j\right)}\right)^{2}h_{n}^{2}}{2}\right)\left(\rho^{\left(i\right)}h_{n}\right)\\
    &=\frac{h_{n}A^{\left(i,j\right)}\left(x\right)}{\rho^{\left(i\right)}\rho^{\left(j\right)}}\left(-\frac{\left(\rho^{\left(i\right)}\right)^{2}\rho^{\left(j\right)}}{2}+\frac{\rho^{\left(i\right)}\left(\rho^{\left(j\right)}\right)^{2}}{2}+\rho^{\left(i\right)}\rho^{\left(j\right)}-\frac{\left(\rho^{\left(j\right)}\right)^{3}}{6}\right),
\end{align*}
and for the case (v),
\begin{align*}
    &{h_{n}
	\mathbf{K}^{\left(i,j\right)}\left(x|\rho\right)}\\
	&=    \int_{0}^{\left(p+1\right)h_{n}}\int_{0}^{\left(p+1\right)h_{n}}A^{\left(i,j\right)}\left(x\right)\min\left\{s,s'\right\}\left(V_{\rho,h_{n}}^{\left(i,i\right)}\left(\left(p+1\right)h_{n}-s\right)\right)\\
    &\hspace{3cm}\left(V_{\rho,h_{n}}^{\left(j,j\right)}\left(\left(p+1\right)h_{n}-s'\right)-V_{\rho,h_{n}}^{\left(j,j\right)}\left(ph_{n}-s'\right)\right)\mathrm{d}s'\mathrm{d}s\\
    &=\frac{A^{\left(i,j\right)}\left(x\right)}{\rho^{\left(i\right)}\rho^{\left(j\right)}h_{n}^{2}}\left(\frac{\left(p+1\right)^{2}h_{n}^{2}}{2}-\frac{\left(p+1-\rho^{\left(i\right)}\right)^{2}h_{n}^{2}}{2}\right)\left(\left(p+1\right)h_{n}\right)\\
    &\quad-\frac{A^{\left(i,j\right)}\left(x\right)}{\rho^{\left(i\right)}\rho^{\left(j\right)}h_{n}^{2}}\left(\frac{\left(p+1\right)^{3}h_{n}^{3}}{3}-\frac{\left(p+1-\rho^{\left(i\right)}\right)^{3}h_{n}^{3}}{3}\right)\\
    &\quad+\frac{A^{\left(i,j\right)}\left(x\right)}{\rho^{\left(i\right)}\rho^{\left(j\right)}h_{n}^{2}}\left(\frac{\left(p+1\right)^{3}h_{n}^{3}}{6}-\frac{\left(p+1-\rho^{\left(i\right)}\right)^{3}h_{n}^{3}}{6}\right)\\
    &\quad-\frac{A^{\left(i,j\right)}\left(x\right)}{\rho^{\left(i\right)}\rho^{\left(j\right)}h_{n}^{2}}\left(\rho^{\left(i\right)}h_{n}\right)\left(\frac{\left(p+1-\rho^{\left(j\right)}\right)^{2}h_{n}^{2}}{2}\right)\\
    &\quad-\frac{A^{\left(i,j\right)}\left(x\right)}{\rho^{\left(i\right)}\rho^{\left(j\right)}h_{n}^{2}}\mathbf{1}_{\left(0,1\right]}\left(\rho^{\left(i\right)}\right)\left(\frac{p^{2}h_{n}^{2}}{2}-\frac{\left(p-\rho^{\left(j\right)}\right)^{2}h_{n}^{2}}{2}\right)\left(\rho^{\left(i\right)}h_{n}\right)\\
    &=\frac{h_{n}A^{\left(i,j\right)}\left(x\right)}{\rho^{\left(i\right)}\rho^{\left(j\right)}}\left(\rho^{\left(i\right)}\rho^{\left(j\right)}-\frac{\left(\rho^{\left(i\right)}\right)^{3}}{6}\right);
\end{align*}
therefore, we obtain
\begin{align*}
    &\mathbf{G}^{\left(i,j\right)}\left(x|\rho\right)+\mathbf{D}_{0}^{\left(i,j\right)}\left(x|\rho\right)\\
    &=\begin{cases}
    A^{\left(i,j\right)}\left(x\right)& \text{ if }\rho^{\left(i\right)}=0,\\
    A^{\left(i,j\right)}\left(x\right)\left(1-\frac{\rho^{\left(i\right)}}{2}\right)&\text{ if }\rho^{\left(i\right)}\in\left(0,1\right],\rho^{\left(j\right)}=0,\\
    A^{\left(i,j\right)}\left(x\right)\frac{1}{2\rho^{\left(i\right)}}&\text{ if }\rho^{\left(i\right)}\in\left(1,\overline{\rho}\right],\rho^{\left(j\right)}=0,\\
    A^{\left(i,j\right)}\left(x\right)\frac{\left(\rho^{\left(j\right)}\right)^{2}+\rho^{\left(j\right)}}{2\rho^{\left(i\right)}\rho^{\left(j\right)}}& \text{ if }\rho^{\left(i\right)}\in\left(\rho^{\left(j\right)}+1,\overline{\rho}\right],\rho^{\left(j\right)}>0,\\
    A^{\left(i,j\right)}\left(x\right)\frac{\left(\rho^{\left(i\right)}-\rho^{\left(j\right)}\right)^3 - 3\left(\rho^{\left(i\right)}\right)^2 + 6\rho^{\left(i\right)}\rho^{\left(j\right)} + 3 \rho^{\left(i\right)}- 1}{6\rho^{\left(i\right)}\rho^{\left(j\right)}}&\text{ if }\rho^{\left(i\right)}\in\left(1,\overline{\rho}\right], \rho^{\left(i\right)}\in \left(\rho^{\left(j\right)},\rho^{\left(j\right)}+1\right],\\
    A^{\left(i,j\right)}\left(x\right)\frac{-3\left(\rho^{\left(i\right)}\right)^{2}+6\rho^{\left(i\right)}\rho^{\left(j\right)}+3\rho^{\left(i\right)}-1}{6\rho^{\left(i\right)}\rho^{\left(j\right)}}&\text{ if }\rho^{\left(i\right)}\in\left(1,\overline{\rho}\right],\rho^{\left(i\right)}\le \rho^{\left(j\right)},\\
    A^{\left(i,j\right)}\left(x\right)\frac{-3\left(\rho^{\left(i\right)}\right)^{2}\rho^{\left(j\right)}+3\rho^{\left(i\right)}\left(\rho^{\left(j\right)}\right)^{2}+6\rho^{\left(i\right)}\rho^{\left(j\right)}-\left(\rho^{\left(j\right)}\right)^{3}}{6\rho^{\left(i\right)}\rho^{\left(j\right)}}&\text{ if }\rho^{\left(i\right)}\in\left(0,1\right],\rho^{\left(i\right)}>\rho^{\left(j\right)},\\
    A^{\left(i,j\right)}\left(x\right)\frac{6\rho^{\left(i\right)}\rho^{\left(j\right)}-\left(\rho^{\left(i\right)}\right)^{3}}{6\rho^{\left(i\right)}\rho^{\left(j\right)}}&\text{ if }\rho^{\left(i\right)}\in\left(0,1\right],\rho^{\left(i\right)}\le\rho^{\left(j\right)},
    \end{cases}
\end{align*}
and $\mathbf{G}\left(x|\rho\right)=\mathbb{G}\left(x|\rho\right)=\left.\mathbb{G}\left(x,\alpha|\rho\right)\right|_{\alpha=\alpha_{\star}}$.

\subsection*{Proof of the results in Section 3.1}
\begin{proof}[Proof of Lemma \ref{LemmaReducedQ}]
By following the proof of the Proposition \ref{propEmpQ}, it is sufficient to {evaluate}
\begin{align*}
    &\int_{0}^{\left(p+2\right)h_{n}}\int_{0}^{\left(p+2\right)h_{n}}A^{\left(i,i\right)}\left(x\right)\min\left\{s,s'\right\}\left(V_{\rho,h_{n}}^{\left(i,i\right)}\left(\left(p+2\right)h_{n}-s\right)-V_{\rho,h_{n}}^{\left(i,i\right)}\left(ph_{n}-s\right)\right)\\
    &\hspace{3cm}\left(V_{\rho,h_{n}}^{\left(i,i\right)}\left(\left(p+2\right)h_{n}-s'\right)-V_{\rho,h_{n}}^{\left(i,i\right)}\left(ph_{n}-s'\right)\right)\mathrm{d}s'\mathrm{d}s
\end{align*}
for the asymptotic behaviour of the reduced quadratic variation. If $\rho^{\left(i\right)}=0$,
\begin{align*}
    &\int_{0}^{\left(p+2\right)h_{n}}\int_{0}^{\left(p+2\right)h_{n}}A^{\left(i,i\right)}\left(x\right)\min\left\{s,s'\right\}\left(V_{\rho,h_{n}}^{\left(i,i\right)}\left(\left(p+2\right)h_{n}-s\right)-V_{\rho,h_{n}}^{\left(i,i\right)}\left(ph_{n}-s\right)\right)\\
    &\hspace{3cm}\left(V_{\rho,h_{n}}^{\left(i,i\right)}\left(\left(p+2\right)h_{n}-s'\right)-V_{\rho,h_{n}}^{\left(i,i\right)}\left(ph_{n}-s'\right)\right)\mathrm{d}s'\mathrm{d}s\\
    &=A^{\left(i,i\right)}\left(x\right)\int_{0}^{\left(p+2\right)h_{n}}\int_{0}^{\left(p+2\right)h_{n}}\min\left\{s,s'\right\}\left(\delta\left(\left(p+2\right)h_{n}-s\right)-\delta\left(ph_{n}-s\right)\right)\\
    &\hspace{3cm}\left(\delta\left(\left(p+2\right)h_{n}-s'\right)-\delta\left(ph_{n}-s'\right)\right)\mathrm{d}s'\mathrm{d}s\\
    &=A^{\left(i,i\right)}\left(x\right)\left(\left(p+2\right)h_{n}-2ph_{n}+ph_{n}\right)\\
    &=2h_{n}A^{\left(i,i\right)}\left(x\right),
\end{align*}
and if $\rho^{\left(i\right)}\in\left(0,\overline{\rho}\right]$,
\begin{align*}
    &\int_{0}^{\left(p+2\right)h_{n}}\int_{0}^{\left(p+2\right)h_{n}}A^{\left(i,i\right)}\left(x\right)\min\left\{s,s'\right\}\left(V_{\rho,h_{n}}^{\left(i,i\right)}\left(\left(p+2\right)h_{n}-s\right)-V_{\rho,h_{n}}^{\left(i,i\right)}\left(ph_{n}-s\right)\right)\\
    &\hspace{3cm}\left(V_{\rho,h_{n}}^{\left(i,i\right)}\left(\left(p+2\right)h_{n}-s'\right)-V_{\rho,h_{n}}^{\left(i,i\right)}\left(ph_{n}-s'\right)\right)\mathrm{d}s'\mathrm{d}s\\
    &=\frac{A^{\left(i,i\right)}\left(x\right)}{\left(\rho^{\left(i\right)}h_{n}\right)^{2}}\int_{0}^{\left(p+2\right)h_{n}}\int_{0}^{\left(p+2\right)h_{n}}\min\left\{s,s'\right\}\\
    &\hspace{3cm}\times\left(\mathbf{1}_{\left[0,\rho^{\left(i\right)}h_{n}\right]}\left(\left(p+2\right)h_{n}-s\right)-\mathbf{1}_{\left[0,\rho^{\left(i\right)}h_{n}\right]}\left(ph_{n}-s\right)\right)\\
    &\hspace{3cm}\times\left(\mathbf{1}_{\left[0,\rho^{\left(i\right)}h_{n}\right]}\left(\left(p+2\right)h_{n}-s'\right)-\mathbf{1}_{\left[0,\rho^{\left(i\right)}h_{n}\right]}\left(ph_{n}-s'\right)\right)\mathrm{d}s'\mathrm{d}s\\
    &=\frac{A^{\left(i,i\right)}\left(x\right)}{\left(\rho^{\left(i\right)}h_{n}\right)^{2}}\int_{\left(p+2-\rho^{\left(i\right)}\right)h_{n}}^{\left(p+2\right)h_{n}}\int_{\left(p+2-\rho^{\left(i\right)}\right)h_{n}}^{\left(p+2\right)h_{n}}\min\left\{s,s'\right\}\mathrm{d}s'\mathrm{d}s\\
    &\quad-\frac{2A^{\left(i,i\right)}\left(x\right)}{\left(\rho^{\left(i\right)}h_{n}\right)^{2}}\int_{\left(p+2-\rho^{\left(i\right)}\right)h_{n}}^{\left(p+2\right)h_{n}}\int_{\left(p-\rho^{\left(i\right)}\right)h_{n}}^{ph_{n}}\min\left\{s,s'\right\}\mathrm{d}s'\mathrm{d}s\\
    &\quad+\frac{A^{\left(i,i\right)}\left(x\right)}{\left(\rho^{\left(i\right)}h_{n}\right)^{2}}\int_{\left(p-\rho^{\left(i\right)}\right)h_{n}}^{ph_{n}}\int_{\left(p-\rho^{\left(i\right)}\right)h_{n}}^{ph_{n}}\min\left\{s,s'\right\}\mathrm{d}s'\mathrm{d}s\\
    &=\frac{A^{\left(i,i\right)}\left(x\right)}{\left(\rho^{\left(i\right)}h_{n}\right)^{2}}\int_{\left(p+2-\rho^{\left(i\right)}\right)h_{n}}^{\left(p+2\right)h_{n}}\left(\int_{s}^{\left(p+2\right)h_{n}}s\mathrm{d}s'+\int_{\left(p+2-\rho^{\left(i\right)}\right)h_{n}}^{s}s'\mathrm{d}s'\right)\mathrm{d}s\\
    &\quad-\frac{2A^{\left(i,i\right)}\left(x\right)}{\left(\rho^{\left(i\right)}h_{n}\right)^{2}}\mathbf{1}_{\left(2,\overline{\rho}\right]}\left(\rho^{\left(i\right)}\right)\int_{ph_{n}}^{\left(p+2\right)h_{n}}\int_{\left(p-\rho^{\left(i\right)}\right)h_{n}}^{ph_{n}}s'\mathrm{d}s'\mathrm{d}s\\
    &\quad-\frac{2A^{\left(i,i\right)}\left(x\right)}{\left(\rho^{\left(i\right)}h_{n}\right)^{2}}\mathbf{1}_{\left(2,\overline{\rho}\right]}\left(\rho^{\left(i\right)}\right)\int_{\left(p+2-\rho^{\left(i\right)}\right)h_{n}}^{ph_{n}}\left(\int_{s}^{ph_{n}}s\mathrm{d}s'+\int_{\left(p-\rho^{\left(i\right)}\right)h_{n}}^{s}s'\mathrm{d}s'\right)\mathrm{d}s\\
    &\quad-\frac{2A^{\left(i,i\right)}\left(x\right)}{\left(\rho^{\left(i\right)}h_{n}\right)^{2}}\mathbf{1}_{\left(0,2\right]}\left(\rho^{\left(i\right)}\right)\int_{\left(p+2-\rho^{\left(i\right)}\right)h_{n}}^{\left(p+2\right)h_{n}}\int_{\left(p-\rho^{\left(i\right)}\right)h_{n}}^{ph_{n}}s'\mathrm{d}s'\mathrm{d}s\\
    &\quad+\frac{A^{\left(i,i\right)}\left(x\right)}{\left(\rho^{\left(i\right)}h_{n}\right)^{2}}\int_{\left(p-\rho^{\left(i\right)}\right)h_{n}}^{ph_{n}}\left(\int_{s}^{ph_{n}}s\mathrm{d}s'+\int_{\left(p-\rho^{\left(i\right)}\right)h_{n}}^{s}s'\mathrm{d}s'\right)\mathrm{d}s\\
    &=\frac{A^{\left(i,i\right)}\left(x\right)}{\left(\rho^{\left(i\right)}h_{n}\right)^{2}}\int_{\left(p+2-\rho^{\left(i\right)}\right)h_{n}}^{\left(p+2\right)h_{n}}\left(\left(\left(p+2\right)h_{n}s-s^{2}\right)+\left(\frac{s^{2}}{2}-\frac{\left(p+2-\rho^{\left(i\right)}\right)^{2}h_{n}^{2}}{2}\right)\right)\mathrm{d}s\\
    &\quad-\frac{A^{\left(i,i\right)}\left(x\right)}{\left(\rho^{\left(i\right)}h_{n}\right)^{2}}\mathbf{1}_{\left(2,\overline{\rho}\right]}\left(\rho^{\left(i\right)}\right)\int_{ph_{n}}^{\left(p+2\right)h_{n}}\left(p^{2}h_{n}^{2}-\left(p-\rho^{\left(i\right)}\right)^{2}h_{n}^{2}\right)\mathrm{d}s\\
    &\quad-\frac{2A^{\left(i,i\right)}\left(x\right)}{\left(\rho^{\left(i\right)}h_{n}\right)^{2}}\mathbf{1}_{\left(2,\overline{\rho}\right]}\left(\rho^{\left(i\right)}\right)\int_{\left(p+2-\rho^{\left(i\right)}\right)h_{n}}^{ph_{n}}\left(\left(ph_{n}s-s^{2}\right)+\left(\frac{s^{2}}{2}-\frac{\left(p-\rho^{\left(i\right)}\right)^{2}h_{n}^{2}}{2}\right)\right)\mathrm{d}s\\
    &\quad-\frac{2A^{\left(i,i\right)}\left(x\right)}{\left(\rho^{\left(i\right)}h_{n}\right)^{2}}\mathbf{1}_{\left(0,2\right]}\left(\rho^{\left(i\right)}\right)\int_{\left(p+2-\rho^{\left(i\right)}\right)h_{n}}^{\left(p+2\right)h_{n}}\left(\frac{p^{2}h_{n}^{2}}{2}-\frac{\left(p-\rho^{\left(i\right)}\right)^{2}h_{n}^{2}}{2}\right)\mathrm{d}s\\
    &\quad+\frac{A^{\left(i,i\right)}\left(x\right)}{\left(\rho^{\left(i\right)}h_{n}\right)^{2}}\int_{\left(p-\rho^{\left(i\right)}\right)h_{n}}^{ph_{n}}\left(\left(ph_{n}s-s^{2}\right)+\left(\frac{s^{2}}{2}-\frac{\left(p-\rho^{\left(i\right)}\right)^{2}h_{n}^{2}}{2}\right)\right)\mathrm{d}s\\
    &=\frac{A^{\left(i,i\right)}\left(x\right)}{\left(\rho^{\left(i\right)}h_{n}\right)^{2}}\int_{\left(p+2-\rho^{\left(i\right)}\right)h_{n}}^{\left(p+2\right)h_{n}}\left(-\frac{s^{2}}{2}+\left(p+2\right)h_{n}s-\frac{\left(p+2-\rho^{\left(i\right)}\right)^{2}h_{n}^{2}}{2}\right)\mathrm{d}s\\
    &\quad-\frac{A^{\left(i,i\right)}\left(x\right)}{\left(\rho^{\left(i\right)}h_{n}\right)^{2}}\mathbf{1}_{\left(2,\overline{\rho}\right]}\left(\rho^{\left(i\right)}\right)\left(p^{2}h_{n}^{2}-\left(p-\rho^{\left(i\right)}\right)^{2}h_{n}^{2}\right)2h_{n}\\
    &\quad-\frac{2A^{\left(i,i\right)}\left(x\right)}{\left(\rho^{\left(i\right)}h_{n}\right)^{2}}\mathbf{1}_{\left(2,\overline{\rho}\right]}\left(\rho^{\left(i\right)}\right)\int_{\left(p+2-\rho^{\left(i\right)}\right)h_{n}}^{ph_{n}}\left(-\frac{s^{2}}{2}+ph_{n}s-\frac{\left(p-\rho^{\left(i\right)}\right)^{2}h_{n}^{2}}{2}\right)\mathrm{d}s\\
    &\quad-\frac{2A^{\left(i,i\right)}\left(x\right)}{\left(\rho^{\left(i\right)}h_{n}\right)^{2}}\mathbf{1}_{\left(0,2\right]}\left(\rho^{\left(i\right)}\right)\left(\frac{p^{2}h_{n}^{2}}{2}-\frac{\left(p-\rho^{\left(i\right)}\right)^{2}h_{n}^{2}}{2}\right)\rho^{\left(i\right)}h_{n}\\
    &\quad+\frac{A^{\left(i,i\right)}\left(x\right)}{\left(\rho^{\left(i\right)}h_{n}\right)^{2}}\int_{\left(p-\rho^{\left(i\right)}\right)h_{n}}^{ph_{n}}\left(-\frac{s^{2}}{2}+ph_{n}s-\frac{\left(p-\rho^{\left(i\right)}\right)^{2}h_{n}^{2}}{2}\right)\mathrm{d}s\\
    &=-\frac{A^{\left(i,i\right)}\left(x\right)}{\left(\rho^{\left(i\right)}h_{n}\right)^{2}}\left(\frac{\left(p+2\right)^{3}h_{n}^{3}}{6}-\frac{\left(p+2-\rho^{\left(i\right)}\right)^{3}h_{n}^{3}}{6}\right)\\
    &\quad+\frac{A^{\left(i,i\right)}\left(x\right)}{\left(\rho^{\left(i\right)}h_{n}\right)^{2}}\left(\frac{\left(p+2\right)^{2}h_{n}^{2}}{2}-\frac{\left(p+2-\rho^{\left(i\right)}\right)^{2}h_{n}^{2}}{2}\right)\left(p+2\right)h_{n}\\
    &\quad-\frac{A^{\left(i,i\right)}\left(x\right)}{\left(\rho^{\left(i\right)}h_{n}\right)^{2}}\frac{\left(p+2-\rho^{\left(i\right)}\right)^{2}\rho^{\left(i\right)}h_{n}^{3}}{2}\\
    &\quad-\frac{A^{\left(i,i\right)}\left(x\right)}{\left(\rho^{\left(i\right)}h_{n}\right)^{2}}\mathbf{1}_{\left(2,\overline{\rho}\right]}\left(\rho^{\left(i\right)}\right)\left(p^{2}h_{n}^{2}-\left(p-\rho^{\left(i\right)}\right)^{2}h_{n}^{2}\right)2h_{n}\\
    &\quad+\frac{A^{\left(i,i\right)}\left(x\right)}{\left(\rho^{\left(i\right)}h_{n}\right)^{2}}\mathbf{1}_{\left(2,\overline{\rho}\right]}\left(\rho^{\left(i\right)}\right)\left(\frac{p^{3}h_{n}^{3}}{3}-\frac{\left(p+2-\rho^{\left(i\right)}\right)^{3}h_{n}^{3}}{3}\right)\\
    &\quad-\frac{A^{\left(i,i\right)}\left(x\right)}{\left(\rho^{\left(i\right)}h_{n}\right)^{2}}\mathbf{1}_{\left(2,\overline{\rho}\right]}\left(\rho^{\left(i\right)}\right)\left(p^{2}h_{n}^{2}-\left(p+2-\rho^{\left(i\right)}\right)^{2}h_{n}^{2}\right)ph_{n}\\
    &\quad+\frac{A^{\left(i,i\right)}\left(x\right)}{\left(\rho^{\left(i\right)}h_{n}\right)^{2}}\mathbf{1}_{\left(2,\overline{\rho}\right]}\left(\rho^{\left(i\right)}\right)\left(p-\rho^{\left(i\right)}\right)^{2}\left(\rho^{\left(i\right)}-2\right)h_{n}^{3}\\
    &\quad-\frac{A^{\left(i,i\right)}\left(x\right)}{\left(\rho^{\left(i\right)}h_{n}\right)^{2}}\mathbf{1}_{\left(0,2\right]}\left(\rho^{\left(i\right)}\right)\left(p^{2}\rho^{\left(i\right)}h_{n}^{3}-\left(p-\rho^{\left(i\right)}\right)^{2}\rho^{\left(i\right)}h_{n}^{3}\right)\\
    &\quad-\frac{A^{\left(i,i\right)}\left(x\right)}{\left(\rho^{\left(i\right)}h_{n}\right)^{2}}\left(\frac{p^{3}h_{n}^{3}}{6}-\frac{\left(p-\rho^{\left(i\right)}\right)^{3}h_{n}^{3}}{6}\right)\\
    &\quad+\frac{A^{\left(i,i\right)}\left(x\right)}{\left(\rho^{\left(i\right)}h_{n}\right)^{2}}\left(\frac{p^{2}h_{n}^{2}}{2}-\frac{\left(p-\rho^{\left(i\right)}\right)^{2}h_{n}^{2}}{2}\right)ph_{n}\\
    &\quad-\frac{A^{\left(i,i\right)}\left(x\right)}{\left(\rho^{\left(i\right)}h_{n}\right)^{2}}\frac{\left(p-\rho^{\left(i\right)}\right)^{2}\rho^{\left(i\right)}h_{n}^{3}}{2}\\
    &=\begin{cases}
    2h_{n}A^{\left(i,i\right)}\left(x\right)\left(1-\frac{\rho^{\left(i\right)}}{6}\right)&\text{ if }\rho^{\left(i\right)}\in\left(0,2\right],\\
    2h_{n}A^{\left(i,i\right)}\left(x\right)\left(\frac{2}{\rho^{\left(i\right)}}-\frac{4}{3\left(\rho^{\left(i\right)}\right)^{2}}\right)&\text{ if }\rho^{\left(i\right)}\in\left(2,\overline{\rho}\right].
    \end{cases}
\end{align*}
Hence, we obtain the proof.
\end{proof}

\begin{proof}[Proof of Lemma \ref{LemmaFunctionR}]
Continuity is obvious, and monotonicity is obtained as follows: if $\rho^{\left(i\right)}\in\left(0,1\right]$,
\begin{align*}
    \frac{\mathrm{d}}{\mathrm{d}\rho^{\left(i\right)}}\left(6-2\rho^{\left(i\right)}\right)\left(6-\rho^{\left(i\right)}\right)^{-1}&=\left(\left(-2\right)\left(6-\rho^{\left(i\right)}\right)-\left(6-2\rho^{\left(i\right)}\right)\left(-1\right)\right)\left(6-\rho^{\left(i\right)}\right)^{-2}\\
    &=\left(-12\right)\left(6-\rho^{\left(i\right)}\right)^{-2}<0,
\end{align*}
and if $\rho^{\left(i\right)}\in\left(1,2\right]$,
\begin{align*}
    &\frac{\mathrm{d}}{\mathrm{d}\rho^{\left(i\right)}}\left(6\rho^{\left(i\right)}-2\right)\left(6\left(\rho^{\left(i\right)}\right)^{2}-\left(\rho^{\left(i\right)}\right)^{3}\right)^{-1}\\
    &=\left(6\left(6\left(\rho^{\left(i\right)}\right)^{2}-\left(\rho^{\left(i\right)}\right)^{3}\right)-\left(6\rho^{\left(i\right)}-2\right)\left(12\rho^{\left(i\right)}-3\left(\rho^{\left(i\right)}\right)^{2}\right)\right)\left(6\left(\rho^{\left(i\right)}\right)^{2}-\left(\rho^{\left(i\right)}\right)^{3}\right)^{-2}\\
    &=6\rho^{\left(i\right)}\left(-7\rho^{\left(i\right)}+4+2\left(\rho^{\left(i\right)}\right)^{2}\right)\left(6\left(\rho^{\left(i\right)}\right)^{2}-\left(\rho^{\left(i\right)}\right)^{3}\right)^{-2}<0,
\end{align*}
and if $\rho^{\left(i\right)}\in\left(2,\overline{\rho}\right]$,
\begin{align*}
    \frac{\mathrm{d}}{\mathrm{d}\rho^{\left(i\right)}}\left(3\rho^{\left(i\right)}-1\right)\left(6\rho^{\left(i\right)}-4\right)^{-1}&=\left(-18\right)\left(6\rho^{\left(i\right)}-4\right)^{-2}<0.
\end{align*}
The inverse can be obtained directly.
\end{proof}

\subsection*{Proofs of the results in Section 3.2}
\begin{proof}[Proof of Theorem \ref{thmTestH0}]
    We can clearly prove the result by using Lemma 7 in \citet{Kessler-1997}, Proposition 7 in \citet{Nakakita-Uchida-2019a}, and Slutsky's theorem.
\end{proof}

\begin{proof}[Proof of Theorem \ref{thmTestH1}]
    By Lemma \ref{LemmaReducedQ}, there exists a number $\ell<0$ such that
    \begin{align*}
        \frac{1}{nh_{n}}\sum_{k=1}^{n}\left(\overline{X}_{kh_{n},n}^{\left(i\right)}-\overline{X}_{\left(k-1\right)h_{n},n}^{\left(i\right)}\right)^{2}-\frac{1}{nh_{n}}\sum_{2\le 2k\le n}\left(\overline{X}_{2kh_{n},n}^{\left(i\right)}-\overline{X}_{\left(2k-2\right)h_{n},n}^{\left(i\right)}\right)^{2}\to^{P}\ell<0,
    \end{align*}
    and hence it is sufficient to show that
    \begin{align*}
        \sup_{n\in\mathbf{N}}\mathbf{E}\left[\frac{2}{3nh_{n}^{2}}\sum_{k=1}^{n}\left(\overline{X}_{kh_{n},n}^{\left(i\right)}-\overline{X}_{\left(k-1\right)h_{n},n}^{\left(i\right)}\right)^{4}\right]<\infty;
    \end{align*}
    and it is obvious that
    \begin{align*}
        \left(\overline{X}_{kh_{n},n}^{\left(i\right)}-\overline{X}_{\left(k-1\right)h_{n},n}^{\left(i\right)}\right)^{4}\le C\left(\overline{X}_{kh_{n},n}^{\left(i\right)}-X_{\left(k-\overline{\rho}-1\right)h_{n}}\right)^{4}+C\left(\overline{X}_{\left(k-1\right)h_{n},n}^{\left(i\right)}-X_{\left(k-\overline{\rho}-1\right)h_{n}}\right)^{4}
    \end{align*}
    and
    \begin{align*}
        &\mathbf{E}\left[\left.\left(\overline{X}_{kh_{n},n}^{\left(i\right)}-X_{\left(k-\overline{\rho}-1\right)h_{n}}\right)^{4}\right|\mathcal{F}_{\left(k-\overline{\rho}-1\right)h_{n}}\right]\\
        &=\mathbf{E}\left[\left.\left(\frac{1}{\rho_{\star}^{\left(i\right)}h_{n}}\int_{\left(k-\rho_{\star}^{\left(i\right)}\right)h_{n}}^{kh_{n}}\left(X_{s}^{\left(i\right)}-X_{\left(k-\overline{\rho}-1\right)h_{n}}\right)\mathrm{d}s\right)^{4}\right|\mathcal{F}_{\left(k-\overline{\rho}-1\right)h_{n}}\right]\\
        &\le \mathbf{E}\left[\left.\left(\frac{1}{\rho_{\star}^{\left(i\right)}h_{n}}\int_{\left(k-\rho_{\star}^{\left(i\right)}\right)h_{n}}^{kh_{n}}\left|X_{s}^{\left(i\right)}-X_{\left(k-\overline{\rho}-1\right)h_{n}}\right|\mathrm{d}s\right)^{4}\right|\mathcal{F}_{\left(k-\overline{\rho}-1\right)h_{n}}\right]\\
        &\le \mathbf{E}\left[\left.\left(\frac{1}{\rho_{\star}^{\left(i\right)}h_{n}}\int_{\left(k-\rho_{\star}^{\left(i\right)}\right)h_{n}}^{kh_{n}}\sup_{s'\in\left[\left(k-\rho_{\star}^{\left(i\right)}\right)h_{n},kh_{n}\right]}\left|X_{s'}^{\left(i\right)}-X_{\left(k-\overline{\rho}-1\right)h_{n}}\right|\mathrm{d}s\right)^{4}\right|\mathcal{F}_{\left(k-\overline{\rho}-1\right)h_{n}}\right]\\
        &= \mathbf{E}\left[\left.\sup_{s\in\left[\left(k-\rho_{\star}^{\left(i\right)}\right)h_{n},kh_{n}\right]}\left|X_{s}^{\left(i\right)}-X_{\left(k-\overline{\rho}-1\right)h_{n}}\right|^{4}\right|\mathcal{F}_{\left(k-\overline{\rho}-1\right)h_{n}}\right]\\
        &\le Ch_{n}\left(1+\left|X_{\left(k-\overline{\rho}-1\right)h_{n}}\right|\right)^{C}
    \end{align*}
    by Proposition A in \citet{Gloter-2000}, and a parallel result holds for $\left(\overline{X}_{\left(k-1\right)h_{n},n}^{\left(i\right)}-X_{\left(k-\overline{\rho}-1\right)h_{n}}\right)^{4}$. Hence we obtain the result.
\end{proof}

\subsection*{Proof of the results in Section 4}
\begin{proof}[Proof of Theorem \ref{thmThetaEstimate}]
We only deal with the case where $\rho_{\star}$ is unknown because the discussion for the case where $\rho_{\star}$ is known is parallel.
First of all, we prove the consistency of $\hat{\alpha}_{n}$. We obtain that
\begin{align*}
    &\left|\frac{1}{n}\mathbb{H}_{1,n}\left(\alpha|\hat{\rho}_{n}\right)-\frac{1}{n}\mathbb{H}_{1,n}\left(\alpha|\rho_{\star}\right)\right|\\
    &=\left|-\frac{1}{n}\sum_{k=1}^{n}\left\|\frac{1}{h_{n}}\left(\overline{X}_{kh_{n},n}-\overline{X}_{\left(k-1\right)h_{n},n}\right)^{\otimes2}-\mathbb{G}\left(\overline{X}_{\left(k-1\right)h_{n},n},\alpha|\hat{\rho}_{n}\right)\right\|^{2}\right.\\
    &\qquad\left.+\frac{1}{n}\sum_{k=1}^{n}\left\|\frac{1}{h_{n}}\left(\overline{X}_{kh_{n},n}-\overline{X}_{\left(k-1\right)h_{n},n}\right)^{\otimes2}-\mathbb{G}\left(\overline{X}_{\left(k-1\right)h_{n},n},\alpha|\rho_{\star}\right)\right\|^{2}\right|\\
    &\le \left|\frac{2}{nh_{n}}\sum_{k=1}^{n}\left(\mathbb{G}\left(\overline{X}_{\left(k-1\right)h_{n},n},\alpha|\hat{\rho}_{n}\right)-\mathbb{G}\left(\overline{X}_{\left(k-1\right)h_{n},n},\alpha|\rho_{\star}\right)\right)\left[\left(\overline{X}_{kh_{n},n}-\overline{X}_{\left(k-1\right)h_{n},n}\right)^{\otimes2}\right]\right|\\
    &\quad+\left|\frac{1}{n}\sum_{k=1}^{n}\left(\left\|\mathbb{G}\left(\overline{X}_{\left(k-1\right)h_{n},n},\alpha|\hat{\rho}_{n}\right)\right\|^{2}-\left\|\mathbb{G}\left(\overline{X}_{\left(k-1\right)h_{n},n},\alpha|\rho_{\star}\right)\right\|^{2}\right)\right|\\
    &\le \frac{2}{nh_{n}}\sum_{k=1}^{n}\left|\overline{X}_{kh_{n},n}-\overline{X}_{\left(k-1\right)h_{n},n}\right|^{2}\\
    &\hspace{2cm}\times\sum_{i=1}^{d}\sum_{j=1}^{d}\left|A^{\left(i,j\right)}\left(\overline{X}_{\left(k-1\right)h_{n},n},\alpha\right)\right|\left|f_{\mathbb{G}}\left(\hat{\rho}_{n}^{\left(i\right)},\hat{\rho}_{n}^{\left(j\right)}\right)-f_{\mathbb{G}}\left(\rho_{\star}^{\left(i\right)},\rho_{\star}^{\left(j\right)}\right)\right|\\
    &\quad+\frac{1}{n}\sum_{k=1}^{n}\sum_{i=1}^{d}\sum_{j=1}^{d}\left|A^{\left(i,j\right)}\left(\overline{X}_{\left(k-1\right)h_{n},n},\alpha\right)\right|^{2}\left|f_{\mathbb{G}}^{2}\left(\hat{\rho}_{n}^{\left(i\right)},\hat{\rho}_{n}^{\left(j\right)}\right)-f_{\mathbb{G}}^{2}\left(\rho_{\star}^{\left(i\right)},\rho_{\star}^{\left(j\right)}\right)\right|\\
    &\le \frac{C}{nh_{n}}\sum_{k=1}^{n}\left(1+\left|\overline{X}_{\left(k-1\right)h_{n},n}\right|\right)^{C}\left|\overline{X}_{kh_{n},n}-\overline{X}_{\left(k-1\right)h_{n},n}\right|^{2}\\
    &\hspace{2cm}\times\sum_{i=1}^{d}\sum_{j=1}^{d}\left|f_{\mathbb{G}}\left(\hat{\rho}_{n}^{\left(i\right)},\hat{\rho}_{n}^{\left(j\right)}\right)-f_{\mathbb{G}}\left(\rho_{\star}^{\left(i\right)},\rho_{\star}^{\left(j\right)}\right)\right|\\
    &\quad+\frac{C}{n}\sum_{k=1}^{n}\left(1+\left|\overline{X}_{\left(k-1\right)h_{n},n}\right|\right)^{C}\sum_{i=1}^{d}\sum_{j=1}^{d}\left|f_{\mathbb{G}}^{2}\left(\hat{\rho}_{n}^{\left(i\right)},\hat{\rho}_{n}^{\left(j\right)}\right)-f_{\mathbb{G}}^{2}\left(\rho_{\star}^{\left(i\right)},\rho_{\star}^{\left(j\right)}\right)\right|\\
    &\to^{P}0\text{ uniformly in }\alpha,
\end{align*}
because continuous mapping theorem holds.
Therefore, it follows from Proposition \ref{propEmpMean} and Proposition \ref{propEmpQ} that
\begin{align*}
    &\frac{1}{n}\mathbb{H}_{1,n}\left(\alpha|\hat{\rho}_{n}\right)-\frac{1}{n}\mathbb{H}_{1,n}\left(\alpha_{\star}|\rho_{\star}\right)\\
    &=\frac{2}{nh_{n}}\sum_{k=1}^{n}\mathbb{G}\left(\overline{X}_{\left(k-1\right)h_{n},n},\alpha|\rho_{\star}\right)\left[\left(\overline{X}_{kh_{n},n}-\overline{X}_{\left(k-1\right)h_{n},n}\right)^{\otimes2}\right]\\
    &\quad-\frac{1}{n}\sum_{k=1}^{n}\left\|\mathbb{G}\left(\overline{X}_{\left(k-1\right)h_{n},n},\alpha|\rho_{\star}\right)\right\|^{2}\\
    &\quad-\frac{2}{nh_{n}}\sum_{k=1}^{n}\mathbb{G}\left(\overline{X}_{\left(k-1\right)h_{n},n},\alpha_{\star}|\rho_{\star}\right)\left[\left(\overline{X}_{kh_{n},n}-\overline{X}_{\left(k-1\right)h_{n},n}\right)^{\otimes2}\right]\\
    &\quad+\frac{1}{n}\sum_{k=1}^{n}\left\|\mathbb{G}\left(\overline{X}_{\left(k-1\right)h_{n},n},\alpha_{\star}|\rho_{\star}\right)\right\|^{2}\\
    &\quad+o_{P}^{\ast}\left(1\right)\\
    &\to^{P}\mathbb{V}_{1}\left(\alpha|\xi_{\star}\right)\text{ uniformly in }\alpha
\end{align*}
where $o_{P}^{\ast}\left(1\right)$ indicates the term converging in probability to zero uniformly in $\theta$. Then we obtain that $\hat{\alpha}_{n}\to\alpha_{\star}$ in the same way as \citet{Kessler-1997} with Assumption [A3].

In the next place, we consider the consistency of $\hat{\beta}_{n}$. Firstly, we consider the case $\max_{i}\rho_{\star}^{\left(i\right)}\in\left(\ell-1,\ell\right)$ for an integer $\ell\in\left\{1,\ldots,\left[\overline{\rho}\right]+1\right\}$. Then { it is sufficient to show}
\begin{align*}
    &\frac{1}{nh_{n}}\mathbb{H}_{2,n}\left(\beta|\hat{\rho}_{n}\right)-\frac{1}{nh_{n}}\mathbb{H}_{2,n}\left(\beta_{\star}|\rho_{\star}\right)
     \to^{P}\mathbb{V}_{2}\left(\beta|\xi_{\star}\right)\text{ uniformly in }\beta
\end{align*}
{due to Assumption [A3].}
{Because} the evaluation $D_{j}\left(x\right)=O$ where $j\ge \left[\max_{i=1,\ldots,n}\rho_{\ast}^{\left(i\right)}\right]+1$ using independent increments of the Wiener process, Proposition \ref{propEmpMean} {and} Proposition \ref{propEmpI} verify
\begin{align*}
    &F_{\ell}\left(\beta\right)-\frac{1}{nh_{n}}\mathbb{H}_{2,n}\left(\beta_{\star}|\rho_{\star}\right)\to^{P}\mathbb{V}_{2}\left(\beta|\xi_{\star}\right)\text{ uniformly in }\beta,
\end{align*}
where
\begin{align*}
    F_{j}\left(\beta\right):=-\frac{1}{nh_{n}^{2}}\sum_{k=1+j}^{n}\left|\overline{X}_{kh_{n},n}-\overline{X}_{\left(k-1\right)h_{n},n}-h_{n}b\left(\overline{X}_{\left(k-1-j\right)h_{n},n},\beta\right)\right|^{2}{.}
\end{align*}
{In addition,} the exact convergences such that
\begin{align*}
    P\left(\mathbf{1}_{\left\{\ell\right\}}\left(\left[\max_{i}\hat{\rho}_{n}^{\left(i\right)}\right]+1\right)=1\right)\to1,\ P\left(\mathbf{1}_{\left\{j\right\}}\left(\left[\max_{i}\hat{\rho}_{n}^{\left(i\right)}\right]+1\right)=1\right)\to0
\end{align*}
{hold} for all $j\neq\ell$, since for all $j=1,\ldots,\left[\overline{\rho}\right]+1$,
\begin{align*}
    P\left(\mathbf{1}_{\left\{j\right\}}\left(\left[\max_{i}\hat{\rho}_{n}^{\left(i\right)}\right]+1\right)=1\right)=P\left(\max_{i}\hat{\rho}_{n}^{\left(i\right)}\in\left[j-1,j\right)\right){.}
\end{align*}
{Therefore,} for any $\epsilon>0$,
\begin{align*}
    &P\left(\sup_{\beta\in\Theta_{2}}\left|\frac{1}{nh_{n}}\mathbb{H}_{2,n}\left(\beta|\hat{\rho}_{n}\right)-\frac{1}{nh_{n}}\mathbb{H}_{2,n}\left(\beta_{\star}|\rho_{\star}\right)-\mathbb{V}_{2}\left(\beta|\xi_{\star}\right)\right|>\epsilon\right)\\
    &\le \sum_{j\neq\ell}P\left(\mathbf{1}_{\left\{j\right\}}\left(\left[\max_{i}\hat{\rho}_{n}^{\left(i\right)}\right]+1\right)=1\right)\\
    &\quad+P\left(\left\{\mathbf{1}_{\left\{\ell\right\}}\left(\left[\max_{i}\hat{\rho}_{n}^{\left(i\right)}\right]+1\right)=1\right\}\right.\\
    &\hspace{2cm}\left.\cap\left\{\sup_{\beta\in\Theta_{2}}\left|F_{\ell}\left(\beta\right)-\frac{1}{nh_{n}}\mathbb{H}_{2,n}\left(\beta_{\star}|\rho_{\star}\right)-\mathbb{V}_{2}\left(\beta|\xi_{\star}\right)\right|>\epsilon\right\}\right)\\
    &\to0.
\end{align*} For the case $\max_{i}\rho_{\star}^{\left(i\right)}=\ell$ for an integer $\ell=\left\{0,\ldots,\left[\overline{\rho}\right]+1\right\}$, we similarly obtain
\begin{align*}
     &\frac{1}{nh_{n}}\mathbb{H}_{2,n}\left(\beta|\hat{\rho}_{n}\right)-\frac{1}{nh_{n}}\mathbb{H}_{2,n}\left(\beta_{\star}|\rho_{\star}\right)
     \to^{P}\mathbb{V}_{2}\left(\beta|\xi_{\star}\right)\text{ uniformly in }\beta
\end{align*}
because we have
\begin{align*}
    &F_{\ell}\left(\beta\right)-\frac{1}{nh_{n}}\mathbb{H}_{2,n}\left(\beta_{\star}|\rho_{\star}\right)\to^{P}\mathbb{V}_{2}\left(\beta|\xi_{\star}\right)\text{ uniformly in }\beta,\\
    &F_{\ell+1}\left(\beta\right)-\frac{1}{nh_{n}}\mathbb{H}_{2,n}\left(\beta_{\star}|\rho_{\star}\right)\to^{P}\mathbb{V}_{2}\left(\beta|\xi_{\star}\right)\text{ uniformly in }\beta,
\end{align*}
and 
\begin{align*}
    &P\left(\mathbf{1}_{\left\{\ell\right\}}\left(\left[\max_{i}\hat{\rho}_{n}^{\left(i\right)}\right]+1\right)+\mathbf{1}_{\left\{\ell+1\right\}}\left(\left[\max_{i}\hat{\rho}_{n}^{\left(i\right)}\right]+1\right)=1\right)\to1,\\
    &P\left(\mathbf{1}_{\left\{j\right\}}\left(\left[\max_{i}\hat{\rho}_{n}^{\left(i\right)}\right]+1\right)=0\right)\to1,\text{ for all } j\neq\ell,\ell+1,
\end{align*}
and it holds that for any $\epsilon>0$,
\begin{align*}
    &P\left(\sup_{\beta\in\Theta_{2}}\left|\frac{1}{nh_{n}}\mathbb{H}_{2,n}\left(\beta|\hat{\rho}_{n}\right)-\frac{1}{nh_{n}}\mathbb{H}_{2,n}\left(\beta_{\star}|\rho_{\star}\right)-\mathbb{V}_{2}\left(\beta|\xi_{\star}\right)\right|>\epsilon\right)\\
    &\le \sum_{j\neq\ell,\ell+1}P\left(\mathbf{1}_{\left\{j\right\}}\left(\left[\max_{i}\hat{\rho}_{n}^{\left(i\right)}\right]+1\right)=1\right)\\
    &\quad+P\left(\left\{\mathbf{1}_{\left\{\ell\right\}}\left(\left[\max_{i}\hat{\rho}_{n}^{\left(i\right)}\right]+1\right)=1\right\}\right.\\
    &\hspace{2cm}\left.\cap\left\{\sup_{\beta\in\Theta_{2}}\left|F_{\ell}\left(\beta\right)-\frac{1}{nh_{n}}\mathbb{H}_{2,n}\left(\beta_{\star}|\rho_{\star}\right)-\mathbb{V}_{2}\left(\beta|\xi_{\star}\right)\right|>\epsilon\right\}\right)\\
    &\quad+P\left(\left\{\mathbf{1}_{\left\{\ell+1\right\}}\left(\left[\max_{i}\hat{\rho}_{n}^{\left(i\right)}\right]+1\right)=1\right\}\right.\\
    &\hspace{2cm}\left.\cap\left\{\sup_{\beta\in\Theta_{2}}\left|F_{\ell+1}\left(\beta\right)-\frac{1}{nh_{n}}\mathbb{H}_{2,n}\left(\beta_{\star}|\rho_{\star}\right)-\mathbb{V}_{2}\left(\beta|\xi_{\star}\right)\right|>\epsilon\right\}\right)\\
    &\to0.
\end{align*}
Hence it is shown that $\hat{\beta}_{n}\to^{P}\beta_{\star}$ with Assumption [A3].
\end{proof}

\section*{Appendix B. Trivial discussion}

The next lemma supports an exchangeability of integrals.

\begin{lemma}\label{LemmaExchangeIntegrals}
    Let us fix $C>0$. For any $t_{1},t_{2}$ such that $0\le t_{1}\le t_{2}$ and $t_{2}-t_{1}\le C$, and $f:\mathbf{R}_{+}\to \mathbf{R}^{d}\otimes\mathbf{R}^{r}$ such that $f\in L^{1}\left(\left[t_{1},t_{2}\right]\right)$, the following equation holds:
    \begin{align*}
        \int_{t_{1}}^{t_{2}}f\left(s\right)\left(\int_{t_{1}}^{s_{1}}\mathrm{d}w_{s_{2}}\right)\mathrm{d}s_{1}=\int_{t_{1}}^{t_{2}}\left(\int_{s_{1}}^{t_{2}}f\left(s_{2}\right)\mathrm{d}s_{2}\right)\mathrm{d}w_{s_{1}};
    \end{align*}
    and both are normally distributed with mean $\mathbf{0}$ and variance  $\int_{t_{1}}^{t_{2}}\left(\int_{s_{1}}^{t_{2}}f\left(s_{2}\right)\mathrm{d}s_{2}\right)^{\otimes 2}\mathrm{d}s_{1}$.
\end{lemma}

\begin{proof}
For $f:\mathbf{R}_{+}\to \mathbf{R}^{d}\otimes\mathbf{R}^{r}$, we set $F\left(t\right)=\int_{0}^{t}f\left(s\right)\mathrm{d}s$, and then it follows from It\^{o}'s formula that
\begin{align*}
    F\left(t\right)w_{t}&=\int_{0}^{t}\left(\left.\frac{\mathrm{d}}{\mathrm{d}t}F\left(t\right)\right|_{t=s} \right)w_{s}\mathrm{d}s+\int_{0}^{t}F\left(s\right)\mathrm{d}w_{s}\\
    &=\int_{0}^{t}f\left(s\right)w_{s}\mathrm{d}s+\int_{0}^{t}F\left(s\right)\mathrm{d}w_{s},
\end{align*}
and
\begin{align*}
    \int_{t_{1}}^{t_{2}}f\left(s\right)\left(\int_{t_{1}}^{s_{1}}\mathrm{d}w_{s_{2}}\right)\mathrm{d}s_{1}
    &=\int_{t_{1}}^{t_{2}}f\left(s\right)\left(w_{s_{1}}-w_{t_{1}}\right)\mathrm{d}s_{1}\\
    &=F\left(t_{2}\right)w_{t_{2}}-F\left(t_{1}\right)w_{t_{1}}-\int_{t_{1}}^{t_{2}}f\left(s\right)w_{t_{1}}\mathrm{d}s-\int_{t_{1}}^{t_{2}}F\left(s\right)\mathrm{d}w_{s}\\
    &=F\left(t_{2}\right)w_{t_{2}}-F\left(t_{1}\right)w_{t_{1}}-\left(F\left(t_{2}\right)-F\left(t_{1}\right)\right)w_{t_{1}}-\int_{t_{1}}^{t_{2}}F\left(s\right)\mathrm{d}w_{s}\\
    &=\int_{t_{1}}^{t_{2}}\left(F\left(t_{2}\right)-F\left(s\right)\right)\mathrm{d}w_{s}\\
    &=\int_{t_{1}}^{t_{2}}\left(\int_{s_{1}}^{t_{2}}f\left(s_{2}\right)\mathrm{d}s_{2}\right)\mathrm{d}w_{s_{1}},
\end{align*}
which is normally distributed because of Wiener integral obviously.
\end{proof}

\begin{lemma}\label{LemmaFunctionD}
    $f_{\mathbb{D}_{0}}\left(\rho^{\left(i\right)},\rho^{\left(j\right)}\right)$ is continuous.
\end{lemma}

\begin{proof}
We check the continuity at 
(i) $\rho^{\left(i\right)}=\rho^{\left(j\right)}=0$, 
(ii) $\rho^{\left(i\right)}=0$ and $\rho^{\left(j\right)}\in\left(0,1\right)$, 
(iii) $\rho^{\left(i\right)}=0$ and $\rho^{\left(j\right)}=1$, 
(iv) $\rho^{\left(i\right)}=0$ and $\rho^{\left(j\right)}\in\left(1,\overline{\rho}\right]$,
(v) $\rho^{\left(j\right)}=0$ and $\rho^{\left(i\right)}\in\left(0,\overline{\rho}\right]$,
(vi) $\rho^{\left(i\right)}=\rho^{\left(j\right)}\in\left(0,1\right)$, 
(vii) $\rho^{\left(i\right)}=\rho^{\left(j\right)}=1$,
(viii) $\rho^{\left(i\right)}=\rho^{\left(j\right)}\in\left(1,p\right]$, 
(ix) $\rho^{\left(j\right)}=1$ and $\rho^{\left(i\right)}\in\left(0,1\right)$,
(x) $\rho^{\left(j\right)}=1$ and $\rho^{\left(i\right)}\in\left(1,\overline{\rho}\right]$,
(xi) $\rho^{\left(i\right)}\in\left(0,\overline{\rho}\right]$ and $\rho^{\left(i\right)}+1=\rho^{\left(j\right)}$. 

(i) We have that
\begin{align*}
    \left.f_{\mathbb{D}_{0}}\left(\rho^{\left(i\right)},\rho^{\left(j\right)}\right)\right|_{\rho^{\left(i\right)}=\rho^{\left(j\right)}=0}&=0,\\
    \lim_{\rho^{\left(i\right)}=0,\rho^{\left(j\right)}\downarrow0}f_{\mathbb{D}_{0}}\left(\rho^{\left(i\right)},\rho^{\left(j\right)}\right)&=\lim_{\rho^{\left(i\right)}=0,\rho^{\left(j\right)}\downarrow0}\frac{\rho^{\left(j\right)}}{2}=0,\\
    \lim_{\substack{\rho^{\left(i\right)}\downarrow,\rho^{\left(j\right)}\downarrow0\\\rho^{\left(i\right)}<\rho^{\left(j\right)}}}f_{\mathbb{D}_{0}}\left(\rho^{\left(i\right)},\rho^{\left(j\right)}\right)&=\lim_{\substack{\rho^{\left(i\right)}\downarrow,\rho^{\left(j\right)}\downarrow0\\\rho^{\left(i\right)}<\rho^{\left(j\right)}}}\frac{\left(\rho^{\left(i\right)}\right)^{2}-3\rho^{\left(i\right)}\rho^{\left(j\right)}+3\left(\rho^{\left(j\right)}\right)^{2}}{6\rho^{\left(j\right)}}=0,\\
    \lim_{\substack{\rho^{\left(i\right)}\downarrow,\rho^{\left(j\right)}\downarrow0\\\rho^{\left(i\right)}\ge\rho^{\left(j\right)}}}f_{\mathbb{D}_{0}}\left(\rho^{\left(i\right)},\rho^{\left(j\right)}\right)&=\lim_{\substack{\rho^{\left(i\right)}\downarrow,\rho^{\left(j\right)}\downarrow0\\\rho^{\left(i\right)}\ge\rho^{\left(j\right)}}}\frac{\left(\rho^{\left(j\right)}\right)^{2}}{6\rho^{\left(i\right)}}=0,\\
    \lim_{\rho^{\left(j\right)}=0,\rho^{\left(i\right)}\downarrow0}f_{\mathbb{D}_{0}}\left(\rho^{\left(i\right)},\rho^{\left(j\right)}\right)&=0.\\
\end{align*}
(ii)  It holds that
\begin{align*}
    \left.f_{\mathbb{D}_{0}}\left(\rho^{\left(i\right)},\rho^{\left(j\right)}\right)\right|_{\rho^{\left(i\right)}=0,\rho^{\left(j\right)}\in\left(0,1\right)}&=\frac{\rho^{\left(j\right)}}{2},\\
    \lim_{\rho^{\left(j\right)}\in\left(0,1\right),\rho^{\left(i\right)}\downarrow0}f_{\mathbb{D}_{0}}\left(\rho^{\left(i\right)},\rho^{\left(j\right)}\right)&=\lim_{\rho^{\left(j\right)}\in\left(0,1\right),\rho^{\left(i\right)}\downarrow0}\frac{\left(\rho^{\left(i\right)}\right)^{2}-3\rho^{\left(i\right)}\rho^{\left(j\right)}+3\left(\rho^{\left(j\right)}\right)^{2}}{6\rho^{\left(j\right)}}=\frac{\rho^{\left(j\right)}}{2}.
\end{align*}
(iii) We obtain that
\begin{align*}
    \left.f_{\mathbb{D}_{0}}\left(\rho^{\left(i\right)},\rho^{\left(j\right)}\right)\right|_{\rho^{\left(i\right)}=0,\rho^{\left(j\right)}=1}&=\frac{1}{2},\\
    \lim_{\rho^{\left(i\right)}=0,\rho^{\left(j\right)}\uparrow1}f_{\mathbb{D}_{0}}\left(\rho^{\left(i\right)},\rho^{\left(j\right)}\right)
    &=\frac{1}{2},\\
    \lim_{\rho^{\left(i\right)}=0,\rho^{\left(j\right)}\downarrow1}f_{\mathbb{D}_{0}}\left(\rho^{\left(i\right)},\rho^{\left(j\right)}\right)&=\lim_{\rho^{\left(i\right)}=0,\rho^{\left(j\right)}\downarrow1}\frac{2\rho^{\left(j\right)}-1}{2\rho^{\left(j\right)}}\\
    &=\frac{1}{2},\\
    \lim_{\substack{\rho^{\left(i\right)}\downarrow0,\rho^{\left(j\right)}\uparrow1\\\rho^{\left(i\right)}<\rho^{\left(j\right)}}}f_{\mathbb{D}_{0}}\left(\rho^{\left(i\right)},\rho^{\left(j\right)}\right)&=\lim_{\substack{\rho^{\left(i\right)}\downarrow0,\rho^{\left(j\right)}\uparrow1\\\rho^{\left(i\right)}<\rho^{\left(j\right)}}}\frac{\left(\rho^{\left(i\right)}\right)^{2}-3\rho^{\left(i\right)}\rho^{\left(j\right)}+3\left(\rho^{\left(j\right)}\right)^{2}}{6\rho^{\left(j\right)}}\\
    &=\frac{1}{2},\\
    \lim_{\substack{\rho^{\left(i\right)}\downarrow0,\rho^{\left(j\right)}\downarrow1,\\\rho^{\left(i\right)}<\rho^{\left(j\right)}\le \rho^{\left(i\right)}+1}}f_{\mathbb{D}_{0}}\left(\rho^{\left(i\right)},\rho^{\left(j\right)}\right)&=\lim_{\substack{\rho^{\left(i\right)}\downarrow0,\rho^{\left(j\right)}\downarrow1,\\
    \rho^{\left(i\right)}<\rho^{\left(j\right)}\le \rho^{\left(i\right)}+1}}\frac{\left(\rho^{\left(i\right)}-\rho^{\left(j\right)}\right)^{3}+3\left(\rho^{\left(j\right)}\right)^{2}-3\rho^{\left(j\right)}+1}{6\rho^{\left(i\right)}\rho^{\left(j\right)}}\\
    &=\lim_{\substack{\rho^{\left(i\right)}\downarrow0,\rho^{\left(j\right)}\downarrow1,\\\rho^{\left(i\right)}<\rho^{\left(j\right)}\le \rho^{\left(i\right)}+1}}\frac{\left(\rho^{\left(i\right)}\right)^{3}-3\left(\rho^{\left(i\right)}\right)^{2}\rho^{\left(j\right)}+3\rho^{\left(i\right)}\left(\rho^{\left(j\right)}\right)^{2}}{6\rho^{\left(i\right)}\rho^{\left(j\right)}}\\
    &\quad-\lim_{\substack{\rho^{\left(i\right)}\downarrow0,\rho^{\left(j\right)}\downarrow1,\\\rho^{\left(i\right)}<\rho^{\left(j\right)}\le \rho^{\left(i\right)}+1}}\frac{\left(\rho^{\left(j\right)}-1\right)^{3}}{6\rho^{\left(i\right)}\rho^{\left(j\right)}}\\
    &=\frac{1}{2},\\
    \lim_{\substack{\rho^{\left(i\right)}\downarrow0,\rho^{\left(j\right)}\downarrow1,\\\rho^{\left(j\right)}>\rho^{\left(i\right)}+1}}f_{\mathbb{D}_{0}}\left(\rho^{\left(i\right)},\rho^{\left(j\right)}\right)&=\lim_{\substack{\rho^{\left(i\right)}\downarrow0,\rho^{\left(j\right)}\downarrow1,\\\rho^{\left(j\right)}> \rho^{\left(i\right)}+1}}\frac{6\rho^{\left(i\right)}\rho^{\left(j\right)}-3\left(\rho^{\left(i\right)}\right)^{2}-3\rho^{\left(i\right)}}{6\rho^{\left(i\right)}\rho^{\left(j\right)}}\\
    &=\frac{1}{2}.
\end{align*}
(iv) We can have that
\begin{align*}
    \left.f_{\mathbb{D}_{0}}\left(\rho^{\left(i\right)},\rho^{\left(j\right)}\right)\right|_{\rho^{\left(i\right)}=0,\rho^{\left(j\right)}\in\left(1,\overline{\rho}\right]}&=\frac{2\rho^{\left(j\right)}-1}{2\rho^{\left(j\right)}},\\
    \lim_{\rho^{\left(i\right)}\downarrow0,\rho^{\left(j\right)}\in\left(1,\overline{\rho}\right]}f_{\mathbb{D}_{0}}\left(\rho^{\left(i\right)},\rho^{\left(j\right)}\right)&=\lim_{\rho^{\left(i\right)}\downarrow0,\rho^{\left(j\right)}\in\left(1,\overline{\rho}\right]}\frac{6\rho^{\left(i\right)}\rho^{\left(j\right)}-3\left(\rho^{\left(i\right)}\right)^{2}-3\rho^{\left(i\right)}}{6\rho^{\left(i\right)}\rho^{\left(j\right)}}\\
    &=\frac{2\rho^{\left(j\right)}-1}{2\rho^{\left(j\right)}}.
\end{align*}
(v) It holds that
\begin{align*}
    \left.f_{\mathbb{D}_{0}}\left(\rho^{\left(i\right)},\rho^{\left(j\right)}\right)\right|_{\rho^{\left(i\right)}\in\left(0,\overline{\rho}\right],\rho^{\left(j\right)}=0}&=0,\\
    \lim_{\substack{\rho^{\left(i\right)}\in\left(0,\overline{\rho}\right],\rho^{\left(j\right)}\downarrow0\\\rho^{\left(i\right)}\ge\rho^{\left(j\right)}}}f_{\mathbb{D}_{0}}\left(\rho^{\left(i\right)},\rho^{\left(j\right)}\right)&=\lim_{\substack{\rho^{\left(i\right)}\in\left(0,\overline{\rho}\right],\rho^{\left(j\right)}\downarrow0\\\rho^{\left(i\right)}\ge\rho^{\left(j\right)}}}\frac{\left(\rho^{\left(j\right)}\right)^{3}}{6\rho^{\left(i\right)}\rho^{\left(j\right)}}=0.
\end{align*}
(vi) We have that
\begin{align*}
    \left.f_{\mathbb{D}_{0}}\left(\rho^{\left(i\right)},\rho^{\left(j\right)}\right)\right|_{\rho^{\left(i\right)}=\rho^{\left(j\right)}\in\left(0,1\right)}&=\frac{\rho^{\left(i\right)}}{6}\\
    \lim_{\substack{\rho^{\left(i\right)}\in\left(0,1\right),\rho^{\left(i\right)}-\rho^{\left(j\right)}\to0,\\\rho^{\left(i\right)}\ge \rho^{\left(j\right)}}}f_{\mathbb{D}_{0}}\left(\rho^{\left(i\right)},\rho^{\left(j\right)}\right)&=\frac{\rho^{\left(i\right)}}{6},\\
    \lim_{\substack{\rho^{\left(i\right)}\in\left(0,1\right),\rho^{\left(i\right)}-\rho^{\left(j\right)}\to0,\\\rho^{\left(i\right)}< \rho^{\left(j\right)}}}f_{\mathbb{D}_{0}}\left(\rho^{\left(i\right)},\rho^{\left(j\right)}\right)&=\lim_{\substack{\rho^{\left(i\right)}\in\left(0,1\right),\rho^{\left(i\right)}-\rho^{\left(j\right)}\to0,\\\rho^{\left(i\right)}< \rho^{\left(j\right)}}}\frac{\left(\rho^{\left(i\right)}-\rho^{\left(j\right)}\right)^{3}+\left(\rho^{\left(j\right)}\right)^{3}}{6\rho^{\left(i\right)}\rho^{\left(j\right)}}=\frac{\rho^{\left(i\right)}}{6}.
\end{align*}
(vii) It holds that
\begin{align*}
    \left.f_{\mathbb{D}_{0}}\left(\rho^{\left(i\right)},\rho^{\left(j\right)}\right)\right|_{\rho^{\left(i\right)}=\rho^{\left(j\right)}=1}&=\frac{1}{6},\\
    \lim_{\substack{\rho^{\left(i\right)}\to1,\rho^{\left(j\right)}\uparrow1\\\rho^{\left(i\right)}<\rho^{\left(j\right)}}}f_{\mathbb{D}_{0}}\left(\rho^{\left(i\right)},\rho^{\left(j\right)}\right)&=\lim_{\substack{\rho^{\left(i\right)}\to1,\rho^{\left(j\right)}\uparrow1\\\rho^{\left(i\right)}<\rho^{\left(j\right)}}}\frac{\left(\rho^{\left(i\right)}-\rho^{\left(j\right)}\right)^{3}+\left(\rho^{\left(j\right)}\right)^{3}}{6\rho^{\left(i\right)}\rho^{\left(j\right)}}=\frac{1}{6},\\
    \lim_{\substack{\rho^{\left(i\right)}\to1,\rho^{\left(j\right)}\uparrow1\\\rho^{\left(i\right)}\ge\rho^{\left(j\right)}}}f_{\mathbb{D}_{0}}\left(\rho^{\left(i\right)},\rho^{\left(j\right)}\right)&=\lim_{\substack{\rho^{\left(i\right)}\to1,\rho^{\left(j\right)}\uparrow1\\\rho^{\left(i\right)}\ge\rho^{\left(j\right)}}}\frac{\left(\rho^{\left(j\right)}\right)^{3}}{6\rho^{\left(i\right)}\rho^{\left(j\right)}}=\frac{1}{6},\\
    \lim_{\substack{\rho^{\left(i\right)}\to1,\rho^{\left(j\right)}\downarrow1\\\rho^{\left(i\right)}<\rho^{\left(j\right)}}}f_{\mathbb{D}_{0}}\left(\rho^{\left(i\right)},\rho^{\left(j\right)}\right)&=\lim_{\substack{\rho^{\left(i\right)}\to1,\rho^{\left(j\right)}\downarrow1\\\rho^{\left(i\right)}<\rho^{\left(j\right)}}}\frac{\left(\rho^{\left(i\right)}-\rho^{\left(j\right)}\right)^{3}+3\left(\rho^{\left(j\right)}\right)^{2}-3\rho^{\left(j\right)}+1}{6\rho^{\left(i\right)}\rho^{\left(j\right)}}=\frac{1}{6},\\
    \lim_{\substack{\rho^{\left(i\right)}\to1,\rho^{\left(j\right)}\downarrow1\\\rho^{\left(i\right)}\ge\rho^{\left(j\right)}}}f_{\mathbb{D}_{0}}\left(\rho^{\left(i\right)},\rho^{\left(j\right)}\right)&=\lim_{\substack{\rho^{\left(i\right)}\to1,\rho^{\left(j\right)}\downarrow1\\\rho^{\left(i\right)}\ge\rho^{\left(j\right)}}}\frac{3\left(\rho^{\left(j\right)}\right)^{2}-3\rho^{\left(j\right)}+1}{6\rho^{\left(i\right)}\rho^{\left(j\right)}}=\frac{1}{6}.
\end{align*}
(viii) We have that
\begin{align*}
    \left.f_{\mathbb{D}_{0}}\left(\rho^{\left(i\right)},\rho^{\left(j\right)}\right)\right|_{\rho^{\left(i\right)}=\rho^{\left(j\right)}\in\left(1,\overline{\rho}\right]}&=\frac{3\left(\rho^{\left(i\right)}\right)^{2}-3\rho^{\left(i\right)}+1}{6\left(\rho^{\left(i\right)}\right)^{2}}\\
    \lim_{\substack{\rho^{\left(i\right)}\in\left(1,\overline{\rho}\right],\rho^{\left(i\right)}-\rho^{\left(j\right)}\to0,\\\rho^{\left(i\right)}\ge \rho^{\left(j\right)}}}f_{\mathbb{D}_{0}}\left(\rho^{\left(i\right)},\rho^{\left(j\right)}\right)&=\frac{3\left(\rho^{\left(i\right)}\right)^{2}-3\rho^{\left(i\right)}+1}{6\left(\rho^{\left(i\right)}\right)^{2}},\\
    \lim_{\substack{\rho^{\left(i\right)}\in\left(1,\overline{\rho}\right],\rho^{\left(i\right)}-\rho^{\left(j\right)}\to0,\\\rho^{\left(i\right)}< \rho^{\left(j\right)}}}f_{\mathbb{D}_{0}}\left(\rho^{\left(i\right)},\rho^{\left(j\right)}\right)&=\lim_{\substack{\rho^{\left(i\right)}\in\left(1,\overline{\rho}\right],\rho^{\left(i\right)}-\rho^{\left(j\right)}\to0,\\\rho^{\left(i\right)}< \rho^{\left(j\right)}}}\frac{\left(\rho^{\left(i\right)}-\rho^{\left(j\right)}\right)^{3}+3\left(\rho^{\left(j\right)}\right)^{2}-3\rho^{\left(j\right)}+1}{6\rho^{\left(i\right)}\rho^{\left(j\right)}}\\
    &=\frac{3\left(\rho^{\left(i\right)}\right)^{2}-3\rho^{\left(i\right)}+1}{6\left(\rho^{\left(i\right)}\right)^{2}}.
\end{align*}
(ix) It holds that
\begin{align*}
    \left.f_{\mathbb{D}_{0}}\left(\rho^{\left(i\right)},\rho^{\left(j\right)}\right)\right|_{\rho^{\left(i\right)}\in\left(0,1\right),\rho^{\left(j\right)}=1}&=\frac{\left(\rho^{\left(i\right)}-1\right)^{3}+1}{6\rho^{\left(i\right)}},\\
    \lim_{\rho^{\left(i\right)}\in\left(0,1\right),\rho^{\left(j\right)}\uparrow1}f_{\mathbb{D}_{0}}\left(\rho^{\left(i\right)},\rho^{\left(j\right)}\right)&=\frac{\left(\rho^{\left(i\right)}-1\right)^{3}+1}{6\rho^{\left(i\right)}},\\
    \lim_{\rho^{\left(i\right)}\in\left(0,1\right),\rho^{\left(j\right)}\downarrow1}f_{\mathbb{D}_{0}}\left(\rho^{\left(i\right)},\rho^{\left(j\right)}\right)&=\lim_{\rho^{\left(i\right)}\in\left(0,1\right),\rho^{\left(j\right)}\downarrow1}\frac{\left(\rho^{\left(i\right)}-\rho^{\left(j\right)}\right)^{3}+3\left(\rho^{\left(j\right)}\right)^{2}-3\rho^{\left(j\right)}+1}{6\rho^{\left(i\right)}\rho^{\left(j\right)}}\\
    &=\frac{\left(\rho^{\left(i\right)}-1\right)^{3}+1}{6\rho^{\left(i\right)}}.
\end{align*}
(x) We have that
\begin{align*}
    \left.f_{\mathbb{D}_{0}}\left(\rho^{\left(i\right)},\rho^{\left(j\right)}\right)\right|_{\rho^{\left(i\right)}\in\left(1,\overline{\rho}\right],\rho^{\left(j\right)}=1}&=\frac{1}{6\rho^{\left(i\right)}},\\
    \lim_{\rho^{\left(i\right)}\in\left(1,\overline{\rho}\right],\rho^{\left(j\right)}\uparrow1}f_{\mathbb{D}_{0}}\left(\rho^{\left(i\right)},\rho^{\left(j\right)}\right)&=\frac{1}{6\rho^{\left(i\right)}},\\
    \lim_{\rho^{\left(i\right)}\in\left(1,\overline{\rho}\right],\rho^{\left(j\right)}\downarrow1}f_{\mathbb{D}_{0}}\left(\rho^{\left(i\right)},\rho^{\left(j\right)}\right)&=\lim_{\rho^{\left(i\right)}\in\left(1,\overline{\rho}\right],\rho^{\left(j\right)}\downarrow1}\frac{3\left(\rho^{\left(j\right)}\right)^{2}-3\rho^{\left(j\right)}+1}{6\rho^{\left(i\right)}\rho^{\left(j\right)}}\\
    &=\frac{1}{6\rho^{\left(i\right)}}.
\end{align*}
(xi) It holds that
\begin{align*}
    \left.f_{\mathbb{D}_{0}}\left(\rho^{\left(i\right)},\rho^{\left(j\right)}\right)\right|_{\rho^{\left(i\right)}\in\left(0,\overline{\rho}\right],\rho^{\left(i\right)}+1=\rho^{\left(j\right)}}&=\frac{1}{2},\\
    \lim_{\substack{\rho^{\left(i\right)}\in\left(0,\overline{\rho}\right],\rho^{\left(i\right)}+1-\rho^{\left(j\right)}\to0\\\rho^{\left(j\right)}<\rho^{\left(i\right)}+1}}f_{\mathbb{D}_{0}}\left(\rho^{\left(i\right)},\rho^{\left(j\right)}\right)&=\frac{1}{2},\\
    \lim_{\substack{\rho^{\left(i\right)}\in\left(0,\overline{\rho}\right],\rho^{\left(i\right)}+1-\rho^{\left(j\right)}\to0\\\rho^{\left(j\right)}>\rho^{\left(i\right)}+1}}f_{\mathbb{D}_{0}}\left(\rho^{\left(i\right)},\rho^{\left(j\right)}\right)&=\lim_{\substack{\rho^{\left(i\right)}\in\left(0,\overline{\rho}\right],\rho^{\left(i\right)}+1-\rho^{\left(j\right)}\to0\\\rho^{\left(j\right)}>\rho^{\left(i\right)}+1}}\frac{6\rho^{\left(i\right)}\rho^{\left(j\right)}-3\left(\rho^{\left(i\right)}\right)^{2}-3\rho^{\left(i\right)}}{6\rho^{\left(i\right)}\rho^{\left(j\right)}}\\
    &=\lim_{\substack{\rho^{\left(i\right)}\in\left(0,\overline{\rho}\right],\rho^{\left(i\right)}+1-\rho^{\left(j\right)}\to0\\\rho^{\left(j\right)}>\rho^{\left(i\right)}+1}}\frac{6\left(\rho^{\left(i\right)}+1\right)-3\rho^{\left(i\right)}-3}{6\left(\rho^{\left(i\right)}+1\right)}\\
    &=\frac{1}{2}.
\end{align*}
Therefore we obtain the continuity of $f_{\mathbb{D}_{0}}$.
\end{proof}

\begin{lemma}\label{LemmaFunctionG}
$f_{\mathbb{G}}\left(\rho^{\left(i\right)},\rho^{\left(j\right)}\right)$ is continuous.
\end{lemma}
\begin{proof}
Since the continuity of $f_{\mathbb{D}_{0}}$ is shown in Lemma \ref{LemmaFunctionD}, it is sufficient to show the continuity of $f_{\mathbb{G}}+f_{\mathbb{D}_{0}}$. Note that
\begin{align*}
    &\left(f_{\mathbb{G}}+f_{\mathbb{D}_{0}}\right)\left(\rho^{\left(i\right)},\rho^{\left(j\right)}\right)\\
    &=\begin{cases}
    1& \text{ if }\rho^{\left(i\right)}=0,\\
    1-\frac{\rho^{\left(i\right)}}{2}&\text{ if }\rho^{\left(i\right)}\in\left(0,1\right],\rho^{\left(j\right)}=0,\\
    \frac{1}{2\rho^{\left(i\right)}}&\text{ if }\rho^{\left(i\right)}\in\left(1,\overline{\rho}\right],\rho^{\left(j\right)}=0,\\
    \frac{\left(\rho^{\left(j\right)}\right)^{2}+\rho^{\left(j\right)}}{2\rho^{\left(i\right)}\rho^{\left(j\right)}}& \text{ if }\rho^{\left(i\right)}\in\left(\rho^{\left(j\right)}+1,\overline{\rho}\right],\rho^{\left(j\right)}>0,\\
    \frac{\left(\rho^{\left(i\right)}-\rho^{\left(j\right)}\right)^3 - 3\left(\rho^{\left(i\right)}\right)^2 + 6\rho^{\left(i\right)}\rho^{\left(j\right)} + 3 \rho^{\left(i\right)}- 1}{6\rho^{\left(i\right)}\rho^{\left(j\right)}}&\text{ if }\rho^{\left(i\right)}\in\left(1,\overline{\rho}\right], \rho^{\left(i\right)}\in \left(\rho^{\left(j\right)},\rho^{\left(j\right)}+1\right],\\
    \frac{-3\left(\rho^{\left(i\right)}\right)^{2}+6\rho^{\left(i\right)}\rho^{\left(j\right)}+3\rho^{\left(i\right)}-1}{6\rho^{\left(i\right)}\rho^{\left(j\right)}}&\text{ if }\rho^{\left(i\right)}\in\left(1,\overline{\rho}\right],\rho^{\left(i\right)}\le \rho^{\left(j\right)},\\
    \frac{-3\left(\rho^{\left(i\right)}\right)^{2}\rho^{\left(j\right)}+3\rho^{\left(i\right)}\left(\rho^{\left(j\right)}\right)^{2}+6\rho^{\left(i\right)}\rho^{\left(j\right)}-\left(\rho^{\left(j\right)}\right)^{3}}{6\rho^{\left(i\right)}\rho^{\left(j\right)}}&\text{ if }\rho^{\left(i\right)}\in\left(0,1\right],\rho^{\left(i\right)}>\rho^{\left(j\right)},\\
    \frac{6\rho^{\left(i\right)}\rho^{\left(j\right)}-\left(\rho^{\left(i\right)}\right)^{3}}{6\rho^{\left(i\right)}\rho^{\left(j\right)}}&\text{ if }\rho^{\left(i\right)}\in\left(0,1\right],\rho^{\left(i\right)}\le\rho^{\left(j\right)}.
    \end{cases}
\end{align*}
We check the continuity of $f_{\mathbb{G}}+f_{\mathbb{D}_{0}}$ at 
(i) $\rho^{\left(i\right)}=\rho^{\left(j\right)}=0$,  
(ii) $\rho^{\left(i\right)}\in\left(0,1\right)$ and $\rho^{\left(j\right)}=0$,
(iii) $\rho^{\left(i\right)}=1$ and $\rho^{\left(j\right)}=0$, 
(iv) $\rho^{\left(i\right)}\in\left(1,\overline{\rho}\right]$ and $\rho^{\left(j\right)}=0$,
(v) $\rho^{\left(i\right)}\in\left(1,\overline{\rho}\right]$ and $\rho^{\left(i\right)}=\rho^{\left(j\right)}+1$,
(vi) $\rho^{\left(i\right)}\in\left(0,1\right)$ and $\rho^{\left(i\right)}=\rho^{\left(j\right)}$,
(vii) $\rho^{\left(i\right)}=1$ and $\rho^{\left(j\right)}\in\left(0,1\right)$,
(viii) $\rho^{\left(i\right)}=\rho^{\left(j\right)}=1$,
(ix) $\rho^{\left(i\right)}=1$ and $\rho^{\left(j\right)}\in\left(1,\overline{\rho}\right]$,
(x) $\rho^{\left(i\right)}\in\left(1,\overline{\rho}\right]$ and $\rho^{\left(i\right)}=\rho^{\left(j\right)}$,
(xi) $\rho^{\left(i\right)}=0$, $\rho^{\left(j\right)}\in\left(0,\overline{\rho}\right]$. 

\noindent
For (i), we obtain that
\begin{align*}
    \left.\left(f_{\mathbb{G}}+f_{\mathbb{D}_{0}}\right)\left(\rho^{\left(i\right)},\rho^{\left(j\right)}\right)\right|_{\rho^{\left(i\right)}=\rho^{\left(j\right)}=0}&=1,\\
    \lim_{\rho^{\left(i\right)}\downarrow0,\rho^{\left(j\right)}=0}\left(f_{\mathbb{G}}+f_{\mathbb{D}_{0}}\right)\left(\rho^{\left(i\right)},\rho^{\left(j\right)}\right)&=\lim_{\rho^{\left(i\right)}\downarrow0,\rho^{\left(j\right)}=0}\left(1-\frac{\rho^{\left(i\right)}}{2}\right)=1,\\
    \lim_{\substack{\rho^{\left(i\right)}\downarrow0,\rho^{\left(j\right)}\downarrow0,\\\rho^{\left(i\right)}>\rho^{\left(j\right)}}}\left(f_{\mathbb{G}}+f_{\mathbb{D}_{0}}\right)\left(\rho^{\left(i\right)},\rho^{\left(j\right)}\right)&=\lim_{\substack{\rho^{\left(i\right)}\downarrow0,\rho^{\left(j\right)}\downarrow0,\\\rho^{\left(i\right)}>\rho^{\left(j\right)}}}\frac{3\rho^{\left(i\right)}\left(\rho^{\left(j\right)}\right)^{2}-3\left(\rho^{\left(i\right)}\right)^{2}\rho^{\left(j\right)}+6\rho^{\left(i\right)}\rho^{\left(j\right)}-\left(\rho^{\left(j\right)}\right)^{3}}{6\rho^{\left(i\right)}\rho^{\left(j\right)}}\\
    &=1,\\
    \lim_{\substack{\rho^{\left(i\right)}\downarrow0,\rho^{\left(j\right)}\downarrow0,\\\rho^{\left(i\right)}\le\rho^{\left(j\right)}}}\left(f_{\mathbb{G}}+f_{\mathbb{D}_{0}}\right)\left(\rho^{\left(i\right)},\rho^{\left(j\right)}\right)&=\lim_{\substack{\rho^{\left(i\right)}\downarrow0,\rho^{\left(j\right)}\downarrow0,\\\rho^{\left(i\right)}\le\rho^{\left(j\right)}}}\frac{6\rho^{\left(i\right)}\rho^{\left(j\right)}-\left(\rho^{\left(i\right)}\right)^{3}}{6\rho^{\left(i\right)}\rho^{\left(j\right)}}=1,\\
    \lim_{\rho^{\left(i\right)}=0,\rho^{\left(j\right)}\downarrow0}\left(f_{\mathbb{G}}+f_{\mathbb{D}_{0}}\right)\left(\rho^{\left(i\right)},\rho^{\left(j\right)}\right)&=\lim_{\rho^{\left(i\right)}=0,\rho^{\left(j\right)}\downarrow0}1=1.
\end{align*}
For (ii), one has that
\begin{align*}
    &\left.\left(f_{\mathbb{G}}+f_{\mathbb{D}_{0}}\right)\left(\rho^{\left(i\right)},\rho^{\left(j\right)}\right)\right|_{\rho^{\left(i\right)}\in\left(0,1\right),\rho^{\left(j\right)}=0}=1-\frac{\rho^{\left(i\right)}}{2},\\
    &\lim_{\rho^{\left(i\right)}\in\left(0,1\right),\rho^{\left(j\right)}\downarrow0}\left(f_{\mathbb{G}}+f_{\mathbb{D}_{0}}\right)\left(\rho^{\left(i\right)},\rho^{\left(j\right)}\right)\\
    &\quad=\lim_{\rho^{\left(i\right)}\in\left(0,1\right),\rho^{\left(j\right)}\downarrow0}\frac{-3\left(\rho^{\left(i\right)}\right)^{2}\rho^{\left(j\right)}+3\rho^{\left(i\right)}\left(\rho^{\left(j\right)}\right)^{2}+6\rho^{\left(i\right)}\rho^{\left(j\right)}-\left(\rho^{\left(j\right)}\right)^{3}}{6\rho^{\left(i\right)}\rho^{\left(j\right)}}\\
    &\quad=1-\frac{\rho^{\left(i\right)}}{2}.
\end{align*}
For (iii), we can evaluate that
\begin{align*}
    &\left.\left(f_{\mathbb{G}}+f_{\mathbb{D}_{0}}\right)\left(\rho^{\left(i\right)},\rho^{\left(j\right)}\right)\right|_{\rho^{\left(i\right)}=1,\rho^{\left(j\right)}=0}=\frac{1}{2},\\
    &\lim_{\rho^{\left(i\right)}\uparrow1,\rho^{\left(j\right)}=0}\left(f_{\mathbb{G}}+f_{\mathbb{D}_{0}}\right)\left(\rho^{\left(i\right)},\rho^{\left(j\right)}\right)=\lim_{\rho^{\left(i\right)}\uparrow1,\rho^{\left(j\right)}=0}\left(1-\frac{\rho^{\left(i\right)}}{2}\right)=\frac{1}{2},\\
    &\lim_{\substack{\rho^{\left(i\right)}\to1,\rho^{\left(j\right)}\to0\\\rho^{\left(i\right)}\in\left(0,1\right],\rho^{\left(i\right)}>\rho^{\left(j\right)}}}\left(f_{\mathbb{G}}+f_{\mathbb{D}_{0}}\right)\left(\rho^{\left(i\right)},\rho^{\left(j\right)}\right)\\
    &\quad=\lim_{\substack{\rho^{\left(i\right)}\to1,\rho^{\left(j\right)}\to0\\\rho^{\left(i\right)}\in\left(0,1\right],\rho^{\left(i\right)}>\rho^{\left(j\right)}}}\frac{3\rho^{\left(i\right)}\left(\rho^{\left(j\right)}\right)^{2}-3\left(\rho^{\left(i\right)}\right)^{2}\rho^{\left(j\right)}+6\rho^{\left(i\right)}\rho^{\left(j\right)}-\left(\rho^{\left(j\right)}\right)^{3}}{6\rho^{\left(i\right)}\rho^{\left(j\right)}}=\frac{1}{2},\\
    &\lim_{\substack{\rho^{\left(i\right)}\to1,\rho^{\left(j\right)}\to0\\\rho^{\left(i\right)}\in\left(1,\overline{\rho}\right],\rho^{\left(i\right)}\in\left(\rho^{\left(j\right)},\rho^{\left(j\right)}+1\right]}}\left(f_{\mathbb{G}}+f_{\mathbb{D}_{0}}\right)\left(\rho^{\left(i\right)},\rho^{\left(j\right)}\right)\\
    &\quad=\lim_{\substack{\rho^{\left(i\right)}\to1,\rho^{\left(j\right)}\to0\\\rho^{\left(i\right)}\in\left(1,\overline{\rho}\right],\rho^{\left(i\right)}\in\left(\rho^{\left(j\right)},\rho^{\left(j\right)}+1\right]}}\frac{\left(\rho^{\left(i\right)}-\rho^{\left(j\right)}\right)^3 - 3\left(\rho^{\left(i\right)}\right)^2 + 6\rho^{\left(i\right)}\rho^{\left(j\right)} + 3 \rho^{\left(i\right)}- 1}{6\rho^{\left(i\right)}\rho^{\left(j\right)}}\\
    &\quad=\lim_{\substack{\rho^{\left(i\right)}\to1,\rho^{\left(j\right)}\to0\\\rho^{\left(i\right)}\in\left(1,\overline{\rho}\right],\rho^{\left(i\right)}\in\left(\rho^{\left(j\right)},\rho^{\left(j\right)}+1\right]}}\frac{\left(\rho^{\left(i\right)}\right)^3 -3\left(\rho^{\left(i\right)}\right)^{2}\rho^{\left(j\right)}- 3\left(\rho^{\left(i\right)}\right)^2 +3 \rho^{\left(i\right)}- 1}{6\rho^{\left(i\right)}\rho^{\left(j\right)}}+1\\
    &\quad=\lim_{\substack{\rho^{\left(i\right)}\to1,\rho^{\left(j\right)}\to0\\\rho^{\left(i\right)}\in\left(1,\overline{\rho}\right],\rho^{\left(i\right)}\in\left(\rho^{\left(j\right)},\rho^{\left(j\right)}+1\right]}}\frac{\left(\rho^{\left(i\right)}\right)^3- 3\left(\rho^{\left(i\right)}\right)^2 +3 \rho^{\left(i\right)}- 1}{6\rho^{\left(i\right)}\rho^{\left(j\right)}}+\frac{1}{2}\\
    &\quad=\lim_{\substack{\rho^{\left(i\right)}\to1,\rho^{\left(j\right)}\to0\\\rho^{\left(i\right)}\in\left(1,\overline{\rho}\right],\rho^{\left(i\right)}\in\left(\rho^{\left(j\right)},\rho^{\left(j\right)}+1\right]}}\frac{\left(\rho^{\left(i\right)}-1\right)^3}{6\rho^{\left(i\right)}\rho^{\left(j\right)}}+\frac{1}{2}\\
    &\quad=\frac{1}{2},\\
    &\lim_{\substack{\rho^{\left(i\right)}\to1,\rho^{\left(j\right)}\to0\\\rho^{\left(i\right)}\in\left(1,\overline{\rho}\right],\rho^{\left(i\right)}>\rho^{\left(j\right)}+1}}\left(f_{\mathbb{G}}+f_{\mathbb{D}_{0}}\right)\left(\rho^{\left(i\right)},\rho^{\left(j\right)}\right)=\lim_{\substack{\rho^{\left(i\right)}\to1,\rho^{\left(j\right)}\to0\\\rho^{\left(i\right)}\in\left(1,\overline{\rho}\right],\rho^{\left(i\right)}>\rho^{\left(j\right)}+1}}\frac{\left(\rho^{\left(j\right)}\right)^{2}+\rho^{\left(j\right)}}{2\rho^{\left(i\right)}\rho^{\left(j\right)}}=\frac{1}{2},\\
    &\lim_{\rho^{\left(i\right)}\downarrow1,\rho^{\left(j\right)}=0}\left(f_{\mathbb{G}}+f_{\mathbb{D}_{0}}\right)\left(\rho^{\left(i\right)},\rho^{\left(j\right)}\right)=\lim_{\rho^{\left(i\right)}\downarrow1,\rho^{\left(j\right)}=0}\frac{1}{2\rho^{\left(i\right)}}=\frac{1}{2}.
\end{align*}
For (iv), we obtain that
\begin{align*}
    \left.\left(f_{\mathbb{G}}+f_{\mathbb{D}_{0}}\right)\left(\rho^{\left(i\right)},\rho^{\left(j\right)}\right)\right|_{\rho^{\left(i\right)}=1,\rho^{\left(j\right)}=0}&=\frac{1}{2\rho^{\left(i\right)}},\\
    \lim_{\rho^{\left(i\right)}\in\left(1,\overline{\rho}\right],\rho^{\left(j\right)}\downarrow0}\left(f_{\mathbb{G}}+f_{\mathbb{D}_{0}}\right)\left(\rho^{\left(i\right)},\rho^{\left(j\right)}\right)
    &=\lim_{\rho^{\left(i\right)}\in\left(1,\overline{\rho}\right],\rho^{\left(j\right)}\downarrow0}\frac{\left(\rho^{\left(j\right)}\right)^{2}+\rho^{\left(j\right)}}{2\rho^{\left(i\right)}\rho^{\left(j\right)}}=\frac{1}{2\rho^{\left(i\right)}}.
\end{align*}
For (v), it holds that
\begin{align*}
    \left.\left(f_{\mathbb{G}}+f_{\mathbb{D}_{0}}\right)\left(\rho^{\left(i\right)},\rho^{\left(j\right)}\right)\right|_{\rho^{\left(i\right)}\in\left(1,\overline{\rho}\right],\rho^{\left(i\right)}=\rho^{\left(j\right)}+1}&=\frac{1}{2},\\
    \lim_{\substack{\rho^{\left(i\right)}\in\left(1,\overline{\rho}\right],\rho^{\left(j\right)}-\rho^{\left(i\right)}+1\to0\\\rho^{\left(j\right)}>\rho^{\left(i\right)}-1 }}\left(f_{\mathbb{G}}+f_{\mathbb{D}_{0}}\right)\left(\rho^{\left(i\right)},\rho^{\left(j\right)}\right)&=\frac{1}{2},\\
    \lim_{\substack{\rho^{\left(i\right)}\in\left(1,\overline{\rho}\right],\rho^{\left(j\right)}-\rho^{\left(i\right)}+1\to0\\\rho^{\left(j\right)}\le\rho^{\left(i\right)}-1 }}\left(f_{\mathbb{G}}+f_{\mathbb{D}_{0}}\right)\left(\rho^{\left(i\right)},\rho^{\left(j\right)}\right)&=\lim_{\substack{\rho^{\left(i\right)}\in\left(1,\overline{\rho}\right],\rho^{\left(j\right)}-\rho^{\left(i\right)}+1\to0\\\rho^{\left(j\right)}\le \rho^{\left(i\right)}-1}}\frac{\left(\rho^{\left(j\right)}\right)^{2}+\rho^{\left(j\right)}}{2\rho^{\left(i\right)}\rho^{\left(j\right)}}=\frac{1}{2}.
\end{align*}
For (vi), we have that
\begin{align*}
    &\left.\left(f_{\mathbb{G}}+f_{\mathbb{D}_{0}}\right)\left(\rho^{\left(i\right)},\rho^{\left(j\right)}\right)\right|_{\rho^{\left(i\right)}\in\left(0,1\right),\rho^{\left(i\right)}=\rho^{\left(j\right)}}=1-\frac{\rho^{\left(i\right)}}{6},\\
    &\lim_{\substack{\rho^{\left(i\right)}\in\left(0,1\right),\rho^{\left(i\right)}-\rho^{\left(j\right)}\to0\\\rho^{\left(i\right)}\le\rho^{\left(j\right)}}}\left(f_{\mathbb{G}}+f_{\mathbb{D}_{0}}\right)\left(\rho^{\left(i\right)},\rho^{\left(j\right)}\right)=1-\frac{\rho^{\left(i\right)}}{6},\\
    &\lim_{\substack{\rho^{\left(i\right)}\in\left(0,1\right),\rho^{\left(i\right)}-\rho^{\left(j\right)}\to0\\\rho^{\left(i\right)}>\rho^{\left(j\right)}}}\left(f_{\mathbb{G}}+f_{\mathbb{D}_{0}}\right)\left(\rho^{\left(i\right)},\rho^{\left(j\right)}\right)\\
    &=\lim_{\substack{\rho^{\left(i\right)}\in\left(0,1\right),\rho^{\left(i\right)}-\rho^{\left(j\right)}\to0\\\rho^{\left(i\right)}>\rho^{\left(j\right)}}}\frac{-3\left(\rho^{\left(i\right)}\right)^{2}\rho^{\left(j\right)}+3\rho^{\left(i\right)}\left(\rho^{\left(j\right)}\right)^{2}+6\rho^{\left(i\right)}\rho^{\left(j\right)}-\left(\rho^{\left(j\right)}\right)^{3}}{6\rho^{\left(i\right)}\rho^{\left(j\right)}}\\
    &=1-\frac{\rho^{\left(i\right)}}{6}.
\end{align*}
For (vii), it holds that 
\begin{align*}
    &\left.\left(f_{\mathbb{G}}+f_{\mathbb{D}_{0}}\right)\left(\rho^{\left(i\right)},\rho^{\left(j\right)}\right)\right|_{\rho^{\left(i\right)}=1,\rho^{\left(j\right)}\in\left(0,1\right)}=\frac{1}{2}+\frac{\rho^{\left(j\right)}}{2}-\frac{\left(\rho^{\left(j\right)}\right)^{2}}{6},\\
    &\lim_{\substack{\rho^{\left(i\right)}\to1,\rho^{\left(j\right)}\in\left(0,1\right)\\\rho^{\left(i\right)}>\rho^{\left(j\right)}}}\left(f_{\mathbb{G}}+f_{\mathbb{D}_{0}}\right)\left(\rho^{\left(i\right)},\rho^{\left(j\right)}\right)=\frac{1}{2}+\frac{\rho^{\left(j\right)}}{2}-\frac{\left(\rho^{\left(j\right)}\right)^{2}}{6},\\
    &\lim_{\substack{\rho^{\left(i\right)}\to1,\rho^{\left(j\right)}\in\left(0,1\right)\\\rho^{\left(i\right)}\in\left(\rho^{\left(j\right)},\rho^{\left(j\right)}+1\right]}}\left(f_{\mathbb{G}}+f_{\mathbb{D}_{0}}\right)\left(\rho^{\left(i\right)},\rho^{\left(j\right)}\right)\\
    &\quad=\lim_{\substack{\rho^{\left(i\right)}\to1,\rho^{\left(j\right)}\in\left(0,1\right)\\\rho^{\left(i\right)}\in\left(\rho^{\left(j\right)},\rho^{\left(j\right)}+1\right]}}\frac{\left(\rho^{\left(i\right)}-\rho^{\left(j\right)}\right)^3 - 3\left(\rho^{\left(i\right)}\right)^2 + 6\rho^{\left(i\right)}\rho^{\left(j\right)} + 3 \rho^{\left(i\right)}- 1}{6\rho^{\left(i\right)}\rho^{\left(j\right)}}\\
    &\quad=\lim_{\substack{\rho^{\left(i\right)}\to1,\rho^{\left(j\right)}\in\left(0,1\right)\\\rho^{\left(i\right)}\in\left(\rho^{\left(j\right)},\rho^{\left(j\right)}+1\right]}}\frac{\left(\rho^{\left(i\right)}-\rho^{\left(j\right)}\right)^{3} - 1}{6\rho^{\left(i\right)}\rho^{\left(j\right)}}+1\\
    &\quad=\lim_{\substack{\rho^{\left(i\right)}\to1,\rho^{\left(j\right)}\in\left(0,1\right)\\\rho^{\left(i\right)}\in\left(\rho^{\left(j\right)},\rho^{\left(j\right)}+1\right]}}\frac{\left(\rho^{\left(i\right)}\right)^{3}-3\left(\rho^{\left(i\right)}\right)^{2}\rho^{\left(j\right)}+3\rho^{\left(i\right)}\left(\rho^{\left(j\right)}\right)^{2}-\left(\rho^{\left(j\right)}\right)^{3}-1}{6\rho^{\left(i\right)}\rho^{\left(j\right)}}+1\\
    &=\frac{1}{2}+\frac{\rho^{\left(j\right)}}{2}-\frac{\left(\rho^{\left(j\right)}\right)^{2}}{6}.
\end{align*}
For (viii), we have that
\begin{align*}
    &\left.\left(f_{\mathbb{G}}+f_{\mathbb{D}_{0}}\right)\left(\rho^{\left(i\right)},\rho^{\left(j\right)}\right)\right|_{\rho^{\left(i\right)}=\rho^{\left(j\right)}=1}=\frac{5}{6},\\
    &\lim_{\substack{\rho^{\left(i\right)}\to1,\rho^{\left(j\right)}\to1\\\rho^{\left(i\right)}\in\left(0,1\right),\rho^{\left(i\right)}\le\rho^{\left(j\right)}}}\left(f_{\mathbb{G}}+f_{\mathbb{D}_{0}}\right)\left(\rho^{\left(i\right)},\rho^{\left(j\right)}\right)=\frac{5}{6},\\
    &\lim_{\substack{\rho^{\left(i\right)}\to1,\rho^{\left(j\right)}\to1\\\rho^{\left(i\right)}\in\left(0,1\right),\rho^{\left(i\right)}>\rho^{\left(j\right)}}}\left(f_{\mathbb{G}}+f_{\mathbb{D}_{0}}\right)\left(\rho^{\left(i\right)},\rho^{\left(j\right)}\right)=\frac{5}{6},\\
    &\lim_{\substack{\rho^{\left(i\right)}\to1,\rho^{\left(j\right)}\to1\\\rho^{\left(i\right)}\in\left(1,\overline{\rho}\right],\rho^{\left(i\right)}\in\left(\rho^{\left(j\right)},\rho^{\left(j\right)}+1\right]}}\left(f_{\mathbb{G}}+f_{\mathbb{D}_{0}}\right)\left(\rho^{\left(i\right)},\rho^{\left(j\right)}\right)=\frac{5}{6},\\
    &\lim_{\substack{\rho^{\left(i\right)}\to1,\rho^{\left(j\right)}\to1\\\rho^{\left(i\right)}\in\left(1,\overline{\rho}\right],\rho^{\left(i\right)}\le\rho^{\left(j\right)}}}\left(f_{\mathbb{G}}+f_{\mathbb{D}_{0}}\right)\left(\rho^{\left(i\right)},\rho^{\left(j\right)}\right)=\frac{5}{6}
\end{align*}
For (ix), we obtain that
\begin{align*}
    \left.\left(f_{\mathbb{G}}+f_{\mathbb{D}_{0}}\right)\left(\rho^{\left(i\right)},\rho^{\left(j\right)}\right)\right|_{\rho^{\left(i\right)}=1,\rho^{\left(j\right)}\in\left(1,\overline{\rho}\right]}&=1-\frac{1}{6\rho^{\left(j\right)}},\\
    \lim_{\substack{\rho^{\left(i\right)}\to1,\rho^{\left(j\right)}\in\left(1,\overline{\rho}\right]\\\rho^{\left(i\right)}\in\left(0,1\right)}}\left(f_{\mathbb{G}}+f_{\mathbb{D}_{0}}\right)\left(\rho^{\left(i\right)},\rho^{\left(j\right)}\right)&=1-\frac{1}{6\rho^{\left(j\right)}},\\
    \lim_{\substack{\rho^{\left(i\right)}\to1,\rho^{\left(j\right)}\in\left(1,\overline{\rho}\right]\\\rho^{\left(i\right)}\in\left(1,\overline{\rho}\right]}}\left(f_{\mathbb{G}}+f_{\mathbb{D}_{0}}\right)\left(\rho^{\left(i\right)},\rho^{\left(j\right)}\right)&=\lim_{\substack{\rho^{\left(i\right)}\to1,\rho^{\left(j\right)}\in\left(1,\overline{\rho}\right]\\\rho^{\left(i\right)}\in\left(1,\overline{\rho}\right]}}\frac{-3\left(\rho^{\left(i\right)}\right)^{2}+6\rho^{\left(i\right)}\rho^{\left(j\right)}+3\rho^{\left(i\right)}-1}{6\rho^{\left(i\right)}\rho^{\left(j\right)}}\\
    &=1-\frac{1}{6\rho^{\left(j\right)}}.
\end{align*}
For (x), it holds that
\begin{align*}
    \left.\left(f_{\mathbb{G}}+f_{\mathbb{D}_{0}}\right)\left(\rho^{\left(i\right)},\rho^{\left(j\right)}\right)\right|_{\rho^{\left(i\right)}=\rho^{\left(j\right)}\in\left(1,\overline{\rho}\right]}&=\frac{3\left(\rho^{\left(i\right)}\right)^{2}+3\rho^{\left(i\right)}-1}{6\left(\rho^{\left(i\right)}\right)^{2}},\\
    \lim_{\substack{\rho^{\left(i\right)}-\rho^{\left(j\right)}\to0\\\rho^{\left(i\right)}\le\rho^{\left(j\right)}}}\left(f_{\mathbb{G}}+f_{\mathbb{D}_{0}}\right)\left(\rho^{\left(i\right)},\rho^{\left(j\right)}\right)&=\frac{3\left(\rho^{\left(i\right)}\right)^{2}+3\rho^{\left(i\right)}-1}{6\left(\rho^{\left(i\right)}\right)^{2}},\\
    \lim_{\substack{\rho^{\left(i\right)}-\rho^{\left(j\right)}\to0\\\rho^{\left(i\right)}>\rho^{\left(j\right)}}}\left(f_{\mathbb{G}}+f_{\mathbb{D}_{0}}\right)\left(\rho^{\left(i\right)},\rho^{\left(j\right)}\right)&=\frac{3\left(\rho^{\left(i\right)}\right)^{2}+3\rho^{\left(i\right)}-1}{6\left(\rho^{\left(i\right)}\right)^{2}}.
\end{align*}
For (xi), we obtain that
\begin{align*}
    \left.\left(f_{\mathbb{G}}+f_{\mathbb{D}_{0}}\right)\left(\rho^{\left(i\right)},\rho^{\left(j\right)}\right)\right|_{\rho^{\left(i\right)}=0,\rho^{\left(j\right)}\in\left(0,\overline{\rho}\right]}&=1,\\
    \lim_{\rho^{\left(i\right)}\downarrow0,\rho^{\left(j\right)}\in\left(0,\overline{\rho}\right]}\left(f_{\mathbb{G}}+f_{\mathbb{D}_{0}}\right)\left(\rho^{\left(i\right)},\rho^{\left(j\right)}\right)=1.
\end{align*}
Hence we have the continuity of $f_{\mathbb{G}}$.
\end{proof}

\end{document}